%% file: KAM_part1.tex
\documentclass[a4paper,english,12pt]{amsbook}
\pdfoutput=1
\usepackage{pstricks}
\usepackage[protrusion=true,expansion=true]{microtype}
\usepackage{graphicx}
\usepackage[all]{xy}
\usepackage{hyperref}
\usepackage[utf8]{inputenc}
\usepackage{eucal}
\usepackage{mathrsfs}
\usepackage{amssymb}
\usepackage{amsxtra}
\usepackage{enumerate}
\hypersetup{colorlinks=true,citecolor=black,filecolor=black,linkcolor=blue,urlcolor=black} 
 
\newtheorem{theorem}{Theorem}[chapter]
\newtheorem{proposition}[theorem]{ Proposition} 
\newtheorem{lemma}[theorem]{ Lemma}
\newtheorem{corollary}[theorem]{Corollary}

\newtheorem{definition}[theorem]{Definition}

\theoremstyle{remark}

\newtheorem{example}[theorem]{\it Example}
 
 \def \1{\mathbb {1}}
\def \R{\mathbb {R}}
\def \RM{\mathbb {R}}
\def \NM{\mathbb{N}}
\def \ZM{\mathbb{Z}}
\def \CM{\mathbb{C}}

\def \QM{\mathbb{Q}}

\def \Ker {{\rm Ker\,}}

\def \Spec {{\rm Spec\,}}
\def \Der {{\rm Der\,}}

\def \Ham {{\rm Ham\,}}
\def \Symp {{\rm Symp\,}}

\def \p {{\rm exp\,}}
\def \Id {{\rm Id\,}}

\def \d{\partial}
\def\dt{\delta} 
\def\a{\alpha}
\def\b{\beta}
\def\e{\varepsilon}  
\def\g{\gamma}

\def\l{\lambda}
\def\L{\Lambda}
\def\p{\varphi}

\def \s{\sigma}
\def \t{\tilde}

\def \to{\longrightarrow} 
\def \w{\wedge}
\def \alg{\mathfrak{g}}

\def\del{\nabla}
\def \< {{\langle }}
\def \> {{\rangle }}
\def \( {\left( }
\def \) {\right) }

\newcommand{\Ct}{{\mathcal C}}

\newcommand{\Jt}{{\mathcal J}}

\newcommand{\Lt}{{\mathcal L}}

\newcommand{\Pt}{{\mathcal P}}

\newcommand{\lra}{\longrightarrow}
\renewcommand{\mod}{{\rm  mod\,}}
\title[Th\'eorie KAM]{{\sc  Th\'eorie KAM}}
\author{ Mauricio  Garay }

\makeindex
 
\begin{document}
 \thispagestyle{empty}
 \begin{center}{\LARGE  \bf KAM THEORY\\} 
  \vskip0.8cm { \large \sc M. Garay and D. van Straten}\\
  \vskip0.5cm
  \rule{0.3\textwidth}{0.5pt}\\
    \vskip1cm
    
 { \large \sc  \textbf{I \\ \vskip0.5cm GROUP ACTIONS AND THE KAM PROBLEM}}
 
 \vskip1cm
  \begin{figure}[ht]
  \begin{minipage}[b]{0.6\linewidth}   
  \includegraphics[height=0.6\linewidth,width=0.75\linewidth]{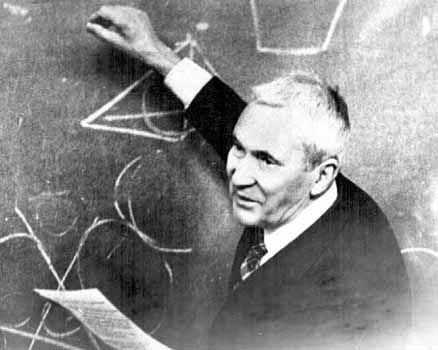}
 
  \end{minipage}
   \begin{minipage}[b]{0.60\linewidth}
     \includegraphics[height=0.42\linewidth,width=0.75\linewidth]{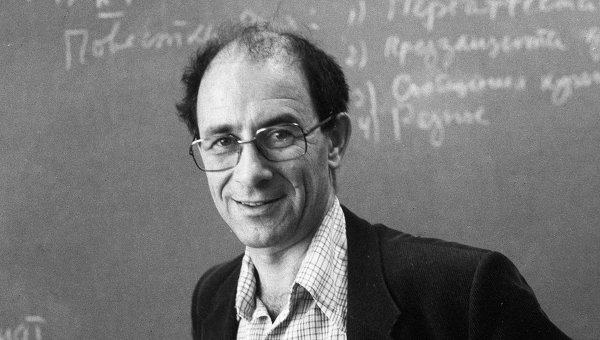}

             \end{minipage}\hfill
         \begin{minipage}[b]{0.40\linewidth}
           \includegraphics[height=0.9\linewidth,width=0.6\linewidth]{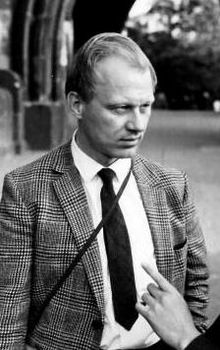}
     \end{minipage}
  \end{figure}

  \end{center}

\newpage
\ \\
\thispagestyle{empty}

\clearpage
\thispagestyle{empty}
\ \\
\vskip5cm
 {\em This book originated from courses taught by the first author 
 in 2011 at the Dijon University (France) and in 2013 at the University of Ouargla (Algeria) on the Herman invariant tori conjecture. 
 
 The first author thanks M. Bahayou and A. Zeglaoui for the invitation in Ouargla and R. Uribe and A. Dubouloz
 for the invitation in Dijon. Thanks also to the audience and in particular to   N. A'Campo, P. Cartier, Z. Fernane, A. Kessi, F. Laudenbach,
  D. Smai and N. Yousfi.}
 \newpage
 \cleardoublepage
\thispagestyle{empty}
\ \\
\vskip5cm
\begin{flushright}
{\em  To the memory of V.I. Arnold (1937-2010).}
\end{flushright}
\include{introduction}
\tableofcontents
\include{part1_v3}

\chapter*{Photo Credits}
{\em J. Moser} by Konrad Jacobs, Erlangen (1969).\\

\end{document}

%% file: part1_v3.tex

\chapter{ Darboux theorems and action-angle variables}
We review classical mechanics in its symplectic formulation and sketch proofs of the basic results 
like the existence of action-angle coordinates. Some familiarity with the language of global analysis 
is assumed. 

\section{Hamiltonian vector fields}
A real $C^{\infty}$ manifold $M$ is called {\em symplectic\index{symplectic manifold}} if it is endowed with a closed $2$-form $\omega$ which induces
an  isomorphism between the tangent and cotangent bundles:
   $$TM \to T^*M,\ X \mapsto i_X \omega. $$
Here $i_X$ denotes the interior product of a differential form with a vector field $X$, that is:
   $$i_X\omega(Y)=\omega(X,Y). $$
This isomorphism associates to each differential $1$-form a unique vector field, which we call 
{\em symplectically associated} to it. Given a function
\[H: M \to \RM, \]
the vector field $X_H$ associated to the $1$-form $dH$ is called the {\em Hamiltonian vector field of $H$}.
  
 The space $\RM^{2n}$, equipped with the $2$-form 
 $$\omega:=\sum_{i=1}^n dq_i \w dp_i,$$
gives a basic example of a symplectic manifold.
The isomorphism between tangent and cotangent bundles is  given by
    $$\d_{q_i} \mapsto dp_i,\ \d_{p_i} \mapsto -dq_i. $$
In this way we recover the classical definition of the Hamiltonian vector field
  $$X_H:=\sum_{i=1}^n(\d_{p_i} H \d_{q_i}-\d_{q_i} H \d_{p_i}),$$
as it is symplectically associated to $dH$. Its integral curves are the solutions to  {\em Hamilton's canonical 
equations of motion}:
  $$\left\{ \begin{matrix}\dot q_i&=& \d_{p_i} H \\
  \dot p_i&=& -\d_{q_i} H\;\;. \end{matrix} \right. $$

A smooth map $\p:M \to M'$ between two symplectic manifolds $(M,\omega)$ and
$(M',\omega')$ is called {\em symplectic}\index{symplectic map}, if it preserves the symplectic forms: 
\[\p^*(\omega')=\omega .\]
By a {\em symplectomorphism}\index{symplectomorphism} we will mean a symplectic
diffeomorphism. A symplectomorphism 
  $$\p:M \to M $$
maps the Hamiltonian vector field of $H \circ \p$ to that of $H$.  Therefore the qualitative behaviour of a dynamical system only depends on the orbit of Hamiltonian functions under the group of symplectomorphism. 
 
If one can integrate the Hamiltonian vector field $X_H$ up to time $t$, we obtain a {\em flow}\index{flow of a vector field} 
\[\p_t:=\p^t_{X_H}:M \to M,\;\;\;\p_0=Id_M,\;\;\;\p_t \circ \p_s=\p_{t+s},\]
which is a one parameter family of symplectomorphism depending on $t$. To show this, recall the formula
$$\frac{d}{dt}(\p_t^* \omega)=\p_t^*(\Lt_{X_H}\omega) $$
where $\Lt_\cdot$ stands for the Lie derivative. Applying Cartan's formula\index{Cartan's formula} 
\[ \Lt_X=di_X+i_X d\] 
we find:
$$ \Lt_{X_H}\omega=di_{X_H}\omega+i_{X_H}d\omega=ddH+0=0.$$
Hence $\p_t^* \omega$ is constant in $t$, and as $\p_0=Id_M$ we find $\p_t^* \omega =\omega$.

So any function $H$ gives rise to a differential $dH$ and therefore to a vector field which preserves the symplectic form, but the converse is not always true. Vector fields preserving the symplectic form 
will be called {\em symplectic vector fields}\index{symplectic vector field}. 
These are usually called {\em Hamiltonian vector fields},
 but we wish to distinguish them from the vector fields defined by Hamiltonian functions 
(which are usually called {\em exact Hamiltonian vector fields}). So we use the name Hamiltonian vector fields\index{Hamiltonian vector field} for the vector fields defined by a Hamiltonian function.

\begin{proposition} In a symplectic manifold the closed one-forms are symplectically associated to symplectic vector fields.
\end{proposition}
\begin{proof}
 A vector field $X$ is symplectic if and only if
 $$\Lt_X\omega=0. $$
 By Cartan's formula
 $$\Lt_X\omega=di_X\omega+i_Xd\omega=di_X\omega. $$
 Therefore $i_X \omega$ is a closed 1-form if and only if $X$ is symplectic. 
\end{proof}

We denote by $\Omega^1(M)$ the vector space of $C^{\infty}$ real valued $1$-forms on $M$ and let \index{$1$-forms, closed and exact}
\[ \Omega^1_{closed}(M):=\{ \eta \;|\;d\eta=0\} \supset  \Omega^1_{exact}(M):=\{dH\;|\;H \in C^{\infty}(M)\}\]
be the sub-spaces of closed and exact $1$-forms. 
The exterior derivative is a map
$$ C^{\infty}(M) \to \Omega_{closed}^1(M)$$
which has the space of locally constant functions $H^0(M)$ as kernel, and the
de Rham cohomology space \index{De Rham cohomology}  $H^1(M)$ as cokernel, so
we obtain an exact sequence of vector spaces
\[ 0 \to \Omega^1_{exact}(M) \to \Omega^1_{closed}(M) \to H^1(M) \to 0 .\]
Looking at the symplectically associated spaces of vector fields, the above
exact sequence is converted into the exact sequence
\[0 \to \Ham(M) \to \Symp(M) \to H^1(M) \to 0,\]
where $\Ham(M)$ denotes the vector space of Hamiltonian vector fields and 
$\Symp(M)$ the space of vector fields which preserve the symplectic form.
So we see that the difference between these two types of vector fields has a 
simple cohomological interpretation.

\begin{example}\label{E::torus}
Consider the $n$-dimensional torus $(\RM/\ZM)^n$; its points are described by
$n$ angular coordinates $\theta_1,\dots,\theta_n$. These $\theta_i$ are multivalued functions, but their differentials $\a_i:=d\theta_i$ are well-defined closed
$1$-forms on the torus.
We consider on the space $M:=(\RM/\ZM)^n \times \RM^n$ the symplectic form
$$\omega=\sum_{i=1}^n d\theta_i \w dp_i. $$
where the $p_i$ are coordinates on the $\RM^n$-factor.
The cohomology classes $[\a_i], i=1,2,\ldots,n$
define a basis of the de Rham cohomology space
$$H^1(M) =\oplus_{i=1}^n \RM[\a_i] \approx \RM^n.$$
These forms are associated to the symplectic vector fields:
$$X_1=\d_{p_1},\dots,X_n=\d_{p_n}. $$
Therefore any symplectic vector field on $M$ is of the form
$$\sum_{i=1}^n a_i\d_{p_i}+X_H $$
where 
$$H: M \to \RM $$
is a $C^\infty$ function and $a_1,\dots,a_n \in \RM$.  
\end{example}

\section{Cotangent spaces}
The cotangent space $M=T^*L$ of a manifold $L$ naturally inherits a symplectic form.
 The points of this manifold are pairs $(q,p)$, where $p \in T^*_qL$ is a covector 
at a point $q \in L$. The derivative of the bundle projection
$$\pi: M \to L $$
is a map
$$d \pi:TM \to TL .$$
The {\em action one-form}\index{action form on cotangent bundles} $\a \in \Omega^1(M)$ is defined by the formula
$$\a_{q,p}(\xi)=p(d \pi_{q,p} (\xi)),\ \ \xi \in T_{q,p}M. $$
Then the exact two form:
\[ \omega:=d\a\]
is non-degenerate and is called the {\em canonical symplectic form}\index{canonical symplectic form on cotangent bundles} on $T^*L$.\\ 

 \vskip0.3cm  \begin{figure}[h!]
 \includegraphics[width=12cm]{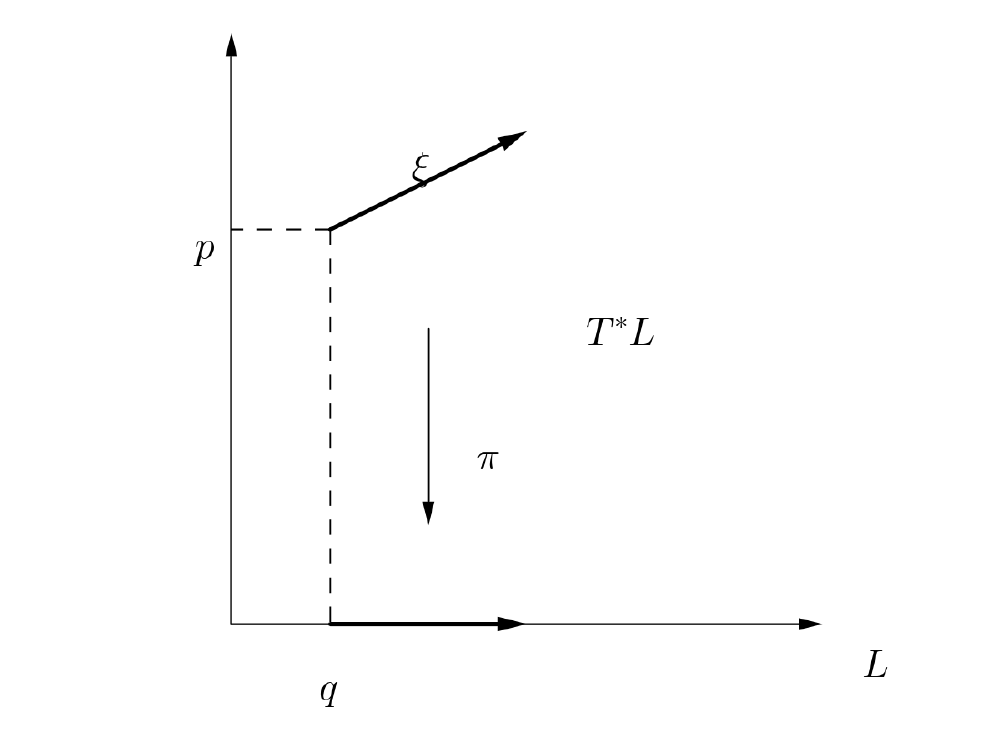}
\end{figure} \vskip0.3cm  

Let us analyse this construction for $L=\RM$. In this case, the cotangent bundle can be 
identified with the projection
$$\RM^{2} \to \RM,(q,p)\mapsto q.$$
If $\xi=(\xi_1,\xi_2) \in T_{(q,p)}\RM^2$ is a tangent vector, then
$$d\pi_{q,p} (\xi)=\xi_1 $$
so the action form is defined by
$$\a_{q,p}(\xi)=p\,\xi_1. $$
This means that 
$$\a=pdq$$ 
and therefore $\omega=dp  \wedge dq$.
Similarly, for the space $\RM^{2n}$, viewed as $T^*\RM^n$, we recover the 
symplectic form defined previously, up to a sign and in Example~\ref{E::torus} 
we were in fact dealing with the cotangent space to the $n$-torus $(\RM/\ZM)^n$.

Note that the canonical symplectic form vanishes on the zero-section of $T^*L$. More generally,
an $n$-dimensional submanifold $L$ of a $2n$ dimensional symplectic manifold $(M,\omega)$ 
is called {\em Lagrangian\index{Lagrangian manifold}} if $\omega_{|L}=0$. In such a case
the class of $\omega$ vanishes inside the de Rham cohomology group $H^2(L)$. In any neighbourhood
of $L$ which retracts onto $L$, the symplectic form is exact. A form $\a$ such that 
$$d\a=\omega $$ is called an
{\em action form \index{action form}}.
 
\begin{proposition}
\label{P::generating}
 Consider the symplectic space $\RM^{2n}$ with coordinates $q,p$ and symplectic form
$$\omega=\sum_{i=1}^n dq_i \w dp_i. $$ Let $L \subset \RM^{2n}$ be a Lagrangian manifold given as the graph
of a map 
$$f=(f_1,\dots,f_n):\RM^n \to \RM^n$$ over the $q$-space. Then there exists a function 
$$S:\RM^n \to \RM$$ 
such that
$$f=\del S=(\d_{q_1}S,\d_{q_2}S,\dots,\d_{q_n} S). $$
Conversely, any such graph is Lagrangian.
\end{proposition}
\begin{proof}
The action form
$$\sum_{i=1}^n  p_i dq_i $$
is closed on $L$, thus the Poincar\'e lemma implies that it is the differential of a function $S$:
\[ dS=\sum_{i=1}^n p_i dq_i,\]
hence $p_i=\partial_iS$.  
\end{proof}
A function $S$ as in the proposition is called a {\em generating function}\index{generating function} of $L$. It is unique up to an additive constant.
\section{The Darboux theorem}
Recall that two mappings define the same {\em germ} along a subset, if their restriction  to a common neighbourhood of the subset agree. 
We denote by
  $$f:(M,X) \to (N,Y),\ f(X)=Y $$
the germ of $f:M \to N$ along $X$.
The following fundamental result is called the {\em Darboux theorem\index{Darboux theorem}}.
  \begin{theorem} The germ of a $2n$-dimensional symplectic manifold
$(M,\omega)$ at an arbitrary point is isomorphic to the germ of $(\RM^{2n},\sum_i dq_i \w dp_i)$ at the origin.
  \end{theorem}
  \begin{proof}
Take a local chart
  $$\p:(M,p) \to (\RM^{2n},0)$$
at a point $p \in M$.
 We get two symplectic forms on a neighbourhood of the origin in $\RM^{2n}$:
 $$\omega_0=(\p^{-1})^* \omega,\;\;\textup{and}\;\;\; \omega_1=\sum_{i=1}^n dq_i \w dp_i.$$
The value of these forms at the origin are anti-symmetric bilinear forms which are conjugate by a linear transformation.
Thus, up to a linear change of coordinates, we may and will assume that they are equal. It follows that
 $$\omega_t =(1-t)\omega_0+t\omega_1,\;\;t \in [0,1] $$
 defines a $1$-parameter family of symplectic forms in a sufficiently small contractible neighbourhood of the origin. 
 
We now search for a $1$-parameter family $\p_t$ of symplectomorphisms such that:
 $$\p_t^* \omega_t=\omega_0 .$$
We differentiate this equation with respect to $t$ and as the right hand side is $t$-independent we obtain:
\[0=\frac{d}{dt}\p_t^* \omega_t= \p_t ^* (\Lt_{X_t} \omega_t+ \dot \omega_t), \;\;\dot \omega_t:= \frac{d}{dt}\omega_t ,\]  
where $X_t$ is the time-dependent Hamiltonian vector field associated to $v_t$.
Composing with the inverse of $v_t^*$, we get the equation:
$$\Lt_{X_t} \omega_t=-\dot \omega_t$$
Cartan's formula shows that:
$$\Lt_{X_t} \omega_t=d i_{X_t} \omega_t+i_{X_t} d \omega_t= d i_{X_t} \omega_t.$$
As the neighbourhood was assumed to be contractible, the closed form $\omega_t$ is  in fact exact by the Poincar\'e lemma, so we have 
$$\omega_t=d \a_t .$$
It is therefore sufficient to solve the equation
$$i_{X_t} \omega_t=-\dot \a_t$$
for the vector field $X_t$, which is possible, since the interior product with a symplectic form is an isomorphism. 
As the forms $\omega_1$ and $\omega_2$ are equal at the origin, the vector field $X_t$ vanishes at $0$. Thus, the time 
$1$ flow of the vector field $X_t$ exists in a small neighbourhood of the origin and yields the sought for symplectomorphism. This proves the theorem.
\end{proof}

Coordinates such as in the above theorem are called {\em Darboux coordinates\index{Darboux coordinates}}, sometimes {\em canonical coordinates}\index{canonical coordinates}. The existence of Darboux coordinates shows that there are no local symplectic invariants. 

 \begin{corollary} Let $L \subset M$ be a Lagrangian submanifold of a symplectic manifold $(M,\omega)$.
 The germ of $L$ at a point is symplectomorphic to the germ at the origin of the zero section in $T^*\RM^n$ with its canonical
 symplectic structure.
 \end{corollary}
 \begin{proof}
Choosing Darboux coordinates at the point considered, we can reduce to the case
$M=T^*\RM^{n}$ equipped with its canonical symplectic form. The $\pi/2$-rotations in the planes $(q_i,p_i)$ 
$$(q_i,p_i) \mapsto (-p_i,q_i) $$
are symplectomorphisms and by applying these, we may always assume that $L$ is the graph of a map. By Proposition \ref{P::generating},  this map is itself
the gradient of a function $$S:\RM^n \to \RM.$$
The map
$$ (q,p) \mapsto (q,p-\del S)$$
is a symplectomorphism which maps $L$ to the zero section.
 \end{proof}

\section{The classical Darboux-Weinstein theorem}
The theorem of Darboux implies the fundamental fact that a Lagrangian manifold has no local symplectic invariants.
In some situations there are global versions of this result.

Our proof of the Darboux theorem consisted of two parts: starting with two different symplectic
forms in a neighbourhood of the origin in $\RM^{2n}$, we first chose linear coordinates
so that both symplectic forms agree at the origin. Then knowing that the linear path between the symplectic forms remains
inside the space of symplectic forms, we applied the path homotopy method.

 \begin{theorem}
 \label{T::Darboux_Global} Let $L$ be compact submanifold of a symplectic manifold $M$ and $\omega_t, t \in [0,1]$ a one parameter family of
 symplectic forms on $M$ with $C^\infty$ dependence on $t$. Assume that the restrictions of de Rham classes $[\omega_t]$ to  $ H^2(L,\RM)$ are independent of $t$. Then there exists a neighbourhood $T \subset M $ of $L$ such that the $t$-dependent symplectic manifolds $(T,\omega_t)$ are
 all symplectomorphic.
 \end{theorem}
 \begin{proof}
 By compactness of the interval $[0,1]$ it suffices to prove the theorem for 
sufficiently small values of $t$. Choose a tubular neighbourhood 
$T$ of $L$. As $T$ retracts to $L$, the assumptions imply that the class of $\d_t \omega_t$ vanishes in $H^2(T)$. This means that $\d_t \omega_t$
is exact in $T$, that is, we find a family of $1$-forms $\a_t$ 
such that 
$$ \d_t \omega_t=d\a_t .$$ 
As before one finds a time dependent vector field $X_t$ such that
$$i_{X_t} \omega_t=-\a_t$$
By compactness of $L$, the vector field can be integrated for sufficiently small times and
its time $t$-flow sends  $\omega_0$ to $\omega_t$. This concludes the proof of the theorem. 
\end{proof}

From this theorem we will deduce the following celebrated result:
\begin{theorem}
\label{T::Weinstein}
Any compact Lagrangian submanifold  $L$ of a symplectic manifold $(M,\omega)$ admits a neighbourhood symplectomorphic to neighbourhood of $L$ in $T^*L$ with its standard symplectic structure.
\end{theorem}

In the proof we will make use of an important idea of Gromov, namely the
existence of an {\em almost complex structure adapted to a given symplectic form}\index{adapted almost complex structure}.
Recall that that a {\em complex structure} on a real vector space $E$ is a
linear map $J \in GL(E)$ with the property that $J^2=-\Id$. 
It is said to be {\em adapted}\index{adapted complex structure } to a 
linear symplectic form $\omega$ on $E$ if 
\[ \omega(-,J-) \]
is an Euclidean scalar product on $E$. 

\begin{lemma}
 Let $(E,\omega)$ be a symplectic vector space, $L \subset E$ a linear Lagrangian subspace.
 A complex structure $J \in GL( E) $ is adapted if and only if the following three conditions are satisfied: 
\begin{enumerate}[{\rm (i)}]
\item The quadratic form
 $L \to \RM, x \mapsto \omega(x,Jx) $
 is positive definite,
 \item $L$ and $JL$ are complementary,
 \item $\omega(x,y)=\omega(Jx,Jy)$ for any $x,y \in E$.
\end{enumerate}
\end{lemma}
\begin{proof}
We first check that the conditions are necessary. The restriction of a
positive definite quadratic form to a vector sub-space remains positive 
definite, so clearly we have $(i)$. If $x \in L \cap JL$, then $Jx$ lies also 
in $L \cap JL$. As $L$ is Lagrangian, we deduce that
$$\omega(x,Jx)=0 $$
and this implies that $x=0$ since $\omega(-,J-)$ is positive definite, so
we have $(ii)$. To verify check $(iii)$, let $z=Jy$, using that $\omega(-,J-)$ and $\omega  $ are respectively symmetric and antisymmetric,
we get that:
$$\omega(x,y)=\omega(x,-Jz)=\omega(z,-Jx)=\omega(Jy,-Jx)=\omega(Jx,Jy) .$$

Let us now prove that the conditions are sufficient. By $(iii)$ the form
$\omega(-,J-)$ is symmetric and by $(ii)$  any $x \in E$ can be written as
$$x=x_1+Jx_2,\ x_1, x_2 \in L .$$
Then 
$$Jx=Jx_1- x_2, $$
and so
$$\omega(x,Jx)=\omega(x_1,Jx_1)+\omega(x_1,-x_2)+\omega(Jx_2,Jx_1)+\omega(Jx_2,-x_2).$$
The two terms in the middle are zero since $L$ is Lagrangian, whereas by $(i)$
the first and last terms are $\ge 0$ and vanish only if $x_1$ and $x_2$ are both
zero. 
\end{proof}

\begin{example}
Take $E=\RM^2$ with the standard symplectic structure
$$\omega:((x,y),(x',y')) \mapsto xy'-x'y .$$ 
A matrix defines an almost complex structure if its eigenvalues are $\pm i$. Such a matrix is of the form
$$J=\begin{pmatrix} a & b \\
 c & -a \end{pmatrix} $$
 with $a^2+bc=-1$. Therefore the set of complex structures is a two-sheeted hyperboloid. One sheet corresponds to the
 complex structure for which $\omega(-,J-)$ is positive definite and the other one for which it is negative definite.\\

 \vskip0.3cm  \begin{figure}[h!]
 \includegraphics[scale=0.6]{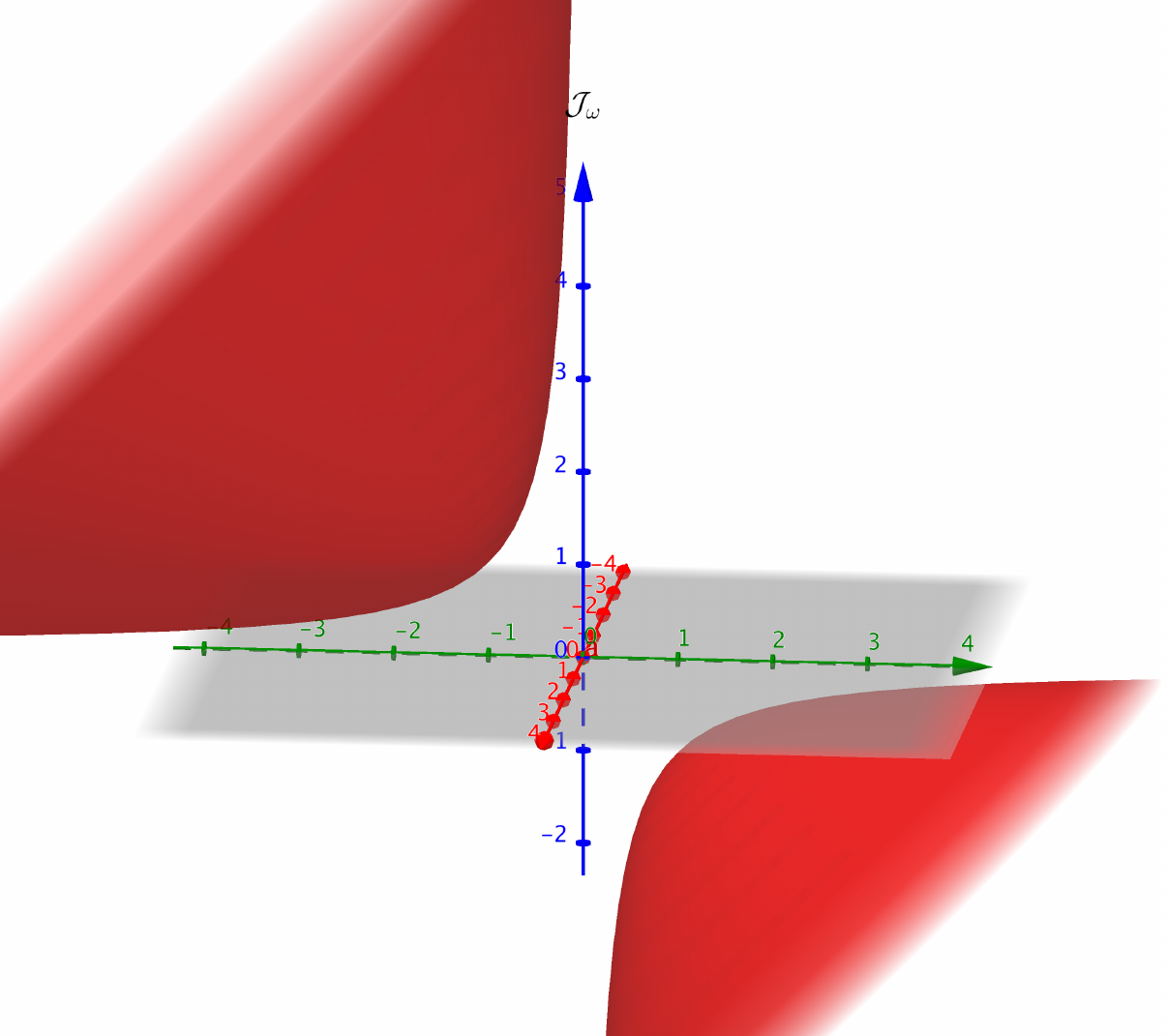}
\end{figure} \vskip0.3cm  
\end{example}

\begin{theorem} Let $(E,\omega)$ be a linear symplectic space. The subset $\Jt_\omega $ of $GL(E)$ consisting of
complex structures adapted to $\omega$ is a non-empty contractible manifold.
\end{theorem}

\begin{proof}
We fix two transverse linear Lagrangian subspace $L_1,L_2 \subset E$. As $J$ is adapted to 
$\omega$, the quadratic form
\[ q(\eta):=\omega(\eta,J\eta),\;\;\;\eta \in L_1\]
is positive definite and $L_1$ and $JL_1$ 
provide a direct sum decomposition of $E$:
\[ E = L_1 \oplus JL_1 .\]

If conversely $L'$ is any linear Lagrangian subspace transverse to $L_1$ and 
$$q:L_1 \to \RM$$
is a positive definite quadratic form, we can recover an adapted 
complex structure $J$ on $E$ as follows: the $q$-orthogonal space to $x \in L_1$ is a hyperplane $H \subset L_1$.
Therefore there is a unique $y \in L'$ which vanishes on $H$ such that
$$ \omega(x,y)=q(x)$$
Now define a complex structure $J$ by mapping $x$ to $y$ and $y$ to $-x$. 

As the Lagrangian subspace $L_2$ is transverse to $L_1$, all other transverse Lagrangian 
subspaces $L'$ are obtained as graphs of the derivative of a quadratic function 
$$S:L_2 \to \RM.$$
Thus the space of all adapted complex structures on our symplectic vector space is in one-to-one 
correspondence  with the contractible set of pairs of quadratic forms $(S,q)$ such that $q$ is 
positive definite.
\end{proof}

Given a manifold $M$, we have a principal bundle
$$GL(TM) \to M, $$  whose fibre above $x \in M$ is the linear group $GL(T_xM)$. 
A section of this bundle
 $$J:M \to GL(TM) $$
is called an {\em almost complex structure} if for any $x \in M$, we have
 $$J(x)^2=-\Id.$$
If now $M$ is equipped with a symplectic form $\omega$, the almost complex structure $J$ is called {\em adapted} 
(or also {\em tamed}) if $\omega(-,J-)$ is a Riemannian metric. 
From the above lemma one deduces the existence of adapted almost complex structures on symplectic manifolds. 
Indeed we have a bundle
$$\Jt_\omega(M) \to M $$
whose fibre above $x \in M$ is the set of $\omega_x$-adapted structures on $T_xM$. This space is contractible and therefore the bundle admits sections.

 \begin{lemma}
 \label{L::cone}
Let $\omega_1$ and $\omega_2$ be two symplectic forms on a manifold $M$. 
Assume that there exists a complex structure $J$ adapted both to $\omega_1$ and $\omega_2$. 
Then all two-forms lying on the positive cone
 $$\a_1\omega_1+\a_2\omega_2,\ \a_i \ge 0 $$
 are symplectic.
 \end{lemma}
 \begin{proof}
 Put
 $$\omega=\a_1\omega_1+\a_2\omega_2,\ \a_i \ge 0 .$$
 Assume that $\omega(\xi,-)$ vanishes then
 $$\omega(\xi,J\xi)=\a_1\omega_1(\xi,J\xi)+\a_2\omega_2(\xi,J\xi)=0$$
 But the $\omega_i(-,J-)$'s are euclidean scalar product thus 
 $$\omega_1(\xi,J\xi)\geq 0,\ \omega_2(\xi,J\xi) \geq 0$$
 therefore the sum cannot be zero unless both terms vanish. Consequently $\xi=0$. This proves the lemma.
 \end{proof}
 
We can now prove the Darboux-Weinstein theorem.
 
\begin{proof}[Proof of theorem]
The abstract normal bundle $NL$ to $L$ is the quotient of the restriction of the tangent bundle of $M$ to $L$, $TM_{|L}$, by the tangent bundle $TL$ to $L$.
By the tubular neighbourhood theorem, there is a diffeomorphism $\phi$
from a neighbourhood of $L$ in $M$ to a neighbourhood of $L$ in the normal bundle $NL$. The interior product of vectors $v$ based at a point $p$ of $L$ with the symplectic form induces an isomorphism of $NL$ with the cotangent bundle $T^*L$ of $L$:
 $$ \nu: N_pL \to T^*_pL,\;\; v \mapsto i_v \omega .$$ 
Therefore we reduced the theorem to the case $(M,\omega)=(T^*L,\omega)$ 
where $\omega$ is a priori not the standard symplectic form $\omega_{std}$ on $T^*L$.

We now choose almost complex structures $J$ and $J_{std}$ on $M=T^*L$ adapted to 
the symplectic forms $\omega$ and $\omega_{std}$. 

First observe that the tangent space $T_{q,0}M$ is a direct sum
$$T_{q,0}M=T_{q}L \oplus JT_q L .$$
The inner product with the symplectic form gives an identification of $JT_q L$ with $T_q^*L$.

Using these data, we will now define a map
$$f: M \to M $$
in the following way. The inclusion of bundles over $L$ 
$$i: T^*L \to TM_{|L} ,$$
is given pointwise by the inclusion: 
$$T^*_qL \hookrightarrow T_{q,0}M .$$
Its composition with the linear bundle map  
$$JJ_{std}^{-1}: TM_{|L} \to TM_{|L}$$ defines a map
$$ JJ_{std}^{-1} \circ i: M \to TM_{|L}. $$ 
The map  $f$ is obtained by composing this map with the exponential map
\[ exp: TM_{|L} \to M,\] defined pointwise by the exponential maps
$$exp_q:T_{(q,0)}M \to M$$ for the Riemannian metric 
$\omega(-,J_{std}-)$. So
$$f:=exp \circ JJ_{std}^{-1} \circ i:M=T^*L \stackrel{i}{\hookrightarrow} TM_{|L} \stackrel{JJ_{std}^{-1}}{\to} TM_{|L} \stackrel{exp}{\to} M. $$
Clearly $f$ restricted to $L$ is the identity.\\

\vskip0.3cm  \begin{figure}[htb!]
\includegraphics[width=12cm]{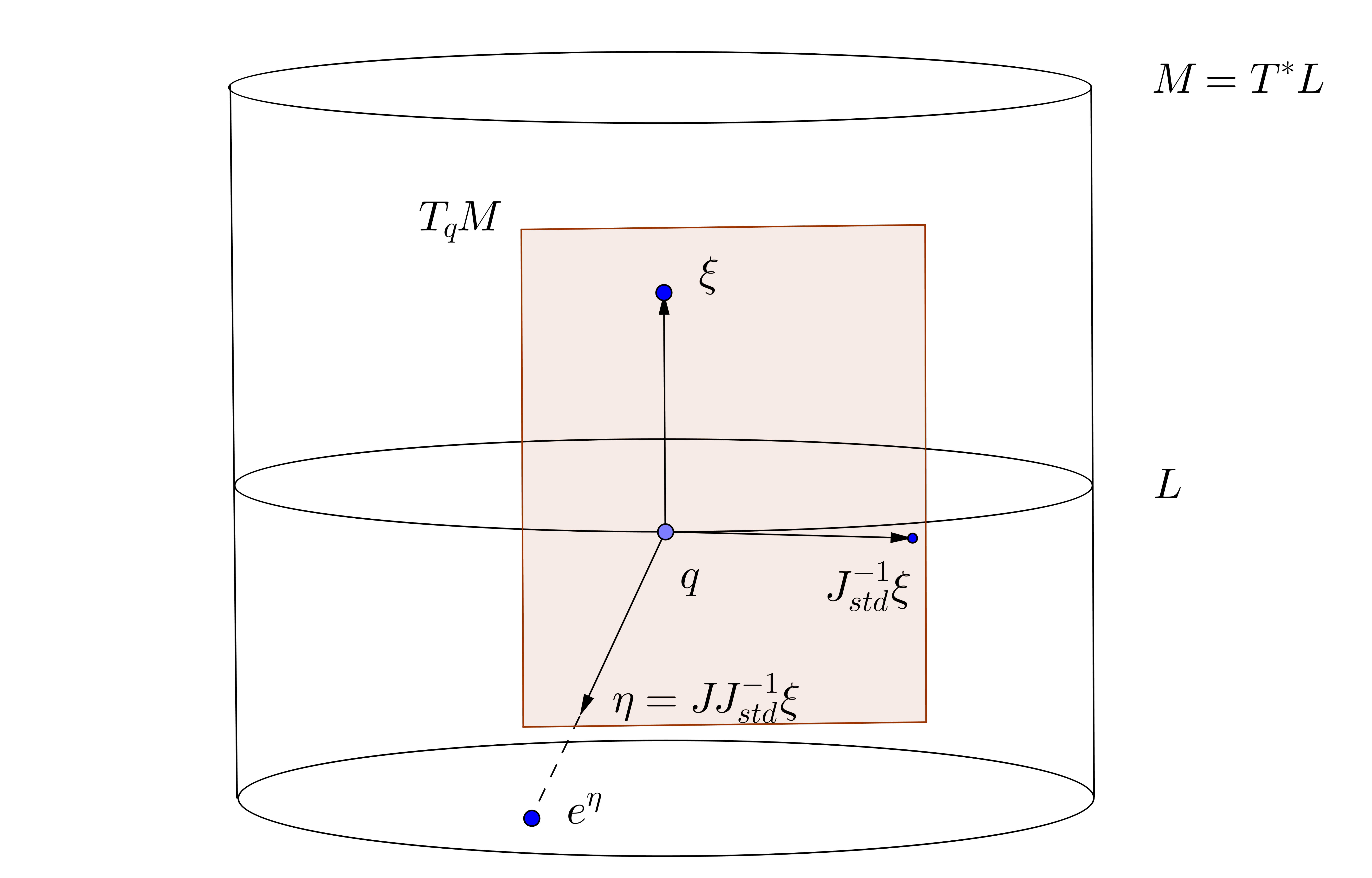}
\end{figure} \vskip0.3cm  
 
The derivative of this map is easily computed at points of $L$:
\begin{align*}
 df_{\mid TL}&=\Id\\
 df_{\mid J_{std}TL}&=JJ_{std}^{-1} 
\end{align*}
We assert that the form $f^*\omega$ is $J_{std}$-adapted in any sufficiently 
small neighbourhood of $L$. Indeed, for any point $q \in L$ and any $\xi \in  T_q L$ , we have 
 \begin{align*}
  f^*\omega_{(q,0)}(\xi,J_{std} \xi)&=\omega_{(q,0)}(df(x)\xi,df(x)J_{std}\xi) \\
  &=\omega_{(q,0)}(\xi,J J_{std}^{-1}J_{std}\xi)\\
  &=\omega_{(q,0)}(\xi, J\xi) >0 
 \end{align*}
As the condition of being adapted is open, this proves the assertion. Lemma~\ref{L::cone} shows that the path 
 $$\omega_t:=t\omega_{std}+(1-t)f^*\omega$$ remains inside the space of symplectic 
forms. Therefore, using Theorem~\ref{T::Darboux_Global} with $M=T^*L$, we get 
that $(T^*L,\omega_{std})$ and $(T^*L,f^*\omega)$ are symplectomorphic, in a sufficiently small neighbourhood of $L$.
This concludes the proof of the theorem.  \end{proof}
\section{The relative Darboux-Weinstein theorem}
As is customary in algebraic geometry, it is useful to look for a {\em relative version} of the Darboux-Weinstein
theorem, where one considers families of manifolds over a base $S$ and which specialise to
it in the case that $S$ reduces to a point. 

A fibration of $C^{\infty}$-manifolds
 $$\pi:M \to S $$
 is called {\em symplectic}, if it carries a two-form which restricts to a symplectic form on all the fibres $M_s=\pi^{-1}(s)$. In the spirit of Grothendieck, $M$ is viewed as a symplectic manifold {\em over the base $S$}. Similarly, a {\em Lagrangian manifold} $L$ over $S$ is a 
fibration $\rho: L \hookrightarrow S$, sitting in a diagram

$$ \xymatrix{L \ar@{^{(}->}[rr] \ar[rd]_{\rho} & & M \ar[ld]^{\pi} \\
   &S&
  }
$$

Thus the fibre $L_s$ at $s$ is a Lagrangian submanifold of $M_s$. Recall that for such a relative manifold $\rho: L \to S$ there is 
a natural  surjective map of vector bundles on $L$:
\[  TL \to \rho^*(TS), \;\;v \mapsto d\rho(v) . \]
The kernel of this map is $T_SL$, called the {\em relative tangent space}. It consists of those
tangent vectors of $L$ that are tangent to the fibres of $\rho$. In particular, if $\rho$ is a trivial fibration with fibre $L_0$
then $T_SL$ is simply $TL_0 \times S$.

The dual of the relative tangent bundle we
denote by $T^*_SL$: it sits in an exact sequence
\[ 0 \to \rho^*(T^*S) \to T^*L \to T^*_SL \to 0\]
of vector bundles on $L$. The restriction of $T^*_SL$ to the fibre $L_s$ is just the
cotangent bundle of $L_s$:
\[ T^*_SL_{|L_s}=T^*L_s\]

The following parametric version of the Darboux-Weinstein theorem expresses the fact 
that there are no local invariants for {\em families} of symplectic manifolds.\\

\begin{theorem} Let $L \to S $ be a proper Lagrangian submanifold of a symplectic manifold $M \to S$
over some base $S$. The germ of $M \to S$ along $L \to S$
is, locally on the base $S$, symplectomorphic to a the germ of $T_S^*L$ along the zero section.
 \end{theorem} 
\begin{proof}
 To adapt the proof the Darboux-Weinstein theorem to this parametric case, we use the relative variant of differential forms. The {\em relative de Rham complex} is defined by
 $$\Omega^\bullet_{\pi}:=\Omega^\bullet_M/(\pi^*\Omega^1_S \w \Omega^{\bullet-1}_M), $$
where $\Omega_N^p$ denotes the space of $C^{\infty}$ $p$-forms on a manifold $N$. 
 If we take local coordinates $s_1,\dots,s_n$ which trivialise the fibration, two differential forms on $M$ are equal as elements of $\Omega^\bullet_{\pi} $ if they are equal modulo a form of the type 
 $$\sum \a_i \w ds_i,\ \a_i \in \Omega^\bullet_M. $$
The cohomology of this relative de Rham complex is, locally on the base, the cohomology of the fibre tensored with
the $C^{\infty}$-functions on the base. In particular, the cohomology class of a symplectic form vanishes in a neighbourhood of a
Lagrangian fibre. The deformation argument remains the same with the only difference that we now consider the forms $\omega_t$ entering 
the equation:
$$i_{X_t} \omega_t=-\dot \a_t$$
as relative differential forms with vanishing de Rham class.
 \end{proof}

 \section{Liouville integrability}

The classical language of {\em Poisson brackets}\index{Poisson-bracket} 
is a very useful way to express the main features of symplectic geometry. 

For two smooth functions
  $$f,g:M \to \RM, $$
the Poisson-bracket of $f$ and $g$ is defined by
  $$\{ f,g \}=\omega(X_f,X_g)=-\{g,f\}. $$
It is readily seen that
\[ \{f,g\}=\Lt_{X_f} (g),\;\;\; [X_f,X_g]=X_{\{f,g\}} .\]

The Poisson-bracket is a {\em bi-derivation}\index{bi-derivation} 
\[ \{f,gh\}=\{f,g\}h+\{f,h\}g,\;\;\{fg,h\}=f\{g,h\}+g\{f,h\}\]
which satisfies the {\em Jacobi identity}\index{Jacobi identity}:
 $$\{ f,\{g,h\} \}+\{ g,\{h,f\} \}+\{ h,\{f,g \}\}=0 $$
  
The Darboux coordinates provide a local model for the Poisson bracket:
\[ \{ f,g \}=\sum_{i=1}^n\d_{p_i} f \d_{q_i}g-\d_{q_i} f \d_{p_i}g.  \]
More generally, a {\em Poisson structure}\index{Poisson structure} on a manifold is 
such an anti-symmetric bi-derivation satisfying the Jacobi identity. 
The dynamics of a Hamiltonian $H: M \to \RM$ can be expressed concisely by the statement
\[ \dot G=\{H,G\}, \]
which indeed reduce to Hamilton's equations of motion for Darboux coordinates.  
From this we see that a quantity
$$G:M \to \RM$$
 is preserved by the flow of $H$ if and only if it {\em Poisson-commutes} with $H$:
$$\{ H,G \}=0.$$
 Such a quantity is called a {\em first integral}.  Any function $\phi$ of $H$ is a first integral: 
 $$ \Lt_{X_H}\phi(H)=\{ H,\phi(H) \}=0$$
In some cases the converse is true as well. For instance, endow $M=\RM^2$ with its canonical symplectic structure
 $\omega=dq \w dp$ and take $H=p$. Then any first integral is a function $G$ with the property that
 $$\{ H,G \}=0 \iff \d_q G=0. $$
so depends only on the variable $p$.
 A Hamiltonian $H$ is called {\em Liouville integrable} or simply {\em integrable} if there exists Poisson commuting functions $H=f_1,\dots,f_n$ with Hamiltonian vector
 fields that are linearly independent 
 $$df_1 \w df_2 \w \dots \w df_n \neq 0$$
 on an open dense subset of $M$. 
The map $f=(f_1,f_2,\dots,f_n)$ is called a {\em moment mapping}\index{moment mapping}. The smooth fibres of this map 
are automatically Lagrangian. Indeed, the Hamiltonian vector fields $X_i$ of the $f_i$'s generate the tangent bundle and
 $$\omega(X_i,X_j)=\{ f_i,f_j\}=0. $$
 
In particular if $f$ is smooth and proper, then it defines what is called a {\em Lagrangian fibration}\index{Lagrangian fibration}. The Hamiltonian flows then
induce a cocompact $\RM^n$-action, and therefore, by general properties of $\RM^n$-actions, the moment mapping is 
a fibration whose fibres are tori.
    
A moment mapping defined over some $n$-dimensional base 
$$f=(f_1,\dots,f_n):X \to S $$
gives rise to  a Lagrangian subspace in a symplectic manifold in the following way: the projection on the 
second factor gives a symplectic manifold
$$M:=X \times S \to S $$
over the base $S$. We denote the graph of $f$  by 
$$L_f \subset M=X \times S.$$ 
It is a Lagrangian submanifold of $M$ over $S$.
By the relative Darboux-Weinstein theorem, locally on $S$, a neighbourhood of $L_f$ in $M$ is 
symplectomorphic to a neighbourhood of the zero section its cotangent bundle.

 \section{Action-angle variables}
If a moment mapping $f$ is proper, then its fibres  have a co-compact $\RM^n$-action coming from 
the commuting flows of $f_1,\dots,f_n$. The Lagrangian manifold $L_f$ is therefore a fibration by tori 
over its base $S$. Action-angle variables\index{action-angle variable} provide Darboux coordinates 
adapted to this situation. The action variables $I_1,I_2,\ldots,I_n$ are functions of the $f_i$, so they are constant on the tori of the fibration and have the further property that the integral curves of the 
corresponding Hamiltonian fields $X_i$ all have a {\em fixed period}. The angle variables $\theta_1,\theta_2,\ldots,\theta_n$
are the corresponding 'time' variables, so have the property that
\[\omega=d\theta_1 \wedge d I_1+d\theta_2\wedge d I_2+\ldots+d\theta_n\wedge d I_n . \]
Before we explain this in detail, let us start with some examples.

\begin{example}
The restriction of the map
$$\RM^{2n} \to \RM^n,\ (q,p) \mapsto  (\frac{1}{2}(q_1^2+p_1^2),\dots,\frac{1}{2}(q_n^2+p_n^2))  $$
to the preimage $X$ of any open subset $S \subset \RM_{>0}^n$ produces 
a Lagrangian fibration
  \[f:X \to S . \]
The Hamiltonian vector field of $f_i=\frac{1}{2}(q_i^2+p_i^2)$ is
\[X_i=p_i\d_{q_i}-q_i\d_{p_i},\ i=1,\dots,n .\]
These vector fields $X_i$ span the tangent spaces to the fibres. 
We get action-angle coordinates by putting
$$I_i=f_i,\;\;\;\; \theta_i=\arctan \frac{p_i}{q_i} $$
or equivalently
  $$p_i:=\sqrt{I_i}\cos \theta_i,\ q_i:=\sqrt{I_i}\sin \theta_i .$$

  \vskip0.3cm  \begin{figure}[htb!]
  \includegraphics[width=12cm]{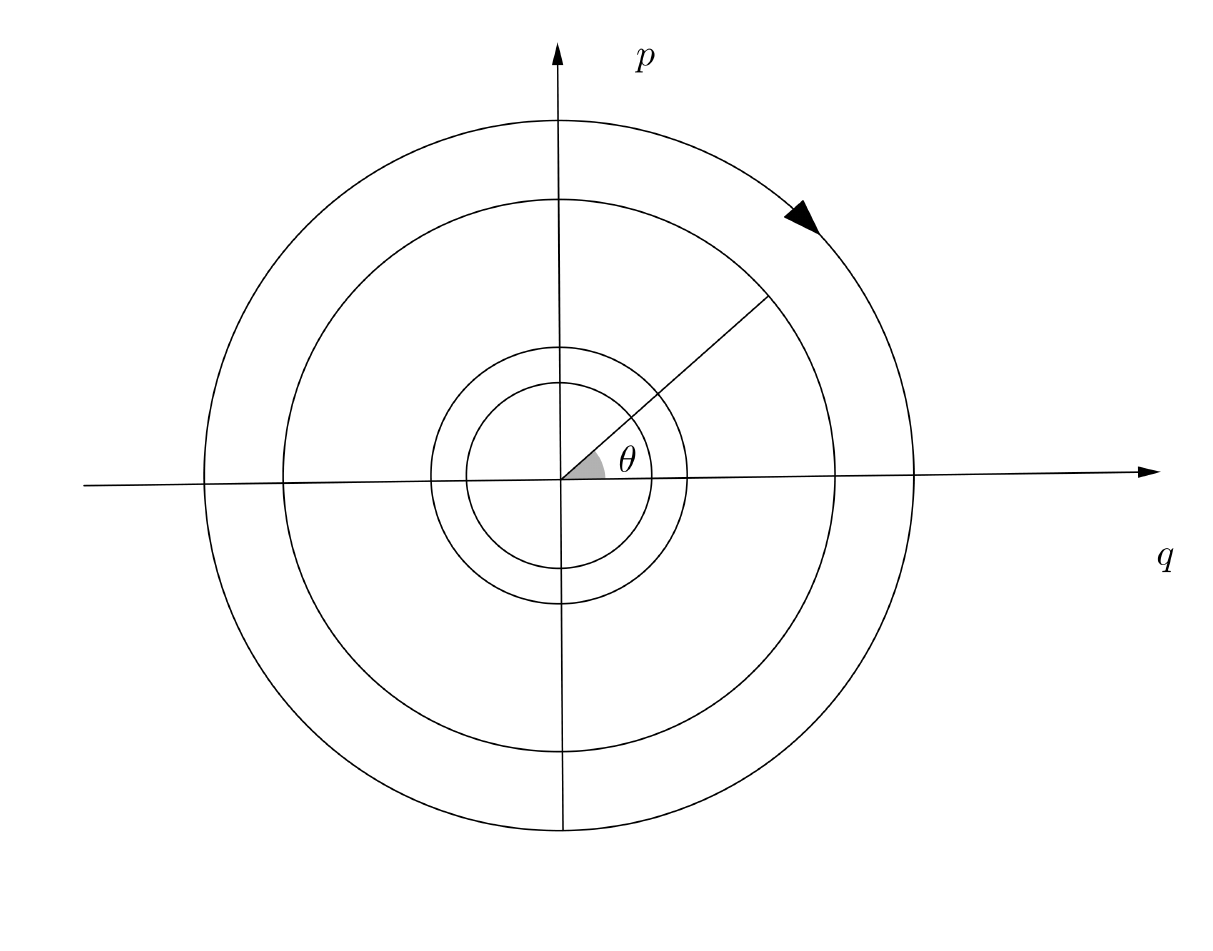}
\end{figure} \vskip0.3cm  

Note that the function $\theta_i$ is multi-valued, but its 
differential  $d\theta_i$ is equal to the closed one-form
\[ \alpha_i:= \frac{-q_i}{p_i^2+q_i^2} dp_i +\frac{p_i}{p_i^2+q_i^2} dq_i\]
These closed one-forms $\alpha_i$ form a dual basis to commuting
vector fields $X_i$:
\[ \alpha_i(X_j)= \delta_{ij}\]

\end{example}

\begin{example}
The above example is archetypical for any integrable system, but untypical
in the sense that the action variables were very simple. In general the 
determination of the action-angle variables lead to integrals 
defining transcendental functions. Consider for example the case $n=1$ and:
 $$f(q,p)=p^2+\frac{1}{2}q^2-\frac{1}{3}q^3.$$
 Let $S$ be a pointed disc  of radius $r<1/6$ and $X$ be the preimage of $S$ under the polynomial
 map $f$. This defines a locally trivial $S^1$-fibration, that we denote in the same way:
 $$f:X  \to S .$$
 
\begin{center}
  \includegraphics[width=11cm]{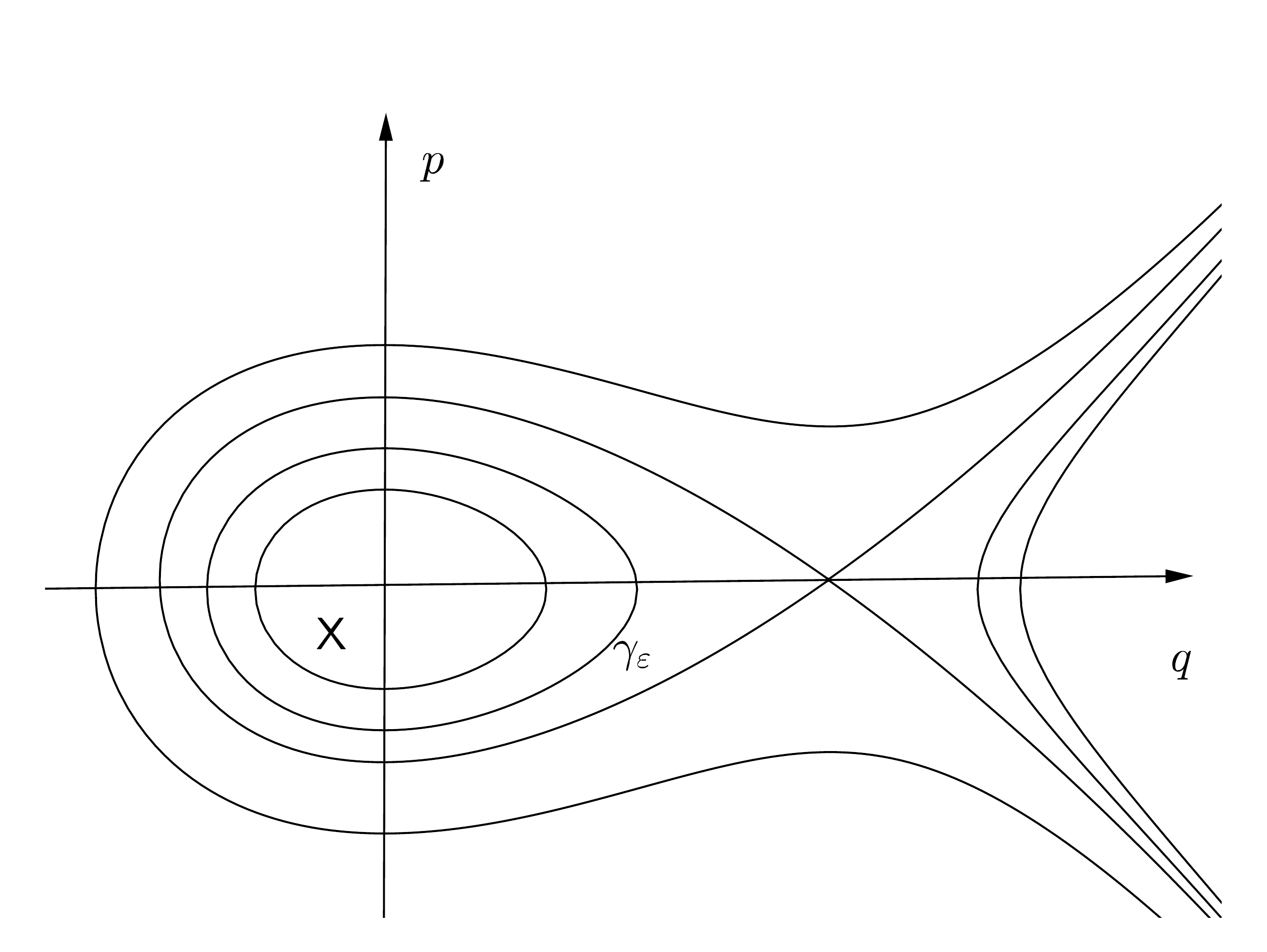}
\end{center}
In that case, the action is the function
$$I=\int_{\g_\e} pdq$$
 where $\g_\e$ is the small loop inside the level set 
 $$\{(q,p) \in X:|\;f(q,p)=\e\} \approx S^1.$$
 By Stokes' formula, this is just the area of the disc bounded by $\g_{\e}$
and is thus given as an elliptic integral of the second kind.
If we regard $I$ as a function of $q,p$, then the {\em angle $\theta$} is the time of this vector field. It is given by the indefinite integral
$$\theta=\int_{*}^{(q,p)}\frac{dp \w dq}{dI} $$
and is defined only modulo a period. This is a generic feature: computing action-angle coordinates involves abelian integrals and theta functions.
\end{example}

Let us now consider the general case. We consider a proper moment map 
$$f=(f_1,\dots,f_n) :X \to S,$$
pick a reference point $0 \in S$ and let $L_0:=f^{-1}(0)$ be the corresponding
reference torus. By shrinking $S$, we may assume that $S$ is contractible and
the fibration is trivial. The total space $X$ retracts to the torus $L_0$ and 
because the symplectic form $\omega$ vanishes on $L_0$, the form $\omega$ is
exact on $X$, so we can find an action form $\a$ for  $\omega$:
$$\omega=d\a.$$ 

For each cycle $\gamma(0) \in H_1(L_0,\ZM)$, we obtain, by parallel transport 
to neighbouring tori, a family of cycles $\gamma(s) \in H_1(L_s,\ZM)$.
This defines an {\em action integral}\index{action integral}:
\[ I_{\gamma}:S \to \R,\;\;s \mapsto I_{\gamma}(s):= \int_{\gamma(s)} \alpha .\]
If we pick a basis $\g_1(s),\g_2(s),\ldots,\g_n(s)$ for $H_1(L_s,\ZM)$, we
obtain the {\em action functions}, or {\em action variables}, $I_j$\index{action variables} 
by composing the action integrals 
$$\int_{\g_j(s)}\a. $$ with the map $f$:
$$I_j(q,p)=I_{\g_j}(f(p,q)).$$
Note that a change of the homology basis induces an $GL(n,\ZM)$-action
on the possible choices for the action variables. 

The classical Arnold-Liouville-Mineur theorem can be stated as follows:
\begin{theorem}
 Let $f:X \to S$ be a proper integrable system on a symplectic manifold $(M,\omega)$.
 There exists functions
 $$\theta_j:X \to S^1 ,$$
 defined locally near a fibre, called the angles, such that the action-angle functions induce a local symplectomorphism
 $$X \to (S^1)^n \times \RM^n ,$$
 where $(S^1)^n \times \RM^n=T^*(S^1)^n$ is equipped with the standard symplectic form.
\end{theorem}
\begin{proof}
The Darboux-Weinstein theorem implies that we may assume that $M=(S^1)^n \times \RM^n=\{(q,p) \} $
equipped with its standard symplectic structure and that the Lagrangian fibre of $f$ over $s_0 \in S$ is the zero
section. Up to a linear change of coordinates we may also assume that
$$f_i(q,p)=p_i+\dots $$
where the dots stand for higher order terms in the $p_i$'s depending on the $q$-variables.

By the implicit function theorem, the fibres of $f$ are the graphs of maps
$$g_s:L \to \RM^n $$
with 
$$f(q,p)=0 \iff p=g_s(q). $$
Comparing the Taylor series of $f$ and $g_s$, we see that
$$g_s=s+o(\| s \|) .$$
We choose the cycle $\g_j(s)$ which projects to the $j$-th 
coordinate circle in the torus. Then the actions are defined by
 \begin{align*}
 I_j(q,p)&=\int_{\g_j(s)}\sum_{i=1}^n p_idq_i \\
   &=\int_{\g_j(s)}\sum_{i=1}^n (g_i)_s(q)dq_i,\\
  & =(g_j)_s(q=0)\\
  &= s_j+o(\|s \|)\\
  &=p_j+o(\|p \|)
 \end{align*}
 The associated angles $\theta_j$ are of the form
 $$\theta_j=q_j+ o(\|q \|)$$
 Thus the map
 $$(q,p) \mapsto (\theta,I) $$
 is a diffeomorphism and therefore a symplectomorphism.
\end{proof}
\section{Integrable systems and the Gauss-Manin connection}
The above proof for the existence of action-angle variables relies on blind
computations. One of the beautiful aspects of classical integrable systems is that the 
construction  of action-angle coordinates can also be explained in terms symplectic geometry 
of the  action integrals\index{action integral} $I_{\gamma}$ that we defined above.

 We consider a torus bundle
 $$f:X \to S $$
 defined by a proper smooth moment mapping over a contractible open set $S \subset \RM^n$.
We denote by $X_i$ the Hamiltonian fields of its components and choose a Lagrangian section of the torus bundle
 $$\s:S \to X,\ \s^*\omega=0. $$
 
 The Hamiltonian flows of the $X_i$'s induce an $\RM^n$-action on the fibres of $f$, which is, as any such $\RM^n$-action, 
 transitive on the fibres. Thus translating the section $\s$ by the Hamiltonian vector fields define sections starting at 
 arbitrary points on a fibre. Thus for each point $x \in X$, we get a vector space $H_x \subset T_x X$ complementary to the 
 tangent space of the fibre, that is, we get an {\em Ehresmann connection}\index{Ehresmann connection} on $X$. 
 
 In particular, each vector field $v$ on the base $S$  is lifted in a unique way to a vector field $\t v$ on $X$.
 \begin{lemma} One has
$$\omega=\sum_{i=1}^n \a_i \w df_i   $$
where $\a_i:=\omega(\tilde v_i,-)$ and 
\[ \tilde v_i := \tilde{\d}_{s_i}\]

 \end{lemma}
 \begin{proof}
 The vector fields $\t v_1,\dots,\t v_n$ lifting $\d_{s_1},\dots,\d_{s_n}$ satisfy
 $$\omega(X_i,\t v_j)=df_i(\t v_j)=ds_i(\d_{s_j})=\dt_{ij}. $$
 Moreover as the  section $\s$ is Lagrangian, we have:
 $$\omega(\t v_i,\t v_j)=0. $$
 Since the $X_i$ pairwise commute, we also have
 $$\omega(X_i,X_j)=0. $$
 This means exactly that   the symplectic form is $\sum_{i=1}^n df_i \w \a_i$. 
 \end{proof}
 \begin{corollary} The lifted vector field $\t v$ is symplectic:
$$\Lt_{\t v}\omega=0 .$$
 \end{corollary}
 \begin{proof}
 Any lifted vector field is a linear combination of the vector fields
$\t v_i$ with coefficients that are pulled back from $S$.
 \end{proof}
 By taking the interior product with the symplectic form $\omega$, a vector field $v$ is lifted to a one-form 
 $\a_v$.
 As symplectic vector fields correspond to closed one-forms, we obtain a map
 $$\Theta(S) \to \Omega^1(X)_{closed},\;\;\; v \mapsto \a_v $$
For each $s \in S$, this induces a map $T_sS \to H^1(L_s,\RM), v \mapsto [\a_v]$. 

\begin{lemma} For each $s \in S$, the map
$$T_sS \to H^1(L_s,\RM), v \mapsto [\a_v]. $$
is an isomorphism.
\end{lemma}

\begin{proof}
The forms $\a_1,\dots,\a_n$ symplectically associated to the vector fields
$\tilde{v_1},\tilde{v_2},\ldots,\tilde{v_n}$ lifting $\d_{s_1},\dots,\d_{s_n}$ 
form at each point the dual basis to hamiltonian vector fields $X_1,\dots,X_n$ 
and are therefore linearly independent.
\end{proof}

\begin{lemma}
The one-form $\a_v$ lifting a vector field $v \in \Theta_S$ satisfies
$$\Lt_{\t v}\a=\a_v+d(i_{\t v}\a) .$$
\end{lemma}
\begin{proof}
 By Cartan's formula
 \begin{align*}
 \Lt_{\t v}\a&= i_{\t v}d\a+d(i_{\t v}\a)\\
   &=i_{\t v}\omega+d(i_{\t v}\a)\\
   &= \a_v+d(i_{\t v}\a)
 \end{align*}
\end{proof}
 The derivative of an action-integral $I_{\g}$ with respect to the
variable $s_i=f_i$ is given by
\[\Lt_v I_{\gamma(s)} =\int_{\gamma(s)} \Lt_{\t v}\a =\int_{\gamma(s)} \alpha_v\]
which means that the Ehresmann connection induces the usual {\em Gauss-Manin connection}\index{Gauss-Manin connection} on relative differential forms.

As a corollary to our discussion, we get the
\begin{theorem}
 If $\g_1(s),\g_2(s),\ldots,\g_n(s) $ form a basis for the lattice 
$H_1(L_s,\ZM)$, then the associated map
\[ \Psi: S \to T \subset \RM^n, s \mapsto (I_{\g_1}(s),I_{\g_2}(s),\ldots,I_{\g_n}(s)) \]
is a local diffeomorphism. 
\end{theorem}
\begin{proof}
By De Rham duality, if the $\g_i(s)$'s form a basis then integration over the $\g_i(s)$'s are linearly independent
linear forms.
\end{proof}
\section{Affine structures and integrable systems}

Let us keep the notation of the previous section.

The base space of a proper smooth moment map carries an {\em affine structure}, that we shall now define.
Integration over a cycle $\gamma(s) \in H_1(L_s,\ZM)$ defines a linear function $$\int_{\gamma(s)}: H^1(L_s,\RM) \to \RM .$$ 
Identifying $H^1(L_s,\RM)$ with $T_sS$, its kernel defines a hyperplane $\g^{\perp}(s)$ in $T_sS$. 
In this way, we obtain a distribution $\gamma^{\perp}$ associated to a cycle and the level sets of $I_{\gamma}$ 
are tangent to this distribution.
 
If $\g_1(s),\g_2(s),\ldots,\g_n(s) $ form a basis for the lattice 
$H_1(L_s,\ZM)$ then the level-sets of the functions $I_{\g_i}$ are tangent to the distribution and define an affine structure
in the base of the moment map.\\
\newpage
\ \\
\vskip0.3cm 
\begin{figure}[htb!]
 \includegraphics[width=12cm]{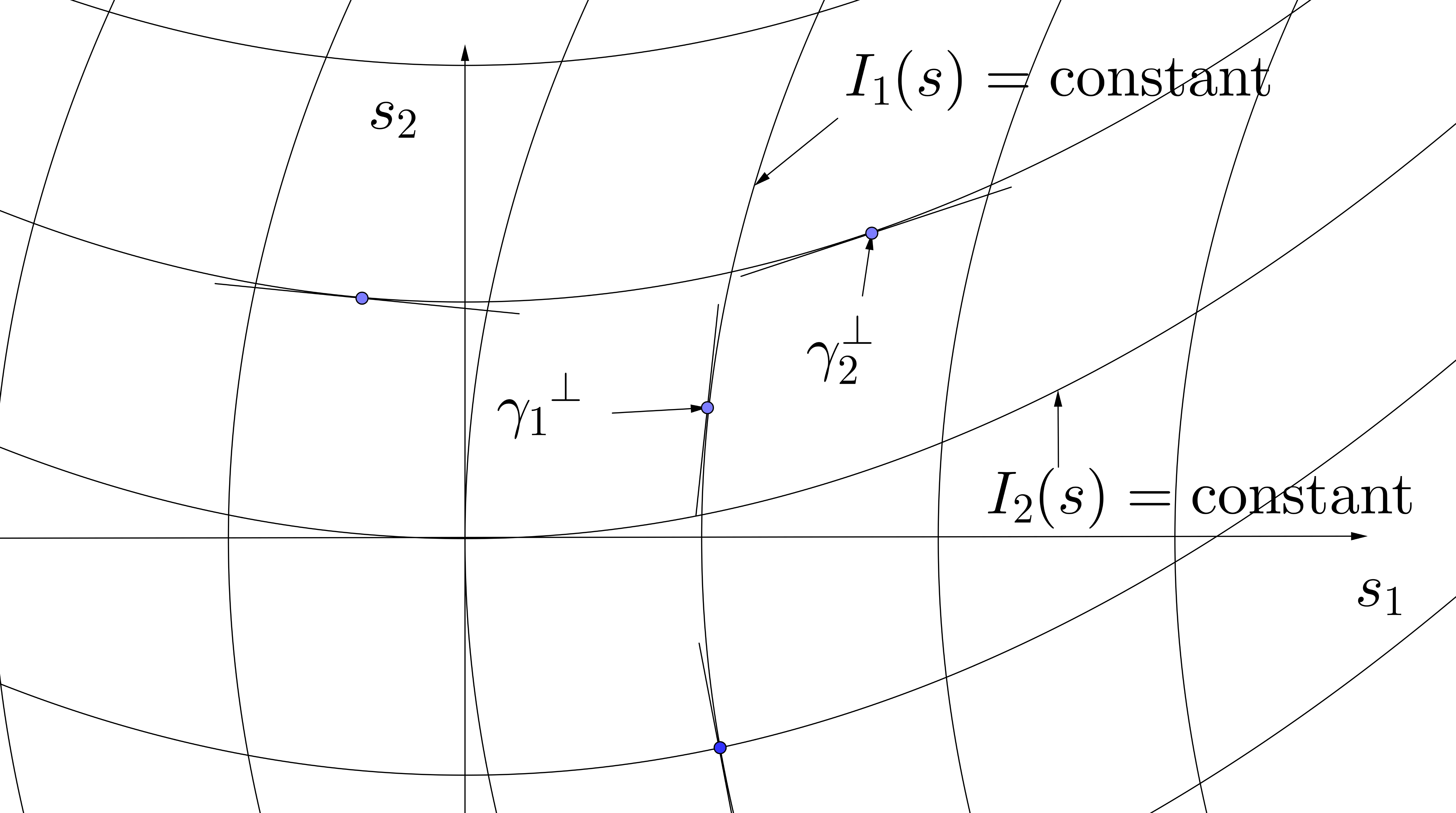}
\end{figure}
\vskip0.3cm 

If we now consider these integrals of functions depending on the variables $(q,p)$:
\[I_j(q,p)=I_{\g_j}(f(q,p)).\]
We get a new integrable system $I=(I_1,\dots,I_n)$ for which the affine structure is now linear:
$$\xymatrix{ & X \ar[ld]_-f \ar[rd]^I & \\  S \ar[rr]^-\Psi & & T}$$
meaning that {\em all action integrals $I_{\gamma}$ become linear functions} in the
standard coordinates $t_1,t_2,\ldots,t_n$ of $\RM^n \supset T$.

Via that map $\Psi$, the distribution  $\g^\perp$ becomes the parallel distribution of linear hyperplanes.\\
\vskip0.3cm 
\begin{figure}[htb!]
 \includegraphics[width=5cm]{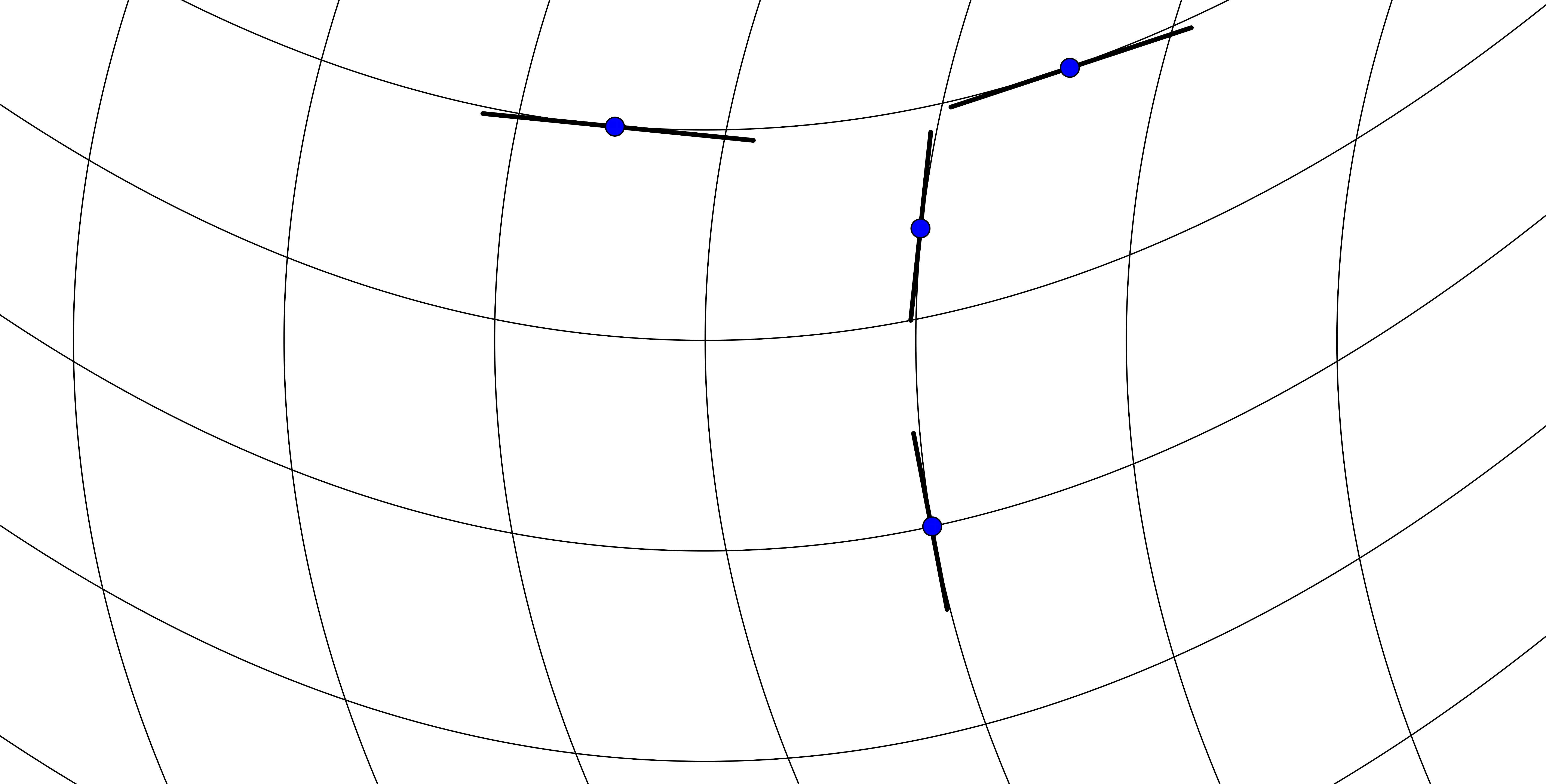}
  \includegraphics[width=2cm]{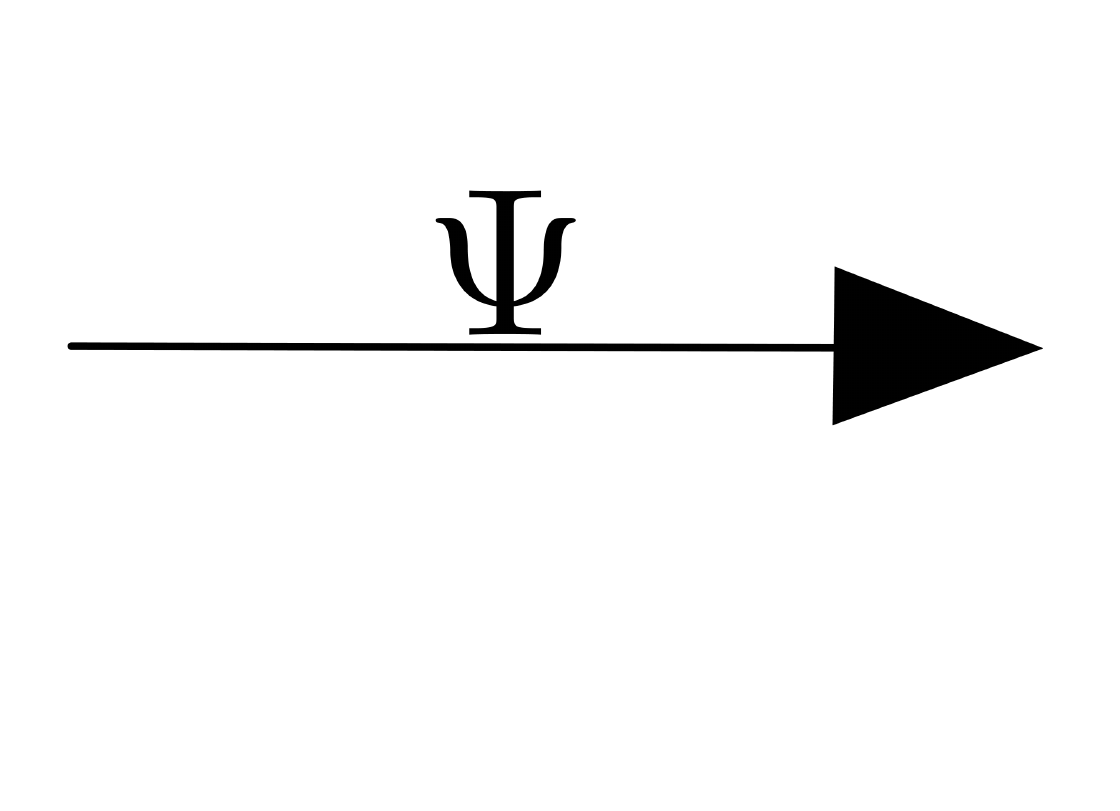}
 \includegraphics[width=5cm]{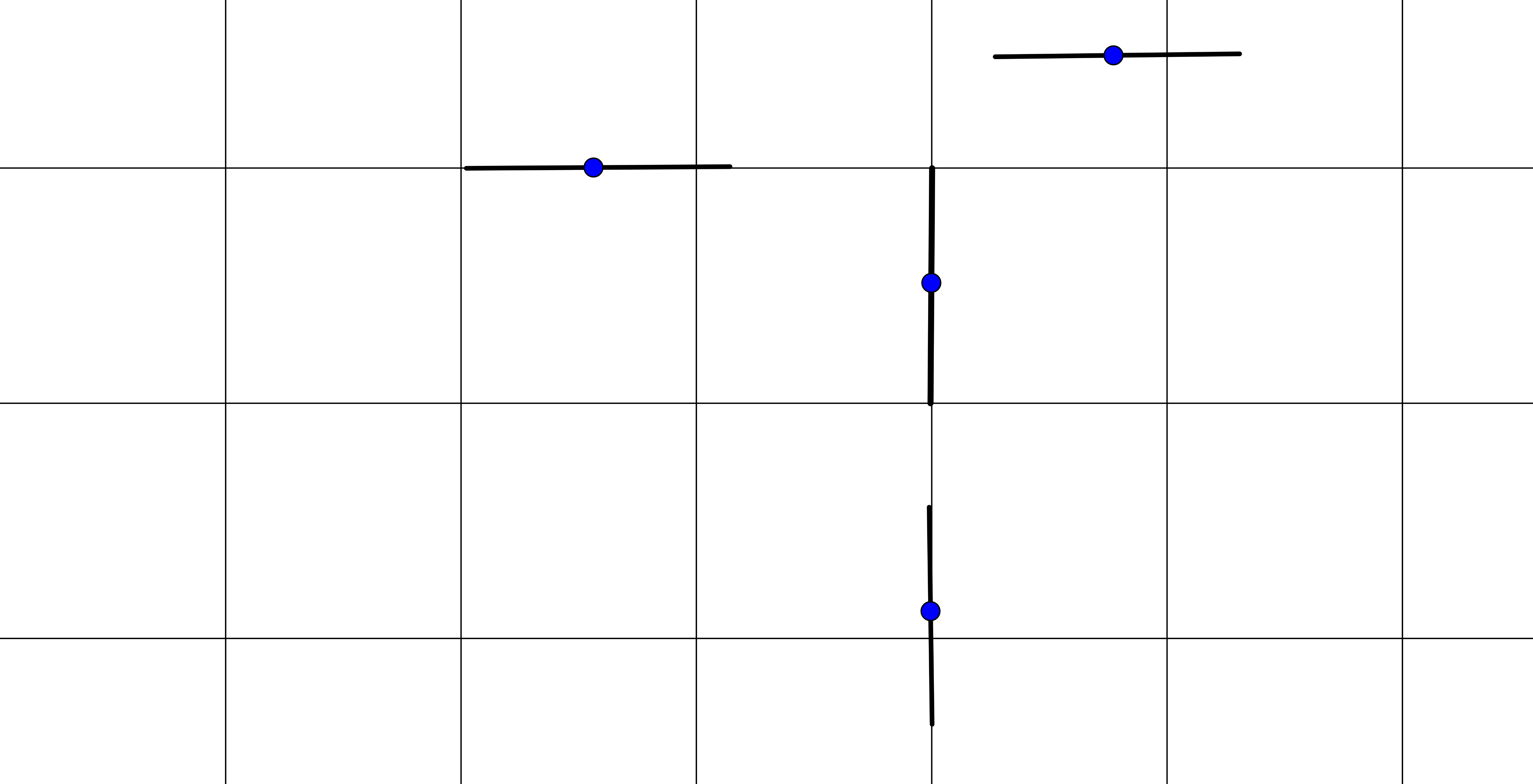}
\end{figure}
\vskip0.3cm 
Our original Lagrangian section of $X \to S$ induces a section of $X \to T$ by composition with the inverse of $\Psi$. 
In this way, we lift the vector fields $\d_{t_i}$ to closed one-forms $\b_1,\b_2,\ldots,\b_n$, using the new Ehresmann connection.

As $I_\g(t)$ is now linear in $t$, the periods of the $\b_i$'s  are all {\em constant}. 
Thus we define:
\[ \l_i:=\int_{\gamma(t)} \b_i=\d_{t_i} I_{\gamma}(t) \in \RM.\] 

The indefinite integral gives a well-defined map modulo the periods:
\[ \int_{\s(t)}^*: L_t \to \oplus_{i=1}^n \RM/\lambda_i\ZM,\ x \to (\int_{\s(t)}^x\beta_1,\int_{\s(t)}^x\beta_2,\ldots,\int_{\s(t)}^x\beta_n) \]
As the $\b_i$'s are linearly independent this map is a local diffeomorphism and by definition of the periods $\l_i$, it is
also one-to-one. Hence the induced torus fibration $X \to T$ is trivialised.

Such a trivialisation is, in general, not possible if the base $S$ of the torus fibration is not contractible.
Such examples occurs when we remove singular fibres from a larger family 
$$\overline{X} \to \overline{S}.$$
As a rule, the base $S$ of the fibration is then not contractible, and interesting global issues arise. 

The cycles $\g(s)$ may undergo monodromy\index{monodromy} under parallel transport along a loop in the base $S$, leading to
a non-trivial representation  of the fundamental group:
$$\rho: \pi_1(S,0) \to Aut(H_1(L_0,\ZM)).$$ 
The lattices $\L_s=H_1(L_s,\ZM)$ then form a non-trivial lattice bundle $\L$ over $S$, commonly referred to as
a {\em local system}\index{local system}.  Although this phenomenon obstructs the existence of action-angle variables,
the above constructions can be done  {\em locally on $S$}. This leads to a version of action angle coordinates
using a certain tautological symplectic  torus bundle $T^*S/\L$ associated to a Lagrangian section of $f:X \to S$ that we will describe now.

As the fibre $L_s$ is a torus with transitive $\RM^n$-action, the choice of a point $x \in L_s$ determines a map
$$ \RM^n \to L_s,\;\;0 \mapsto x$$
and as this map is the universal covering map, the kernel can be identified with $H_1(L,\ZM)$. 
In this way $L_s$ is, after choosing a point as origin, identified with $H_1(L_s,\RM)/H_1(L_s,\ZM)$.
 
When we dualise the isomorphism
$$T_sS \to H^1(L_s,\RM),\ v \mapsto [\a_v] $$
we get an isomorphism
\[ T_sS \to H_1(L_s,\RM) \to T^*_sS\]
Thus we get a natural lattice
$$ \L_s \subset T_s^*S $$
in the cotangent space at $s$, isomorphic to first homology group $H_1(L_s,\ZM)$ of the fibre. When we let $s$ run in $S$, we get a natural symplectic torus bundle attached to the situation:
$$T^*S/\L \to S. $$
isomorphic to
$$X \to S.$$

This dicussion can be summarised as follows
 \begin{theorem} 
 \label{T::Duistermaat}
Let $f:X \to S$ be a proper smooth integrable system with a Lagrangian section
  $$ \s:S \to X .$$
The identification of the fibre of $f$ at $s$ with $H_1(L_s,\RM)/H_1(L_s,\ZM)$ induces a {\em symplectomorphism of torus bundles over $S$}  

$$ \xymatrix{T^*S/\L \ar[rr] \ar[rd] & & X \ar[ld] \\
   &S&
  }
  $$
 where $\L_s$ is identified with $H_1(L_s,\ZM)$ via the isomorphism $T^*_s S \to H_1(L_s,\RM)$ induced by the section $\s$.
 \end{theorem}

\section{Bibliographical notes}
The technique used to prove the Darboux theorem is called {\em Moser's path homotopy method}. It was introduced in:\\

{\sc J. Moser}, {\em On the volume elements on a manifold}, Trans. of the Am. Math. Soc.{\bf 120}(2), 286-294, (1965).\\

The Darboux-Weinstein theorem is proved in:\\

{\sc A. Weinstein}, {\em Lagrangian submanifolds and Hamiltonian systems}, Annals of Math. {\bf 98}, 377-410, (1973).\\

Almost complex structures were introduced in symplectic geometry by Gromov:\\

{\sc Gromov, M. }{\em Pseudo holomorphic curves in symplectic manifolds,} Inventiones Mathematicae, {\bf 82}(2), 307-347, (1985).\\

The existence of action-angle coordinates for smooth proper integrable systems was first proved in:\\

{\sc H. Mineur}, {\em Sur les syst\`emes m\'ecaniques admettant n int\'egrales premieres uniformes et l'extension \`a ces syst\`emes de la m\'ethode de quantification de Sommerfeld}, C. R. Acad. Sci., Paris 200, 1571 - 1573, (1935).\\

For a long time, the work of Mineur was forgotten and the result was rediscovered by Arnold in:\\

{\sc V.I. Arnold}, {\em A theorem of Liouville concerning integrable problems of dynamics}, Sibirsk. Math . Zh. {\bf 4} , 471 - 474, (1963).\\

Theorem \ref{T::Duistermaat} is due to Duistermaat. In his paper, he also discusses monodromy obstructions 
to the global existence of action-angle variables:\\

{\sc H. Duistermaat}, {\em On global action-angle variables}, Communications on
Pure and Applies Mathematics, {\bf 33}(6), 687-706, (1980).\\

The beautiful effect of monodromy on the semi-classical spectrum for the 
spherical pendulum was described in:\\

{\sc R. Cushman, H. Duistermaat},{\em The quantum mechanical spherical pendulum}, Bull. Amer. Math. Soc. (N.S.),
Volume 19, Number 2, 475-479, (1988)\\

and for more general integrable systems in:\\

{\sc S. Vu Ngoc}, {\em Quantum monodromy in integrable systems}, Communications in mathematical physics, 203(2), 465-479, (1999).\\

The proof of the existence of action-angle variables and its generalisation using a relative Darboux-Weinstein theorem over a base was introduced to prove the absence of obstruction to quantisation, see:\\

{\sc M. Garay and D. van Straten} {\em Classical and quantum integrability}. Mosc. Math. Journal {\bf 10}, 519-545, (2010).\\
There are some misprints in the paper which we (hopefully!) corrected here.\\

Classical sources on classical mechanics in the language of global analysis are:\\

{\sc V.I. Arnold}, {\em Mathematical methods of classical mechanics}, Graduate texts in Mathematics 
(Vol. 60). Springer Verlag.\\

{\sc R. Abraham and J. Marsden}, {\em Foundations of Mechanics}, Benjamin/Cummings Publishing Company, (1978).\\

For more detail on symplectic geometry, we also refer to:\\

{\sc D. McDuff, D. Salamon} {\em Introduction to Symplectic Topology}, Oxford mathematical monographs, Clarendon Press, 1998.
\newpage
 \chapter{The KAM problem}
We now come to a central theme of classical KAM-theory, namely that of the persistence of quasi-periodic motions after
perturbation of an integrable Hamiltonian system. It is convenient to work with an algebraic model and formulate
the problem in terms of certain Poisson-algebras associated to the algebraic torus, like the algebra of analytic Fourier series. We discuss the problem first on the level of formal power series and then discuss the special features that arise at the analytic level and identify the first order analytic obstruction. 

 \section{Quasi-periodic motions\index{quasi-periodic motion}}
Let us now look at the motion  for an integrable Hamiltonian $H=f(I)$ in action angle coordinates:
  $$(\theta_1\dots,\theta_n,I_1,\dots,I_n),\ \{ I_j,\theta_k \}=\dt_{jk}.$$
Hamiltons equations of motion are
$$\left\{ \begin{matrix}
\dot I_j &=&0 \\
\dot \theta_j&=&\d_{I_j} f(I) .
\end{matrix} 
\right.
$$
These equations with initial condition 
$$I(0)=c,\;\;\;\theta(0)=\b$$
 can easily be integrated:
$$\left\{  \begin{matrix}
 I_j &=&c_j \\
 \theta_j&=&\omega_j t+\b_j
\end{matrix} \right. $$
where
\[\omega_i:=\d_{I_j} f(c)\;. \]
\newpage
\ \\
\vskip0.3cm
\begin{figure}[htb!]
  \includegraphics[width=11cm]{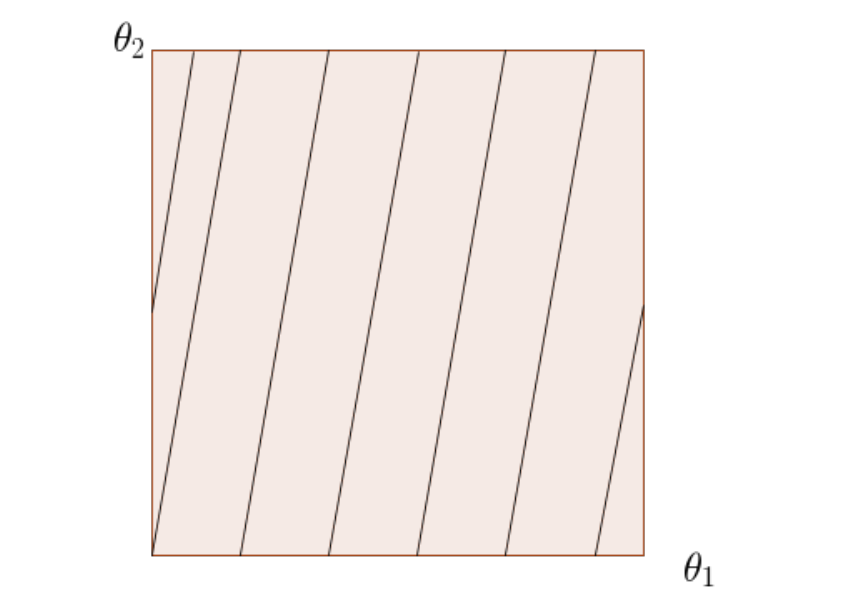}
\end{figure}
\vskip0.3cm
The trajectory lies on the torus $I=c$ and such a motion is called {\em quasi-periodic.} 
A quasi-periodic motion is dense on its torus,  if the {\em the frequency vector} of the motion 
\[\omega:=(\omega_1,\omega_2,\ldots,\omega_n)\]
has $\ZM$-independent coordinates. Of course, the vector $\omega$ will depend on $c$.

The question from which KAM-theory originates is the following:

{\em Given an integrable system in action-angle coordinates,  
do such quasi-periodic motions persist after turning on a 
perturbation of the original Hamiltonian?}  

In order to formulate this problem more precisely, it is convenient to introduce first an algebra-geometric model of the above situation. Then we will consider the formal and the analytic versions.\\

Rather then working with the angular variables $\theta_j$ it is convenient to use variables
\[ q_j := e^{i \theta_j} .\]
For real values of $\theta_j$, the variable $q_j$ runs over the unit circle in the complex plane. 
Adding an imaginary part to $\theta_j$ changes the radius of the circle traced by $q_j$. 
So we are led to consider the {\em algebraic torus}\index{algebraic torus}
$$X:=(\CM^*)^n=\{ (q_1,\dots,q_n) \in \CM^n: q_i \neq 0, i=1,2,\ldots,n\}. $$
Note that $X$ has a anti-holomorphic involution
$$ q_i \mapsto \frac{1}{\bar q_i}$$
and the usual torus $(S^1)^n$ traced by all the angular variables $\theta_j$ is the part of the algebraic torus $X$ that is fixed by this involution, so it may be considered as the 'real part' of $X$ for the real structure defined by this anti-homomorphic involution. 

Note that   
\[ \{I_j,q_k\}=\{I_j,e^{i \theta_k}\} = ie^{i\theta_k} \delta_{jk} .\]

We also will write $p_j$ for the conjugate action variables $-i I_j$, and note that then
\[
 \{p_j,q_k\}=-i \{I_j,e^{i \theta_k}\} = e^{i\theta_k} \delta_{jk} =q_k \delta_{jk} .
\]

\section{Algebraic model}
We started our discussion of symplectic geometry in the context of differential geometry, 
but now we wish to take a more algebraic viewpoint. In differential geometry, the basic 
notion is that of a manifold, but in  algebra one considers rings of functions on the manifold 
under consideration. Usually there are several natural choices for the types of functions one 
might want to consider: $C^{\infty}$, analytic, polynomial, convergent or formal power series.

The Poisson bracket of functions on a symplectic manifold gives the rings of functions the 
additional structure of a {\em Poisson algebra}\index{Poisson algebra}. 
We recall that a Poisson algebra over a (commutative) ring $R$ is an $R$-algebra with an 
additional anti-symmetric operation $\{-,-\}$ which satisfy the Jacobi identity
$$\{ \{ f,g \},h \}+ \{ \{ g,h \},f \}+\{ \{ h,f \},g \}=0$$
and which is a bi-derivation, that is
$$g \mapsto \{ f, g \} $$
is a derivation for any fixed $f$:
\[ \{f,gh\}=\{f,g\}h+g\{f,h\} .\] 
As the bracket is anti-symmetric, it is also a derivation on the first variable. For a given 
real symplectic manifold $M$, the $\RM$-algebra $C^\infty(M)$ is naturally a Poisson $\RM$-algebra.\\

The algebraic model of the family of tori is the {\em Laurent polynomial ring}\index{Laurent polynomials}
\[ A:=\CM[q,q^{-1},p]:=\CM[q_1,q_2,\ldots,q_1^{-1},q_2^{-1},\ldots,q_n^{-1},p_1,p_2,\ldots,p_n] \]     
with the Poisson-bracket  
\[  \{p_j,q_k\}=q_k \dt_{jk} .\]
The Poisson-algebra $A$ can be seen as the {\em affine coordinate ring\index{affine coordinate ring} of the 
cotangent bundle $T^*X$} of our torus $X$. 
The canonical symplectic form on that space is 
\[ \omega:=\sum_{j=1}^n \frac{dq_j}{q_j}\wedge dp_j \]
and leads precisely to the above Poisson-bracket. The fibres of the map
\[ T^*X \longrightarrow \CM^n, (q,p)\mapsto p=(p_1,p_2,\ldots,p_n) \]
are copies of $X$ and the flow of any Hamiltonian $H=f(p) \in \CM[p_1,p_2,\ldots,p_n]$ preserves these fibres.\\

We also note that in the sense of algebraic geometry, quasi-periodic motion 
can be defined ``over any field $K$'' by replacing the coefficients $\CM$ by
$K$, and then considering the Poisson-algebra
\[ A=K[q,q^{-1},p] \]
with the Poisson-bracket as defined before. 
\section{Perturbation theory}
To describe a perturbation of a Hamiltonian  $H \in A$ we ``add a parameter $t$ to the ring $A$'' and consider the ring of Laurent polynomials 
\[ B:=K[q,q^{-1},p,t] .\]
Here $t$ is a parameter, which is supposed to be a {\em central element}\index{central element}, i.e. an element that Poisson-commutes with all elements  $f \in B$: 
\[\{t,f\}=0.\]

By a {\em deformation  or perturbation of $H$}\index{perturbation}\index{deformation} we mean any element $H' \in B$ which reduces to $H$ when we put $t=0$, which means that one can write it in the form
\[ H'=H+t Q,\;\;Q \in B .\]
Note that we can consider $A$ as a {\em subring} of $B$, consisting of 
functions that do not depend on $t$, but also as {\em factor ring} of $B$: $A=B/tB$.\\

One of the main ideas of perturbation theory is that by applying systematically
appropriate automorphisms of the algebra $B$, one may try to bring the perturbed Hamiltonian $H+t Q$ 
to a simpler form or reduce it to a specific {\em normal form}\index{normal form}.

If we consider two $C^\infty$ symplectic manifolds $(M,\omega)$ and 
$(N,\omega')$, then any symplectic map $\p:M \to N$ produces a map of 
Poisson-algebras
$$\phi:=\p^*: C^\infty(N) \to C^\infty(M), f \mapsto f \circ \p, $$
which is a  {\em Poisson morphism\index{Poisson morphism}}. 
In general a Poisson morphism is defined as a map between two Poisson 
$R$-algebras $A$ and $B$ 
\[
\phi: A \to B\]
such that
\[
\begin{array}{c}
\phi(\l f)=\l \phi(f),\;\;\;\phi(f+g)=\phi(f)+\phi(g),\\
\phi(fg)=\phi(f)\phi(g),\;\;\phi(\{f,g\})=\{\phi(f),\phi(g)\}
\end{array}
\]
for all $\lambda \in R,\; f,g \in A$.

A {\em Poisson automorphism}\index{Poisson automorphism} is an invertible Poisson morphism 
$$\phi: A \to A,$$ and can be seen as a generalisation of the notion of 
symplectomorphism. 

A {\em Casimir element} \index{Casimir element} is an element $c$ that Poisson-commutes with all elements of $A$: 
$$\{c,a\}=0\;\;\; \textup{for all}\;\;\; a \in A .$$
We call a Poisson automorphism $\phi$ {\em central}\index{central Poisson automorphism}, if $\phi(c)=c$
for all Casimir elements.\\

The rings $A$ and $B$ introduced above however is too small for many of the constructions we wish to perform. We will encounter
several Poisson algebras similar to $A$ and $B$, that consist of certain power series rather than polynomials. 
These power series can be formal or convergent in some of the variables. We will introduce them in the sequel, but will usually call them $A$ and $B$. 

\begin{definition}
If $A$ is a Poisson algebra\index{Poisson algebra} over a field of characteristic $0$ and $S \in B:=A[[t]]$, 
then the series 
$$\varphi:=e^{t\{-,S\}}=\Id+t\{-,S\}+\frac{t^2}{2!}\{\{-,S\},S\}+\dots $$
is called  the {\em formal flow}\index{formal flow} of $S$.
\end{definition}

\begin{proposition}
The formal flow 
\[\varphi: B \to B\] 
is a central Poisson-automorphism.
\end{proposition}

\begin{proof} As the terms in the series of $\varphi$ consist of Poisson-brackets, it is clear that
$\varphi$ is central. 
One has 
\[
\begin{array}{rcl}
\varphi(f)\cdot \varphi(g)&=&(f+t\{f,S\}+\ldots)(g+t\{g,S\}+\ldots)\\
&=&f\cdot g+t(\{\{f,S\}g\}+f\{g,S\})+\ldots\\
&=&\{f,g\}+t\{f g,S\}+\ldots\\
&=&\varphi(f\cdot g),\\
\end{array}
\]
where the derivation property of the Poisson-bracket is used. Furthermore
\[
\begin{array}{rcl}
\{\varphi(f),\varphi(g)\}&=&\{f+t\{f,S\}+\ldots, g+t\{g,S\}+\ldots\}\\
&=&\{f,g\}+t(\{\{f,S\},g\}+\{f,\{g,S\}\})+\ldots\\
&=&\{f,g\}+t\{\{f,g\},S\}+\ldots\\
&=&\varphi(\{f,g\}),\\
\end{array}
\]
where one uses the Jacobi-identity.


\end{proof}


\section{The formal model}

In this section and in the next one, we denote by $K$ any field of characteristic zero and consider the rings
\[A:=K[q,q^{-1}][[p]], \;\;\;B:=A[[t]]=K[q,q^{-1}][[p,t]] \] 
of formal power series in $p=(p_1,\dots,p_n)$ and a central element $t$ whose coefficients are Laurent polynomials in $q=(q_1,\dots,q_n)$ with the Poisson bracket
\[\{ p_j,q_k \}= q_k \delta_{jk} .\]
Note that $K[[p]] \subset A$ and thus we can and will consider $A$ as a $K[[p]]$-algebra. 
There is also a natural {\em averaging map}\index{averaging} 
\[ av: A \to K[[p]]\]
of  'taking the average over the torus'. It is the $K[[p]]$-linear map that
maps each monomial $q^I$, $I \neq 0$ to $0$, whereas $av$ is the identity on 
$K[[p]] \subset A$.
  
The powers of the maximal ideal of $K[[t]]$ filter the ring $B$. We write 
$$f=g+(t)^n $$
if $f-g$ is of order $n$, i.e., it is a power series with monomials of degree not smaller than $n$ in the $t$ variable.
More generally, given an ideal we write $f=g+I $ if $f-g$ belongs to some ideal $I$. We also write $(p)$ for the ideal $(p_1,p_2,\ldots,p_n)$
generated by the $p_i$'s.

We will consider a special class of Hamiltonians. We have seen that a Hamiltonian written in action-angle variables $(q,p)$ depends only on the action variables $p$. Therefore we start with a Hamiltonian $H  \in K[[p]]$ without constant term:
\[H=\sum_{I \in \NM^n } \omega_I p^I= \sum_{i=1 }^n \omega_i p_i+(p)^2.\]
We will call such a Hamiltonian {\em integrable}\index{integrable Hamiltonian}, as clearly 
\[\{H,p_i\}=0, \;\;i=1,2,\ldots,n .\]
The vector  $\omega:=(\omega_1,\omega_2,\ldots,\omega_n)$ is called the
{\em frequency vector\index{Frequency vector of a Hamiltonian}} of $H$ (at $0$):
\[ \omega=(\frac{\partial H}{\partial p_1}(0), \frac{\partial H}{\partial p_2}(0), \ldots, \frac{\partial H}{\partial p_n}(0)) .\]

If $I \in \ZM^n$ is an integral vector, we write
\[ (\omega,I) :=\sum_{i=1}^n \omega_i I_i\]
for the euclidean scalar product of $\omega$ and $I$.

The Hamiltonian $H$ is called {\em non-resonant\index{resonant and non resonant Hamiltonians}}, if the $\omega_1,\omega_2,\ldots,\omega_n$ are $\ZM$-independent: 
$$  (\omega,I) \neq 0\;\;\;\textup{for all}\;\;\;I \in \ZM^n \setminus \{ 0 \} $$
Otherwise it is called {\em resonant} and in that case a non-zero vector $I \in \ZM^n$ such that
$$ (\omega,I) = 0$$
is called a {\em resonance} of $H$.
 
For instance,
$$H=p_1+a p_2 $$
is resonant if and only if $a \in \QM$.

\begin{proposition}\label{P::formal_PBH}
For a non-resonant $H \in K[[p]]$ the map
\[ \{H,-\}:A \to A\]
has $K[[p]]$ as kernel and the $K[[p]]$-module generated by the monomials $q^I$, $I \neq 0$ 
as image.
\end{proposition}																										
\begin{proof}
We can consider $A=K[q,q^{-1}][[p]]$ as a $K[[p]]$-algebra with the monomials $q^I$, $I \in \ZM^n$ 
as $K$-basis.
As $H \in K[[p]]$, the map $\{H  ,-\}:A \to A$ is $K[[p]]$-linear. Furthermore, due to
the Poisson-commutation rule
\[ \{p_k,q_j\}= q_j \delta_{kj},\]
the map $\{H  ,-\}$ is diagonal in the monomial basis and we can write
$$ \{H,q^I\}=((\omega,I)+(p))q^I, $$
where $(-,-)$ denotes the Euclidean scalar product.
As by assumption 
$$(\omega,I) \neq 0 $$
and all elements of the form
$$c+(p),\ c \neq 0,\ c \in K $$
are invertible in the local ring $K[[p]]$, it follows that the eigenvalue of $\{H,-\}$ associated to $q^I$ is invertible for $I \neq 0$.
As a consequence, the kernel of $\{H,-\}$ is $K[[p]]$ and furthermore the image contains all monomials $q^I$, $I \neq 0$.
\end{proof}

 So in the non-resonant case, the image of the map $\{H,-\}$ consists exactly of the series with average equal to zero. One can express the above proposition by saying
that for a non-resonant $H \in K[[p]]$ we have an exact sequence of the form
\[ 0 \lra K[[p]] \lra A \stackrel{\{H,-\}}{\lra} A \stackrel{av}{\lra} K[[p]] \lra 0,\]
which says that kernel and cokernel of $\{H,-\}$ are both isomorphic to $K[[p]]$.
\section{Formal stability}
The following proposition states that a formal deformation of an integrable Hamiltonian remains integrable in a formal neighborhood of a non-resonant torus.

\begin{proposition} 
\label{P::NF}
Let $H+tQ \in B$ be a deformation of an integrable Hamiltonian
$$H=\sum_{ I \in \NM^n } \omega_I p^I= \sum_{i=1 }^n \omega_i p_i+(p)^2.$$ 
If $H$ is non-resonant,  then there exists a  central Poisson automorphism 
\[\p : B \to B\]
 such that
$$\p(H+tQ) \in K[[p,t]]. $$
\end{proposition}

\begin{proof} Let us write $(t^n) \subset B$ for the terms that are divisible by $t^n$, so we can write
 $$H+tQ=H(p)+(t)$$
 We will construct inductively a sequence of central Poisson automorphisms $\p_0,\p_1,\p_2,\ldots$ of $B$ such that 
\[ \p_n(H+tQ)=P_n+(t^{n+1}) \]
and where
$$P_n \in K[[p,t]],\;P_n=P_{n-1} +(t^n)$$
We take $\p_0=Id$, $P_0=H$. 
Assume we have constructed the sequence up to $\p_n$. Then we look to the next order in $t$ and write:
$$\p_n(H)=P_n(p,t)+t^{n+1}Q_{n+1}(q,q^{-1},p)+(t^{n+2}) .$$

By transferring the part of $Q_{n+1}$ that does not contain $q$ to the $P_n$ part (``averaging''), 
we arrive at the form
$$\p_n(H+tQ)=\widetilde{P}_n(p,t)+t^{n+1}\widetilde{Q}_{n+1}(q,q^{-1},p)+(t^{n+2}), $$
where we can assume that $\widetilde{Q}_{n+1}$ does not contain pure $p$-monomials.
As $H$ is assumed to be non-resonant, we can find $S_{n+1} \in A=K[q,q^{-1}][[p]]$ 
such that
$$\{H,S_{n+1} \}+\widetilde{Q}_{n+1}(p,q,q^{-1}) \in K[[p]].  $$
As $K$ is assumed to be of characteristic zero, the map 
$$\psi_{n+1}=e^{t^{n+1}\{-,S_{n+1}\}}$$
is a well-defined Poisson automorphism of $B$. As
$$\psi_{n+1}(f)=f+t^{n+1}\{ f,S_{n+1}\}+\frac{t^{2n+2}}{2!}\{\{ f,S_{n+1}\},S_{n+1}\}+\dots.$$
the automorphism 
\[\p_{n+1}:=\psi_{n+1} \p_n\]
has the property that
$$\p_{n+1}(H+tQ)=P_{n+1}+(t^{n+2}). $$
This proves the statement by induction on $n$.  
\end{proof}

Given a mechanical system defined by an Hamiltonian of the 
form described by the proposition, the change of variables 
constructed above reduces the motion to a quasi-periodic one. 
In fact the proposition gives the classical method for performing 
computations in celestial mechanics. However, one should notice 
that this reduction to the normal form is done by means of formal 
power series and convergence issues are not considered. 
In practice, one performs this reduction only up to a certain order, 
and the solution one obtains by truncation of higher order terms
can only be expected to be asymptotic to the real solution.

\begin{example}
 Consider the Hamiltonian function $H=p \in \CM[q,q^{-1},p]$ 
  and its deformation
 $$H'=p+tp^2+tpq+tpq^{-1}.$$
Keeping in mind that $\{p,q\}=q, \{f(p),q\}=f'(p) q$ the equations
of motions are:
 \begin{align*}
  \dot q&=\{H',q \}=q+t(q^2+2pq+1),\\
  \dot p&=\{H',p \}=t(pq^{-1}-pq)
 \end{align*}
We put
$$S=pq-pq^{-1} $$
so that
$$p^2+\{S,H\}=p^2+pq+pq^{-1} $$
is the perturbation. The symplectomorphism $\p=e^{t\{-,S\}} $ eliminates the first order term in $H'$, but creates terms of higher order in $t$:
$$\p(H')=p+tp^2-2t^2(p^2q^{-1}+p^2q+p)+\ldots $$
Hence, neglecting the terms of order $2$ in $t$, we get the quasi periodic motion

\begin{align*}
  \dot q&=(1+2t p)q,\\
  \dot p&=0
 \end{align*}
 which can easily be integrated 
 \begin{align*}
  q(\tau)&=ae^{\tau+2 b \tau t},\\
  p(\tau)&=b
 \end{align*}
where $\tau$ denotes the time.
 \end{example}

Another difficulty which appears in celestial mechanics is the 
occurrence of {\em resonances}.\index{resonance} In the resonant case, 
there exist resonance vectors $I \neq 0$ with $(\omega, I)=0$ and 
in that case the corresponding {\em resonant monomial} $q^I$ \index{resonant monomial} is not in the image of the map
$$\{ H,-\}:A \to  A .$$
The above procedure then still gives a sequence of Poisson automorphisms, but it is stopped when we arrive at the first resonant monomial: if $H$ is of the form
$$H=P_n(p)+t^{n+1}Q_{n+1}(p,q,q^{-1})+(t^{n+2}) $$
and $Q_{n+1}(p,q,q^{-1})$ contains a resonant monomial, then the equation
$$\{H,S_{n+1} \}+Q_{n+1}(p,q,q^{-1}) \in K[[p]].  $$
cannot be solved, hence we see:

\begin{proposition} 
\label{P::res}
Let $$H=\sum_{I \in \NM^n } \omega_I p^I= \sum_{i=1 }^n \omega_i p_i+(t^2) $$ 
be such that the frequencies $ \omega_1,\dots,\omega_n$ are $\ZM$-dependent. 
There exists perturbations $H+tQ$ of $H$ which  cannot be reduced to a function of 
the $p_i$'s (and $t$) by a Poisson automorphism. 
\end{proposition}

The proposition is of course a very weak statement, but nevertheless, 
the above simple fact already shows a dichotomy between the different 
tori of an integrable system. For instance the Hamiltonian
\[ H=p_1-2p_2\]
is resonant, $I=(2,1)$ is a resonance vector and we have the resonant monomial $q_1^2q_2$. 
With the perturbation $p_1-2p_2+t q_1^2q_2$ we are in the situation of Proposition~\ref{P::res}: the resonant 
monomial $q_1^2q_2$ cannot be suppressed by a formal symplectomorphism. But if we make a scaling
\[ (p_1,p_2) \mapsto ((1+\lambda)p_1,p_2) \]

with $\lambda$ irrational, we are in the situation of Proposition~\ref{P::NF}, and the
monomial can be transformed away. So already at the formal level, the situation is subtle. This is one of 
the manifestations of formal KAM theory that we will explore more in detail in a later section. 

The beautiful non-integrability theorem of Poincar\'e goes far beyond this elementary remark. It states that over $\RM$ and $\CM$, the only first integrals are, in general,  the functions of the Hamiltonian itself. 

This non-integrability theorem can be understood in terms of global analysis using transversality arguments: 
we fix a time $T$ and consider the space $\Omega(M)$ of $T$-periodic loops in $M$. 
These $T$-periodic orbits are exactly the critical points of the action functional 
$$\Omega(M) \to \RM,\;\;\; \gamma \mapsto \int_{\gamma} pdq .$$
Using Smale's transversality theorem, one can show that for a general $H$, the critical points of 
this function are isolated, hence the periodic orbits with period $T$ are isolated.
If there were an independent first integral $F$, the flow of $F$ would transform each $T$-periodic 
orbit into a one-parameter family of $T$-periodic orbits, contradiction the fact that they are isolated. Therefore the flow of $F$ should be stationary at all periodic orbits and Poincar\'e's theorem reduces to showing that $F$ is then a function $H$.

\section{Analytic model  and holomorphic Fourier series}
Our analysis of the formal case reveals the inherent algebraic structures involved. 
But  the problem we really want to deal with,  is the analytic and not the formal case. We will describe a particular Poisson-algebra 
\[ \CM\{q,q^{-1},p\}\]
that is of importance for the Kolmogorov invariant torus theorem.
We start with some preparations.\\

Consider a $2\pi$-periodic function $f$ that is holomorphic on a 
neighbourhood of the real axis. It has an expansion into a (convergent)
Fourier series
\[ f(\theta)= \sum_{n\in \ZM} a_n e^{i n \theta } .\]
Due to the compactness of the interval $[0,2\pi]$, the function $f$ is 
in fact holomorphic on a strip
\[B_t:=\{ \theta=x+iy \in \CM : x \in \R,\ |y| < t \} .\]
There exist a well-known relation between the rate of vanishing of the
coefficients $a_n$ and the width $t$ of the strip.
 
\begin{proposition} The function $f$ extends holomorphically to $B_t$ if 
and only if
 \[ |a_n| = O(e^{-|n|s})\]
for any $s <t$.
\end{proposition}

\begin{proof}
Let $s < t$ and pick any $s' \in ]s,t[$. Consider the Hilbert space $L^2([0,2\pi] \times [-s',s'],\CM)$ with Hermitian scalar product
\[(f,g):=\frac{1}{2\pi}\int_{[0,2\pi] \times [-s',s']} f \cdot \overline{g}\ dx dy\;  .\]
The functions $e_n:=e^{in\theta}=e^{inx-ny}$ form an orthogonal set with 
\[ (e_n,e_n) =\frac{1}{2\pi} \int_{-s'}^{s'} \int_0^{2\pi} e^{inx-ny} e^{-inx-ny} dx dy = \frac{e^{2ns'}-e^{-2ns'}}{2n}\;. \]
From the expansion of $f$ as Fourier-series 
\[ f =\sum_{n \in \ZM} a_n e^{in\theta}\]
we get
\[ (f,e_n) =a_n (e_n,e_n) .\]
Hence we obtain
\[ |a_n| (e_n,e_n)=|(f,e_n)| \le \|f\| \|e_n\| \]
or 
\[ |a_n|  \le \frac{\|f\|}{\|e_n\|}\le C e^{-|n|s}\]
for an appropriate choice of $C$. Conversely, if the coefficients decrease 
exponentially, then the Fourier series is convergent inside some compact 
subset of the strip and thus defines a holomorphic function.
\end{proof}

Now if the Fourier series $f=\sum_{n \in \ZM} a_n e^{in\theta}$ is analytic in a strip $B_t$, the 
associated series
\[\sum_{n \in \ZM} a_n q^n  \]
is analytic in the annulus
\[ \{ q \in \CM\;\;|\;\; e^{-t} <|q| < e^{t} \}\] 
These annuli form a fundamental system of neighbourhoods of the circle
\[ S^1:=\{ q\in \CM\;\;|\;\;|q|=1\} .\]
\vskip0.3cm
\begin{figure}[htb!]
\includegraphics[width=9cm]{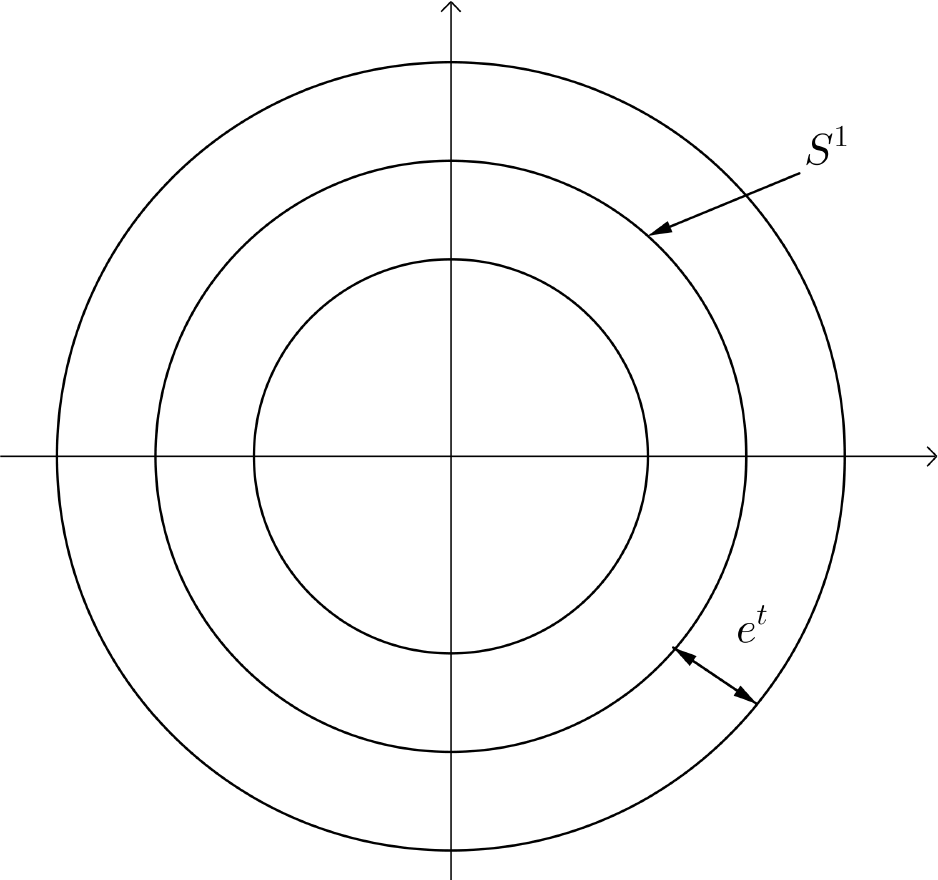}
\end{figure}
\vskip0.3cm
\begin{definition}
We put
\[\CM\{ q,q^{-1}\}:=\{ \sum_{n=-\infty}^{\infty} a_n q^n\;\;|\;\exists t>0 \;\;\forall n: |a_n| \le O(e^{-|n|t})\;\}\] 
and call it {\em the ring of holomorphic Fourier series}\index{ring of holomorphic Fourier series}.\index{holomorphic Fourier series}
\end{definition}
The above discussion shows that $\CM\{q,q^{-1}\}$ can be identified with algebra of
germs of holomorphic functions along the unit circle.\\


This notion extends to $n$ variables $q_1,\dots,q_n$ in an obvious way.
We will be concerned with the Poisson algebras
\[A:=\CM\{q,q^{-1},p\},\;\;\;B:=\CM\{q,q^{-1},p,t\}\]
with Poisson-bracket
\[ \{p_k,q_j\}=q_j\delta_{kj}\]
and $t$ a central element.
In terms of the cotangent space $T^*X$ to the algebraic torus $X:=(\CM^*)^n$, we
can say that the algebra $A$ consists of the germs of holomorphic functions on
$T^*X$ along the central $n$-torus 
\[ (S^1)^n=\{(q_1,q_2,\ldots,q_n) \in (\CM^*)^n\;\;|\;\;|q_j|=1\} \subset X . \]

\section{First order analytic obstruction}
In a previous section we saw that for a non-resonant $H \in K[[p]]$ the operator
\[ L=\{H,-\}:K[q,q^{-1}][[p]] \to K[q,q^{-1}][[p]] \]
is diagonal in the monomial basis with kernel and cokernel isomorphic to $K[[p]]$ and series with vanishing average 
\[ \{ \sum_{I \neq 0,J}a_I q^I p^J \} \]
as image. We showed that this fact implies that any perturbation of a non-resonant integrable Hamiltonian can formally be transformed to the integrable normal form.

As Poincar\'e already observed, an analogous statement does not hold at an analytic level. In fact, if
we consider again an analytic integrable Hamiltonian $H \in \CM\{p\}$, then the analogous operator
\[ L^{an}:=\{H,-\}:\CM\{q,q^{-1},p\} \to \CM\{q,q^{-1},p\} \]
is of great complexity. Let us give an example and consider the Hamiltonian
\[H=p_1+\sqrt{2}p_2+\frac{1}{2}p_2^2\]
The frequency vector 
$$\omega:=(\frac{\partial H}{\partial p_1},\frac{\partial H}{\partial p_2})=(1,\sqrt{2}+p_2) $$
depends linearly on $p_2$ and we find:
$$L^{an}(q_1^mq_2^{-n})=(m-n(\sqrt{2}+p_2))q_1^mq_2^{-n} $$
and so the preimage under $L^{an}$ 
\[ \frac{1}{m-n(\sqrt{2}+p_2)} q_1^m q_2^{-n}\]
of the monomial $q_1^m q_2^{-n}$ has a pole at
\[ p_2^0:=\frac{m}{n}-\sqrt{2}\]
which is near the origin if $m/n$ is close to $\sqrt{2}$. By taking any (non-polynomial) convergent power series with  
such monomials:
$$\sum_{m,n \geq 0} \a_{nm} q_1^mq_2^{-n} $$ 
we get an analytic series which is {\em not in the image of $L^{an}$}. So contrary to what happens in the formal case, the image of $L^{an}$ misses many non-trivial analytic functions and not only the the series with non-zero average, that is, series depending only on the $p$ variables. The image is therefore much 
smaller and  it is non-trivial to see if a series belongs to it.

Let us return to the general case and consider a Hamiltonian
\[H=\sum_{i=1}^n \omega_i p_i + (p)^2 \in \CM\{p\} .\]
and a given perturbation
\[ H+tQ, \;\;\;S \in B \]
We try to find, as in the formal case, a Poisson-automorphism of the form
$$\p: B \to B,\;\;\;\p(f)=e^{t\{f,S\}}=f+t\{f,S\}+(t^2)$$ 
that transforms $H+tQ$ to an element of $B$ that is independent of the $q_i$. 
So we find
$$\p(H+tQ)=H+t \{ H, S \}+tQ+(t^2), $$
that is, we need to find $ S \in B$ such that
$$ \{ H,S\}= - Q .$$
As we indicated above, this is very delicate condition.

However, this equation reduces drastically by restricting both sides to $p=0$: 
we get a new equation
 \[ \{ H,S \}_{\mid p=0}=-Q(q,p=0)  \]
 
where $Q(q,p=0) \in  \CM\{q,q^{-1}\}$. 
This equation reduces to
 $$ L_0( S )=\sum_{i=1}^n \omega_i q_i\d_{q_i} S=  -Q(q,p=0)=:g $$ 
where 
\[L_0=\{ \sum_{i=1}^n \omega_i p_i,-\}\]
is the linear part of the Hamiltonian derivation. 

We will consider the ring of holomorphic Fourier series as sitting inside the space of 
all {\em formal Fourier series}\index{formal fourier series}.
\[\CM\{ q,q^{-1}\} \subset \CM[[q,q^{-1}]]:=\{\sum_{I \in \ZM^n} a_I q^I\;\;|\;\;\; a_I \in \CM\}\,.\]
In general, two series in $\CM[[q,q^{-1}]]$ can not be multiplied in the usual way, so it is not
a ring, but only a vector space. But it will still be useful to to equip  $\CM[[q,q^{-1}]]$  with 
the {\em coefficient-wise product} $\star$:
\[ \sum_I a_I q^I \star \sum_I b_I q^I =\sum_I a_Ib_I q^I\]
called the {\em Hadamard} or {\em convolution product}.\index{convolution} \index{Hadamard product}

We see that the operator $L_0$ is equal to taking the Hadamard product with the function
 $$l=\sum_{I \in \ZM^n \setminus \{ 0 \} } (\omega,I) q^I. $$
\[ L_0(S)=l \star S =g\]
This equation is solved for $S$ as
$$S= h \star g $$
with
$$h:=\sum_{I \in \ZM^n \setminus \{ 0 \} } (\omega,I)^{-1} q^I ,$$
the Hadamard-inverse of $l$.

There is an obvious restriction: $ (\omega,I)$ should not be equal to zero for any $I \neq 0$, 
otherwise $h$ is not defined. This is exactly the non-resonance condition. 
But even if this condition is satisfied, it might happen  $h \star g$ is not an element of
the ring $\CM\{q,q^{-1}\}$ of holomorphic Fourier series.

We thus see that for homomorphic Fourier series there is an obstruction:
$ h\star f$ has to be holomorphic for any $f \in \CM\{q,q^{-1} \}$.

This first order obstruction of analytic nature can be easily identified:
 \begin{proposition}
 \label{P::Hadamard}
 The Hadamard product 
  $$h\star:\CM[[q,q^{-1}]] \to \CM[[q,q^{-1}]],\ f \mapsto h \star f $$
  with the series
 \[h =\sum _{I \in \ZM^n} a_I q^I \in \CM[[q,q^{-1}]]\]
maps the sub-algebra $\CM\{q,q^{-1}\}$ to itself if and only for any $s >0$  
there exists a constant $C=C(s)$ such that the Fourier coefficients of $h$ satisfy the 
estimate:
 \[ |a_I| \leq Ce^{|I|s} .\]
 \end{proposition}
\begin{proof}
Assume that the coefficient $a_I$ of $h$ satisfy the above estimate and take any
element 
$$f(q)=\sum_I b_I q^I \in \CM\{q,q^{-1}\}.$$ 
By definition of $\CM\{q,q^{-1}\}$, there exists $t$ such that
$$|b_I|=O(e^{-|I|t})  .$$ 
Take $s<t$, we get that
$$|a_Ib_I|=O(e^{-|I|(t-s)}) .$$
Thus $h \star f \in  \CM\{q,q^{-1}\}$. 
Conversely if $h \star$ preserves the sub-algebra $\CM\{q,q^{-1}\}$, then 
consider for any $s >0$ the holomorphic series
$$f(q)=\sum_I e^{-|I|s} q^I \in \CM\{q,q^{-1}\}.$$ 
We have $h \star f \in  \CM\{q,q^{-1}\}$, thus
$$ |a_Ie^{-|I|s}|=O(1).$$
Hence the coefficients of $h$ satisfy an estimate $|a_I| \le C e^{|I|s}$. 
\end{proof}

 
\section{Bibliographical notes}

It is difficult to trace the birth of perturbation theory, but it is clear that the book of Poincar\'e has been an essential step:\\

{\sc Poincar\'e, H.}, {\em Les M\'ethodes Nouvelles de la M\'ecanique C\'eleste I}, Gauthier Villars. Paris  (1892).

Poincar\'e discovered the non-integrability of the three body problem in:\\

 {\sc H. Poincar\'e}, {\em Sur le probl\`eme des trois corps et les \'equations de la dynamique}, {Acta mathematica} {\bf 13}(1), 3-270, (1890).\\

Fermi gave a generalisation of the Poincar\'e divergence theorem in~:\\

{\sc E. Fermi}, {\em Dimostrazione che in generale un sistema meccanico \`e quasi ergodico,} Il Nuovo
Cimento {\bf  25}, 267 -269, (1923).\\

See also~:\\

{\sc C.L. Siegel}, {\em On the integrals of canonical systems}, {Annals of Mathematics} {\bf  42} (3), 
806-822, (1941).\\

{\sc G. Benettin, G. Ferrari, L. Galgani L et  A. Giorgilli}, {\em An extension of the Poincar\'e-Fermi theorem on the nonexistence of invariant manifolds in nearly integrable Hamiltonian systems.} Il Nuovo Cimento B, {\bf 72}(2), 137 - 148, (1982).\\

 
\chapter{The Kolmogorov invariant torus theorem}
In this chapter we take a closer look at the the non-resonance condition and
introduce the so called {\em Diophantine condition}, and show that the set
of Diophantine frequency vectors has full measure. Then we formulate 
Kolmogorov's theorem on the existence of invariant tori and reinterpret the 
theorem in terms of infinite dimensional group actions.
 
\section{Kolmogorov's  Diophantine condition}
We have seen that in the context of perturbation theory we are given a frequency 
vector
\[ \omega=(\omega_1,\omega_2,\ldots,\omega_n)\]
and we have to consider the Euclidean scalar product
\[(\omega,I)\]
with lattice vectors $I \in \ZM^n$. In the non-resonant case, the hyperplane
$\omega^{\perp}$ orthogonal to $\omega$ intersects the lattice $\ZM^n$ only at 
the origin. But of course, there will always be be lattice points that are
very close to the hyperplane. 
As a result, the series
\[ h= \sum_I (\omega, I)^{-1} q^I\]
may have coefficients that grow very fast if $|I|$ becomes big: 
{\em the problem of small denominators}.
We saw that the Hadamard product with a formal Fourier series 
\[ h=\sum a_I q^I\]
will preserve the space of analytic Fourier series if the coefficients $a_I$   
grow not faster than exponentially with $|I|$.  

Rather than studying types of sub-exponential growth, we consider the simpler case 
of polynomial growth:

\begin{definition} A vector $\omega=(\omega_1,\dots,\omega_n)$ satisfies  
Kolmogorov's Diophantine condition $K(C,\nu)$ if for all $0 \neq I \in \ZM^n$
$$  |(\omega,I)|\geq \frac{C}{\| I \|^{n-1+\nu}}.$$
Here $C>0$ and  $\nu \in \RM$. (The shift by $n-1$ is conventional).
We say that $\omega$ satisfies {\em Kolmogorov's condition}\index{Kolmogorov's Diophantine condition} if it
satisfies $K(C,\nu)$ for some $C$ and $\nu$.
\end{definition}

The condition $K(C,\nu)$ means that the growth is bounded by a 
polynomial function:
$$ |(\omega,I)|^{-1}\leq \frac{\| I \|^{n-1+\nu}}{C}. $$

In particular, Proposition~\ref{P::Hadamard} implies that if $\omega$ satisfies
Kolmogorov's Diophantine condition, then the endomorphism
$$h\star:\CM[[q,q^{-1}]] \to \CM[[q,q^{-1}]],\ f \mapsto h \star f,\ h:=\sum_{I \neq 0} (\omega,I)^{-1}q^I $$
preserves the sub-space $\CM\{q,q^{-1}\}$ of holomorphic Fourier series:
the first analytic obstruction vanishes.\\
\section{Diophantine approximation}
The question how small $|(\omega,I)|$ can be belongs to the field of {\em Diophantine approximation} \index{Diophantine approximantion}.
Let us first look at the case $n=2$ and $\omega=(1,\alpha)$,  $I=(p,-q)$. Then we have
\[ |(\omega,I)|=|q\alpha-p|,\]
so this becomes small if the rational number $\frac{p}{q}$ is a good approximation to $\alpha$.
If we subdivide the unit interval in $N+1$ sub-intervals of length $1/(N+1)$,
then the fractional part of the $N$ numbers
\[ \alpha, 2\alpha,3\alpha,\ldots,N\alpha\]
all fall in different intervals if $1/(N+1) < \a$. As a consequence, at least one falls in the first or last interval, which means that one of these numbers differs by less than $\le 1/(N+1)$ from an integer.\footnote{This is Dirichlet's {\em pigeon hole principle.}}
So for all $\alpha$ there exist
infinitely many integers $p,q$ so that
\[ |q\alpha-p| < \frac{1}{q},\;\;|\alpha-\frac{p}{q}| <\frac{1}{q^2} .\]
In fact, such good rational approximations can be obtained from the continued fraction expansion of $\alpha$.
In higher dimension, one can show similarly that for any $\omega \in \RM^n$, 
there are always approximations such 
$$ \forall C>0, \exists I \in \ZM^n,\  |(\omega,I)|\leq \frac{C}{\| I \|^{n-1}} .$$

In the theory of Diophantine approximation, one usually says that the vector admits {\em very good rational approximations} if one can find an exponent bigger than $n-1$. Kolmogorov's Diophantine condition goes in the opposite direction.

It is easy to construct vectors which do not satisfy Kolmogorov's Diophantine condition. 
Take: 
\[\alpha=\sum_{j \geq 0}10^{-j!} .\]
The rational numbers:
$$ r_k=\sum_{j = 0}^k10^{-j!} $$ 
satisfy:
$$| \alpha-r_k| \leq 2\cdot 10^{-(k+1)!}. $$
The vector
$$\omega=(1,\a) $$
admits very good rational approximations. To see it put:
$$\b_k=(\sum_{j = 0}^k10^{k!-j!},10^{k!}) \in \ZM^2.$$
Then
$$(\omega,\b_k)=10^{k!}\sum_{j  \geq  k+1}10^{-j!} =O(10^{k!-(k+1)!})$$
and
$$\| \b_k \| =O(10^{k!}) $$
therefore, for any $\nu>0$, the sequence
$$(\omega,\b_k)\| \b_k\|^{1+\nu}=O(10^{(\nu-k)k!}) $$
converges to zero.

Thus the vector $\omega$ does not satisfy Kolmogorov's Diophantine condition.
On the other hand, a similar computation shows that for any $I \in \ZM^n$ with
$$10^{k!} \leq |I | < 10^{(k+1)!} $$
we have
$$\left| (\omega,I ) \right| \geq 10^{-kk!}. $$
Therefore
$$\left| (\omega,I)^{-1}\right| \leq  10^{-kk!} \leq e^{10^{k!}} \leq e^{|I|}.$$
In particular, Proposition~\ref{P::Hadamard} implies that although $\omega$ does not satisfies
Kolmogorov's Diophantine condition, the endomorphism
$$h\star:\CM[[q,q^{-1}]] \to \CM[[q,q^{-1}]],\ f \mapsto h \star f,\ h:=\sum_{I \neq 0} (\omega,I)^{-1}q^I $$
still preserves the space of holomorphic Fourier series.\\

\section{Vectors satisfying the Kolmogorov Diophantine condition}
Kolmogorov's Diophantine condition is not a necessary condition for having an invariant torus theorem, but is sufficiently weak to allow the vectors satisfying it to form a set of positive measure. The condition also has also an abstract meaning in terms of operators as we 
shall see later. For the moment, we continue to investigate how many vectors satisfy the condition.\\

According to the well-known theorem of Liouville, for any algebraic non-rational 
number 
\[\alpha \in \overline{\QM} \setminus \QM\]
of degree $d$ then for some $C >0$ one has
\[ |\alpha-\frac{p}{q}| \ge \frac{C}{q^d}\, ,\]
so the vector $(1,\alpha)$  satisfies Kolmogorov's Diophantine condition. 
Therefore the set of vectors which satisfy the condition obviously form a dense 
subset. But in fact one can prove that this set even has full measure.
Let us denote by
\[ \Omega(C,\nu):=\{ \omega \in \RM^n: \forall 0 \neq I \in \ZM^n,\ |(\omega,I)|\geq \frac{C}{\| I \|^{n-1+\nu}} \}  \]
the set of vectors that satisfy Kolmogorov's condition $K(C,\nu)$.

\begin{proposition} Assume that $\nu>0$. Then for all $R, C >0$ there is 
a constant $k:=k(\nu,R)$ such that
\[   Vol(B_R \setminus \Omega(C,\nu)) \le k \cdot C \]
Here $Vol$ is the Lebesgue measure and $B_R$ the ball of radius $R$.
\end{proposition}
\begin{proof}

 \vskip0.3cm  \begin{figure}[htb!]
  \includegraphics[width=10cm]{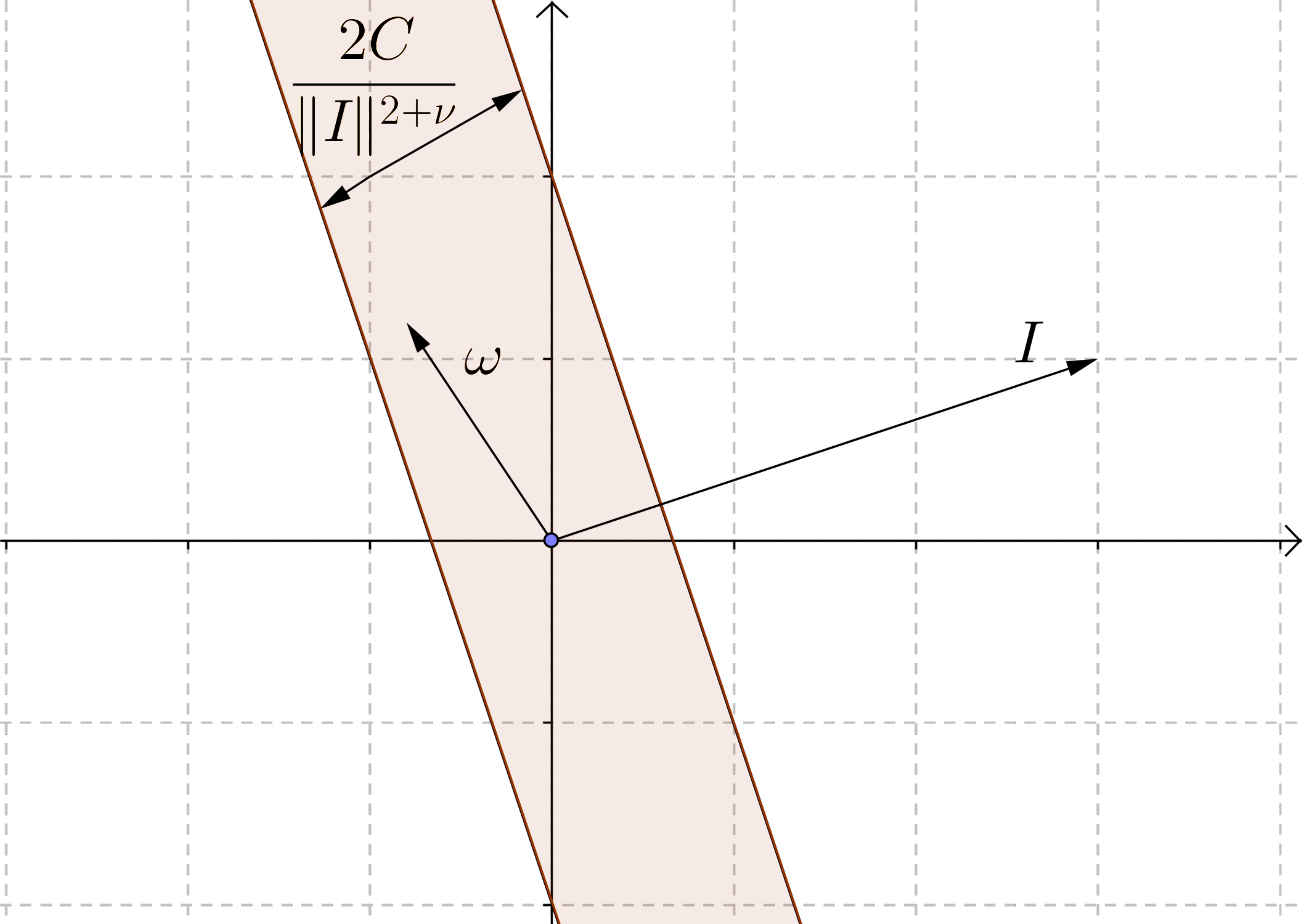}
\end{figure} \vskip0.3cm  

We consider a slightly more general situation.\\
 Let 
 $$f:\ZM^n \to \RM$$
 be  a function that only depends on the Euclidean length $\|I\|$ of $I \in \ZM^n$.  Consider the set
\[D(C):=\{ \omega \in B_R\;|\;\;\; \forall 0 \neq I \in \ZM^n:\;\;\; |(\omega,I)| \ge  C f(I)\} .\]
The complement of this set in $B_R$ can be written as
\[D(C)^c=\bigcup_{0 \neq I \in \ZM^n} B(C,I)\]
where $B(C,I)$ is the set of $\omega \in B_R$ that are 'bad' for $C$ and $I$:
\[ B(C,I):=\{ \omega \in B_R\;|\; |(\omega,I)| < C f(I)\}\]
These $\omega$'s lie in a  band of width
\[ 2 Cf(I)/\|I\|\]
around the intersection of hyperplane $I^{\perp}$ with the ball $B_R$.
So the volume of $B(C,I)$ can be bounded by
\[ Vol(B(C,I)) \le  2 V_{n-1}(R).C.f(I)/\|I\|,\]
where $V_{n-1}(R)$ is the volume of the $n-1$-dimensional radius $R$-ball.
Now if the lattice sum 
\[ \sum_{0 \neq I \in \ZM^n} f(I)/\|I\|\]  
exists and is finite, then the volume of $\cup_{I \in \ZM^n} B(C,I)$ is bounded by
\[Vol\left(\bigcup_{0 \neq I \in \ZM^n} B(C,I)\right) \le k.C, \;\;\;k=2 V_{n-1}(R) \sum_{0 \neq I \in \ZM^n} f(I)/\|I\| .\]
Now taking the intersection over $C>0$ gives
\[Vol(D^c) \le \lim_{C\to 0} K C =0 .\]
Applying this to 
\[ f(I):=1/\|I\|^{n-1+\nu}\]
gives the result, as for any $\nu >0$ one has
\[\sum_I 1/\|I\|^{n+\nu} < \infty .\]

\end{proof}

\begin{corollary}
For $\nu >0$ the complement of the set \[\Omega(\nu)=\bigcup_{C>0} \Omega(C,\nu)\]
has measure zero.
\end{corollary}

\begin{proof} That complement $\Omega(\nu)$ in the ball $B_R$ of radius $R$
is the set $D(R,\nu):=\bigcap_{C>0} B_R \setminus \Omega(C,\nu)$. According to the 
proposition, the volume of $B_R \setminus \Omega(C,\nu)$ goes to zero
linearly with $C$, hence the measure of $D(R,\nu)$ is zero. By choosing 
a sequence of radii going to infinity, we see that the complement
of $\Omega(\nu)$ is the union of countably many sets of measure zero,
hence has itself measure zero.
\end{proof}

Similar considerations hold, of course, over the field of complex numbers. 
Note that for $\nu \leq 0$, the set $\Omega({\nu})$ is empty. 

\section{Kolmogorov's invariant torus theorem: Statement}
We now come to the formulation of the first important result in KAM theory, namely the theorem of Kolmogorov concerning the stability of invariant tori under certain conditions. The proof of the theorem will be given in chapter 11, after we have all technical machinery in place.
  
\begin{theorem} Consider the Poisson algebras $A=\CM\{ q,q^{-1},p\}$ and $B=\CM\{q,q^{-1},p,t\}$. 
Let $I=(p_1,p_2,\ldots,p_n) \subset A$ the ideal generated by the $p_i$'s. 
Consider an element $H \in A$ of the form
   $$H=\sum_{i=1}^n \omega_i p_i+\sum_{i=1}^n \a_{ij} p_ip_j\ +(p)^3 $$
If\\

(D) the vector $(\omega_i)$  satisfies the Kolmogorov Diophantine condition.\\

(N) the matrix $(\a_{ij})$  is invertible.\\

then the pair $(H,I)$ is {\em homotopically stable}\index{homotopic stability}.
 \end{theorem}

Here by homotopic stability we mean the follwing: for any deformation  
\[H+tQ \in \CM\{q,q^{-1},p,t\}\]
of $H$, there exists a central Poisson automorphism $\p$ of $ \CM\{q,q^{-1},p,t\}$, series
$c(t) \in t\CM\{t\}$ and $a_{ij} \in tB$ 
such that:
\[ \p(H+tQ)=H +c(t)+\sum a_{ij} p_i p_j . \]
 
Condition (N) is called {\em Kolmogorov's non-degeneracy condition.}\index{Kolmogorov's non-degeneracy condition} 

As each hamiltonian of the form at the right hand side has $p_1=p_2=\ldots=p_n=0$ as invariant torus, we see that under the assumptions of the theorem, any hamiltonian $H+tQ$ admits a family of invariant Lagrangian manifolds parametrised by $t$, for $t$ small enough. As we shall see, the theorem also holds in  the real analytic context, because all our constructions commute with the conjugation
$q_i \mapsto 1/\overline{q_i}$, the real part of the invariant manifolds are 
tori. 

\section{The one dimensional case}
Although the case $n=1$ of the Kolmogorov theorem is rather trivial, it already
shows the origin of the non-degeneracy condition and its relation to symplectic
but non-hamiltonian vector fields.

Let us take a closer look at the Hamiltonian
\[H(q,p)=\omega p +\a p^2 .\]
The conditions of the theorem become:\\

(D) $\omega \neq 0$,\\

(N) $\a \neq 0$.\\

The theorem produces for any perturbation $H+tQ$ a Poisson morphism $\varphi$ and an element $c(t) \in \CM\{ t \}$ such that
  $$\p(H+tQ)=c(t)+\omega p+(p)^2 .$$

As the Hamiltonian flow preserves the level sets of $H$, we know that, in a neighbourhood of the zero section, the motion will take place along the curves
$$H=\text{constant} .$$
These curves are diffeomorphic to circles, that is one-dimensional tori.
Kolmogorov's theorem tells us that for each $t$, we may select a circle on which the period of the motion is {\em precisely $\omega$}, as is the case for $H_0$. If we omit the non-degeneracy condition, then this is obviously wrong: take for instance
$$H=(\omega+t) p .$$
For fixed value of $t$, the motion along the circles $p=\text{constant}$ has frequency $\omega+t$, so it cannot be equal to $\omega$ unless $t=0$.

Let us consider the specific deformation
$$H'=\omega p + \a p^2+t p  $$
of $H$ and let us compute the function $c$ and the Poisson automorphism $\p$ in this case.

First if we consider a Poisson automorphism induced by a Hamiltonian vector field
$$\p=e^{t\{-,S\}} $$
then the effect is
$$\p(H')=\omega p + \a p^2+t p+t\{\omega p+\a p^2,S\}+(t^2) .$$
As $p$ is not in the image of
$$S \mapsto \{\omega p+\a p^2,S\}$$
we must use a {\em non-hamiltonian vector field} to get rid of the deformation:
$$v=d(t)\d_{p} .$$
We have
$$e^{tv}H'=\omega p + \a p^2+t p+d(t)(\omega+2\a p)+(t^2). $$
Therefore we take
$$d(t)=-\frac{t}{2\a}. $$
The associated Poisson automorphism $\p$ maps $H$ to
$$\p(H')=\omega p +\a p^2-\frac{\omega}{\a} t+\frac{1}{4 \a} t^2. $$
In general, we will have an infinite series, but here the process stops at the first step.

The circle 
 $$\Ct_t:=\{ p+\frac{1}{2\a} t=0 \}$$ is mapped to 
 $$\Ct_0:=\{ p=0\}$$  
 The frequency of $H'$ along $\Ct_t$ and $H$ along $\Ct_0$ are equal:
 $$\d_p H'_{|\Ct_t}=\omega+p+\frac{1}{2\a} t=\omega=\d_p (H)_{|\Ct_0}. $$

This trivial example shows how the non-degeneracy condition works: it enables us to choose among the invariant circles the one for which the frequency is precisely $\omega$. If $\a$ is equal to $0$, such a choice is no longer possible.

 \section{Kolmogorov's theorem and group actions} 

We will now give a slight reformulation Kolmogorov's invariant torus theorem as a 
statement about a certain group action\index{group action} in the infinite dimensional vector space 
 $$E:=\CM\{ q,q^{-1},p,t \}$$ of analytic Fourier series depending on a parameter $t$. 
Clearly, the group of central Poisson automorphisms acts on $E$, but we we will consider the subgroup
$G$ of elements $\varphi$ which are {\em tangent to the identity}, i.e., whose restriction to $t=0$ 
is the {\em identity} on $\CM\{q,q^{-1},p\}$.
When we fix an element $H \in E$, the group $G$ acts on the affine space 
$$H+M \subset E ,$$
of {\em perturbations} of $H$, where $M:=tE \subset E$ consist of all the elements $tQ$ we can
add to $H$.

We will take $H$ to be of the special form 
$$H=\sum_{i=1}^n \omega_i p_i+\sum_{i=1}^n \a_{ij} p_ip_j ,$$
where the frequency $\omega=(\omega_i)$ and $(\a_{ij})$ are fixed.
If we add to $H$ a special perturbation from the linear space
$$F=t(\CM\{ t \}\oplus I^2) \subset M, $$ 
we obtain a Hamiltonian of the form
\[ H+tc(t)+ \sum_{ij} a_{ij}p_ip_j\]
The elements of this affine space $H+ F \subset H+M$ 
have clearly the special property that their flow preserves the subset 
$p_1=p_2=\ldots=p_n=0$ that defines the torus.\\

Kolmogorov's theorem can be formulated as follows: 

\begin{theorem}   If $(\omega_i)$ is Diophantine and $(\a_{ij})$ is invertible, 
then any element of $ H + M $ lies in the $G$-orbit of an element in $F$. In other words, the map
\[G \times F \to H+M,\ (\p,\a) \mapsto \p(H+\a) \]
is surjective. 
\end{theorem}

 \vskip0.3cm  \begin{figure}[htb!]
  \includegraphics[width=10cm]{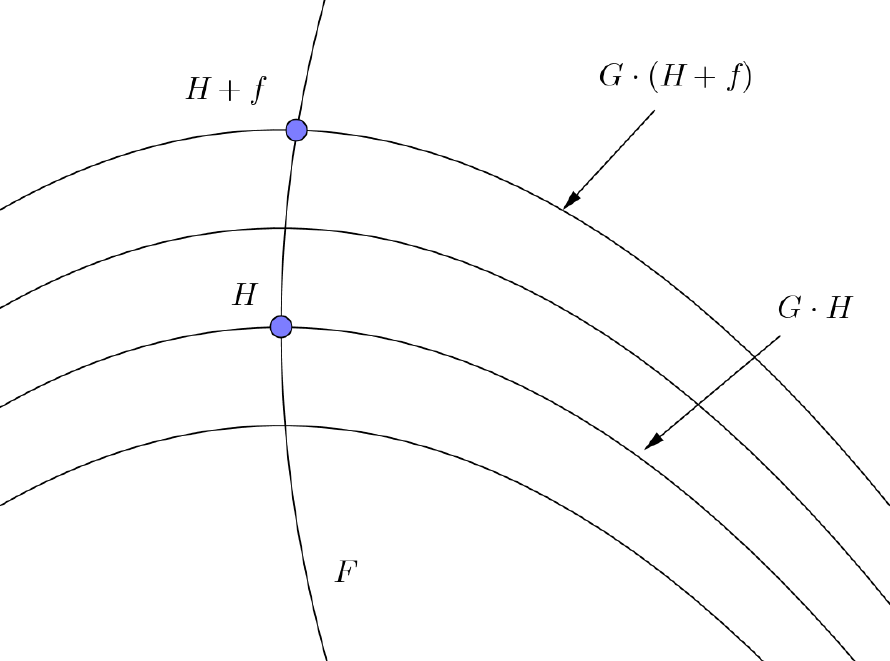}
\end{figure} \vskip0.3cm  
It means that any perturbation of $H$ can be transformed into a normal form \index{normal form} 
belonging to $F$. As all elements of the normal form posses an invariant torus, any perturbation 
of $H$ admits an invariant torus, obtained from the torus $p_1=p_2=\ldots=0$ in the normal form via 
the automorphism $\varphi \in G$. \\

The Kolmogorov theorem gives an answer to a very general question in a particular case: given a group action on a space $M$, how can we ensure that a subset $F \subset M$ is a local transversal to the action?\\

When such a property holds we say that the corresponding element in $F$ is a {\em normal form}\index{normal form} and $H+F$ is called a {\em transversal to the action}. Such a choice for $F$ is usually not unique and correspondingy there 
are different normal forms. The approporiate choice of $F$ is determined by 
utilitarian considerations.\\

In the next chapters of this book we will develop this point of view and define a general iteration scheme to bring elements to normal form. Furthermore, we will formulate a general statement of convergence that implies the above theorem as a special case.

\section{Bibliographical notes}
 
Kolmogorov's original paper on the invariant torus theorem:\\

{\sc A.N. Kolmogorov}, {\em On the conservation of quasi-periodic motions for a small perturbation of the Hamiltonian function}, Dokl. Akad. Nauk SSSR {\bf 98}, 527 - 530, (1954).\\

Note that the paper was not even translated in the english version of the journal.   Arnold gave a first proof in:\\

{\sc V.I. Arnold}, {\em Proof of a theorem of A. N. Kolmogorov on the preservation of conditionally periodic motions under a small perturbation of the Hamiltonian,} Uspehi Mat. Nauk {\bf 18}(5), 13 - 40, (1963).\\

Arnold studied the stability of the three body problem in:\\

{\sc V.I. Arnold},  {\em\ Small denominators and problems of stability of motion in classical and celestial mechanics}, Uspehi Mat. Nauk {\bf 18}(6), 91-192, (1963).\\


Moser proposed the general scheme of group actions in:\\

{\sc J. Moser}, {\em A rapidly convergent iteration method and non-linear partial differential equations II}, {Ann. Scuola Norm Sup. Pisa - Classe di Scienze S\'er. 3},{\bf 20}(3), {499-535}, (1966).\\
 
For details on classical  KAM theory, we refer to the following papers and references therein:\\

{\sc V.I. Arnold, V.V. Kozlov, A.I.  Neishtadt}
{\em Mathematical aspects of classical and celestial mechanics}, Encyclopedia of Mathematical Sciences {\bf 3}, 1-291, Springer Verlag, (1985)\\

{\sc J.-B. Bost}, {\em Tores invariants des syst\`emes dynamiques hamiltoniens}, S\'eminaire Bourbaki {\bf 27} Expos\'e No. 639, (1984-1985).\\

{\sc J. F\'ejoz},\ {\em Introduction to KAM theory,} 19 pp. Lectures given at the Jornadas on interactions between dynamical systems and partial differential equations, Barcelona (2013) and the Ciclo di Lezioni at the Universit\'e di Bicocca, Milano (2011).\\

{\sc M. Sevryuk}, {\em The classical KAM theory at the dawn of the twenty-first century.,} Moscow Mathematical Journal, 3(3), 1113-1144, (2003)\\

Arnold pointed out that the set of vectors which satisfy Kolmogorov's Diophantine condition and which lie on a submanifold might have zero measure. There are many results which prevent such pathologies, see for instance:\\

{\sc A.S. Pyartli}, {\em Diophantine approximations on submanifolds of Euclidean space}, Functional Analysis and Its Applications {\bf 3}(4), 303-306, (1969).\\

{\sc D.Y. Kleinbock et G.A. Margulis}, {\em Flows on homogeneous spaces and Diophantine approximation on manifolds}, Annals of Mathematics {\bf 148} 339 - 360, (1998).\\

See also:\\
{\sc M. Garay},  {\em Arithmetic density}, Proceedings of the Edinburgh Mathematical Society, Vol. 59,  691-700, 2016.\\

\chapter{The generalised KAM problem}
The original theorem of Kolmogorov pertains to the very special problem of perturbations of  
quasi-period motion on tori in a Hamiltonian system given in action-angle variables. It is of 
considerable interest to look for a more intrinsic and coordinate independent understanding of 
the theorem. This is indeed possible and a first step consists of replacing tori to more general 
Lagrangian subvarieties or coisotropic subvarieties that are invariant under the flow of a 
Hamiltonian $H$. This can be done in a general context of Poisson-algebras.

\section{Invariant ideals}
In this section we start with the standard symplectic space $M=K^{2n}$, with canonical coordinates 
$q,p$ and we assume that the field $K$ is algebraically closed of characteristic $0$, e.g. $K=\CM$.  
The ring of polynomial functions $A:=K[q,p]$ is a Poisson-algebra with the
standard Poisson-bracket
\[ \{ F,G \}=\sum_{i=1}^n\d_{p_i} F \d_{q_i}G-\d_{q_i} F \d_{p_i}G.  \]

To a collection of polynomials 
$$f_1,\dots,f_k \in K[q,p]  $$
one associates the {\em affine variety:} \index{affine variety}
$$V:=\{(p,q) \in M:f_1(q,p)=\dots=f_k(q,p)=0 \} \subset M. $$
This variety in fact only depends on the ideal $I=(f_1,\dots,f_k) \subset A$ generated by the 
polynomials $f_i$, hence the notation $V=V(I)$. More generally, 
given a ring $A$ and an ideal $I$ there is an associated subscheme 
$V(I) \subset \Spec(A)$.

Recall that the {\em radical}\index{radical} $\sqrt{I}$ of an ideal $I$ consists of 
those elements of the ring for which a power belongs to the ideal $I$. Hilbert's {\em Nullstellensatz} is the 
statement that $\sqrt{I}$ is equal to the ideal of all polynomials that vanish on $V(I)$. 
A {\em radical ideal} \index{radical ideal} is an ideal such that $\sqrt{I}=I$. So radical ideals $I$ can alternatively be characterised by the the property that any polynomial which vanishes on $V(I)$ belongs to $I$. 

If $G$ is a first integral of $H$,  then the level sets of $G$ are preserved by the Hamiltonian flow of $H$. 
One idea of Kolmogorov was to search for {\em lower dimension invariant manifolds}.\index{invariant manifold} 

\begin{proposition}
Let $I \subset A$ be a radical ideal. The following assertions are equivalent~:
\begin{enumerate}[{\rm i)}]
\item $\{ H,I \} \subset I$
\item $V(I)$ is invariant under the formal flow of $H$.
\end{enumerate}
If this applies, we say that $I$ is $H$-invariant.\index{invariant ideal}
\end{proposition}
\begin{proof} 
We denote by $\p$ the formal flow of the Hamiltonian $H$, which was defined
as the automorphism of $B:=A[[t]]$ defined by the series  $e^{t\{ -,H\}}$:
\[\p(f)=f+t\{f,H\}+\frac{t^2}{2!}\{\{f,H\},H\}+\ldots\]
The ideal $I \subset A$ generates an ideal $IB \subset B$ consisting of series
\[ h_0+h_1t+h_2t^2+\ldots\]
where the $h_i \in I$.
The invariance of $V(I)$ under $\p$ means, by definition: 
\[ \p(\sqrt{IB}) \subset \sqrt{IB}\]
But if $I$ is radical, $IB$ is radical too, $\sqrt{IB}=IB$. To show this, it is
sufficient to show that if $f^2 \in IB$ then in fact $f \in IB$. When we write
$f$ as a series
\[ f=f_0+f_1t+f_2t^2+\ldots,\;\;f_i \in A\]
then $f^2 \in IB$ leads to equations
\[f_0^2 \in I,\;\;2 f_0f_1 \in I,\;\;2f_0f_2+f_1^2 \in I, \ldots\]
If $I$ is radical, the first equation gives $f_0 \in I$, the second gives no
information on $f_1$, but from the third we deduce $f_1^2 \in I$, so $f_1 \in I$. Continuing this way, we get $f_i \in I$ for all $i=0,1,2,\ldots$, that is $f \in IB$, hence $IB$ is radical.
 
It follows that  $V(I)$ is $\p$-invariant if and only if 
\[ f \in I \Longrightarrow \p(f) \in IB\]
By looking at the coefficient of $t$ at the right hand side we see that
$\p$-invariance of $V(I)$ implies that $\{I,H\} \subset I$. On the
other hand, if $\{I,H\} \subset I$, then for $f \in I$, all terms 
\[f,\;\;\{f,H\},\;\;\{\{f,H\},H\},\;\;\ldots\]  
of the series for $\p(f)$ belong to $I$, so $\p(f) \in IB$. Hence
the conditions $i)$ and $ii)$ are equivalent.
\end{proof}

The $K$-algebra 
\[ K[q,p]/I\]
can be interpreted as the ring of polynomial functions on the variety $V(I)$ defined
by the ideal $I$. It is usually called the {\em affine coordinate ring}\index{affine coordinate ring} 
of $V(I)$.

The invariance condition $\{ H,I \} \subset I$  implies that the map
$$ K[q,p]/I \to K[q,p]/I,\ f \mapsto \{ H,f \}  $$
is well-defined. It is a {\em derivation} of the coordinate ring $K[q,p]/I$ and 
represents the restriction of the Hamiltonian field $X_H$ to $V(I)$. We use the same
name for this induced map and consider $\{H,-\} \in Der(K[q,p]/I).$\\

Recall that an ideal $I \subset K[q,p]$ is called {\em involutive}\index{involutive}\index{co-isotropic} (or {\em co-isotropic}) if
$$\{ I, I\} \subset I.  $$
If a radical ideal $I$ is involutive and $V(I)$ is non-empty, then $\dim V(I) \ge n$, 
and if $ V(I)$ is of pure dimension $n$, then $V(I)$ is a Lagrangian subvariety, possibly with singularities.

\begin{proposition} For any involutive ideal $I \subset K[q,p]$ there is a well-defined map
$$I/I^2 \to \Der(K[q,p ]/I),\ H \mapsto \{H,- \}. $$
\end{proposition}

\begin{proof}
As $\{I,I\} \subset I$, the ideal $I$ is $H$-invariant for any $H \in I$, so
for each $H \in I$ we have a well-defined derivation
\[\{H,-\}:K[q,p]/I \to K[q,p]/I\] 
We obtain a map
\[ I \to Der(K[q,p]/I),\;H \mapsto \{H,-\}\]

Let $(f_1,\dots,f_n)$ be generators of $I$ and $H \in I$. Put
  $$H'=H+\sum_{ij}a_{ij}f_if_j $$
We have
$$
\begin{array}{rcl}
\{ H+\sum_{ij}a_{ij}f_if_j ,g \} &=&\{ H ,g \}+  \sum_{ij}a_{ij}f_i \{f_j ,g \} +\sum_{ij}a_{ij}f_j\{f_i ,g \}\\  
&=&\{H,g\}\;\; \mod I
\end{array}
$$
So the polynomials $H$ and $H'$ define the {\em same derivation of the coordinate ring
$K[q,p]/I$}. 
This shows that the above map factors over $I/I^2$ and we get an induced map
$$I/I^2 \to \Der(K[q,p]/I),\ H \mapsto \{H,- \}.  $$
\end{proof} 

In algebraic geometry, the module $I/I^2$ is called the {\em conormal module}\index{conormal module}
 because when $V(I)$ is smooth, it is dual to the space of sections of the normal bundle. Geometrically, the above proposition states that the conormal bundle to an involutive manifold  maps naturally to its tangent space. As we shall see, this elementary algebraic statement turns out to be quite fundamental in KAM theory, as it shows that for the study of the flow of $H$ along $V(I)$ the important 
object is not the Hamiltonian $H$ itself, but rather its {\em class modulo $I^2$}.

There is a second point: the Hamiltonian flow depends on the {\em differential} of the Hamiltonian and not on the Hamiltonian itself. Adding a constant does not change the dynamics, and
therefore the important object is {\em the class of the Hamiltonian modulo the subspace $I^2 \oplus K$}.\\

Also, it is clear that the above algebraic set-up can be applied in a more general context: if $I$ is an involutive ideal in an arbitrary Poisson-algebra $A$, we obtain in precisely the same way a well defined map
\[ I/I^2 \to Der(A/I),\;\;H \mapsto \{H,-\}\] 
In the general situation the role of the constants is taken over by the {\em Casimir elements}\index{Casimir element}, i.e. the elements Poisson-commuting with everything:
\[ H^0(A):=\{ f \in A\;|\;\{f,g\}=0\;\;\textup{\ for all\ } g \in A\}.\]

\begin{definition} We say that two Hamiltonians $H,H' \in A$ are
$I^2$-equivalent, if 
\[ H'=H\;\; \mod H^0(A) +I^2\]
that is, if we can write
\[ H'=H+c+i\]
with $c \in H^0(A)$ and $i \in I^2$. 
\end{definition}

If $H$ and $H'$ are $I^2$-equivalent, then the derivations 
$\{H,-\}$ and $\{H',-\}$ of $A/I$ coincide, so $H$ and $H'$ define exactly 
the same dynamics on the variety defined by $I$.\\

Let us now consider some elementary examples.\\ 
 
\begin{example} Let $n=1$ and consider the ideal $I$ generated by $p \in K[p,q]$.
The Hamiltonians $H= p$  and $H'=p+qp^2+1$ are $I^2$-equivalent. 
Both Hamiltonian fields are equal along the line
$$V(I)=\{ (q,p) \in K^2: p=0 \} $$
although the second one is much more complicated over the whole space:

\begin{align*}
X_H&=\d_q \\
X_{H'}&=(1+2pq)\d_q-p^2\d_p
\end{align*}
 \vskip0.3cm  \begin{figure}[h!]
  \includegraphics[width=12cm]{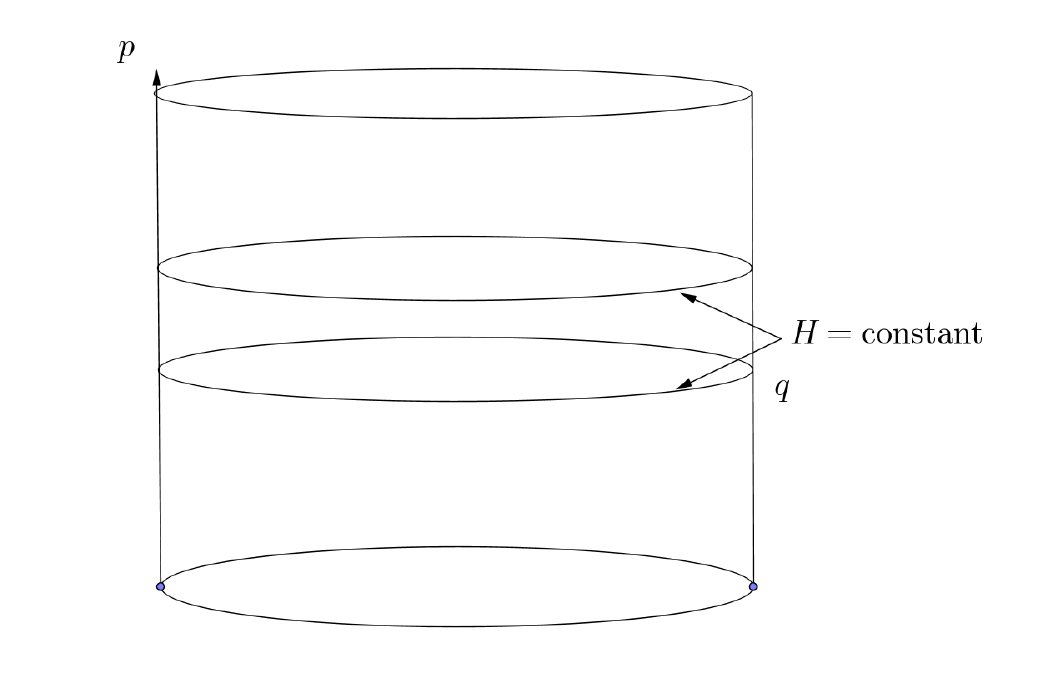}
\end{figure} \vskip0.3cm  
\end{example}

\begin{example} 
Consider again the ring $K[q,q^{-1},p]$ with the Poisson structure
$$ \{ p_j,q_k \}= q_k \dt_{jk}$$

The ideal $I$ generated by the $p_i$ is involutive and the factor ring 
\[K[q,q^{-1},p]/I \approx K[q,q^{-1}]\] 
is the coordinate ring of our algebraic torus $V(I)$.
Consider a linear Hamiltonian
$$H_0=\sum_{i=1}^n\omega_ i p_i .$$
Any Hamiltonian $H$ of the form $$H= H_0\;\; \mod (K \oplus I^2 ) $$
defines the same quasi-periodic motion on the torus $V(I)$ as $H_0$.

Via the isomorphism 
$$K[q,q^{-1},p]/I \approx K[q,q^{-1}]$$
the derivation
$$\{H_0,-\}: K[q,q^{-1},p]/I \to K[q,q^{-1},p]/I$$
is identified with the map 
$$\{H_0,-\}:K[q,q^{-1}] \to K[q,q^{-1}], q^I \mapsto (\omega,I) q^I$$

We see that the kernel reduces to $K$ if and only if the frequencies $\omega_i$'s are
$\ZM$-independent. 

From a formal viewpoint, the situation is therefore similar to that of the dynamics along a non
resonant Hamiltonian. However in our previous approach, the qualitative behaviour of the operator
varies at a resonant torus. But  when we {\em restrict the dynamics to a specific torus} we
can avoid this problem.
\end{example}

\section{Poisson cohomology and Poisson vector fields\index{Poisson cohomology}}
We already encountered the space of {\em Casimir elements\index{Casimir elements}}:
\[ H^0(A):=\{ f \in A\;|\;\{f,g\}=0\;\;\textup{\ for all\ } g \in A\}.\]
As the notation already indicates, there exist also higher cohomology groups $H^i(A)$ that generalise
the de Rham groups of a symplectic manifold. 
In particular, we will see there exists an exact sequence analogous to the exact sequence
$$0 \to \Ham(M) \to \Symp(M) \to H^1(M) \to 0 $$
of a symplectic manifold to the case of a general Poisson algebra.

Let us denote by $\Theta_A$ the space of derivations of $A$ and consider the exterior algebra 
$${\bigwedge}^\bullet \Theta_A.$$
We extend the map
$$A \to \Theta_A,\ f \mapsto \{-,f\} $$
to the exterior algebra so that it becomes a differential graded algebra. The resulting map gives a complex
$$C^\bullet(A):0 \to A \to \Theta_A \to {\bigwedge}^2 \Theta_A \to \cdots $$
called the {\em Poisson complex\index{Poisson complex}.}  If we consider the bi-derivation $\{-,-\}$ as an
element $\pi \in \L^2 \Theta_A$, then the differential of the complex is obtained by taking the Lie bracket with $\pi$.
We denote by $H^{\bullet}(A)$ the cohomology of the Poisson complex. In the symplectic case, the Poisson complex is isomorphic to the de Rham complex. From this definition we see that indeed $H^0(A)$  is the space of Casimir elements. 

A {\em Poisson-derivation}\index{Poisson derivation} is a derivation $\theta$ that 'preserves' the Poisson-bracket:
\[ \theta \{f,g\}=\{\theta(f),g\}+\{f,\theta(g)\}\]
and similarly, the space $H^1(A)$ can be interpreted as the space of such Poisson derivations, modulo those that are Hamiltonian. It fits inside the exact sequence:
$$0 \to \Ham(A) \to \textup{Poiss}(A) \to H^1(A) \to 0 $$
which provides a generalisation of the de Rham type sequence for the symplectic case.




\begin{example}
  Consider the two-dimensional $K$-algebra
\[A:=K[q,q^{-1}][[p]] \] 
of formal power series in $p$ whose coefficients are Laurent polynomials in $q$,
with the Poisson bracket
$$\{ p,q \}= q.$$
A direct computation shows that:
\begin{align*}
H^0(A)=K,\\
H^1(A)=K\d_p
\end{align*}
 \end{example}

\begin{example}
Consider the $2n$-dimensional $K$-algebra
\[A:=K[q,q^{-1}][[p]] \] 
of formal power series in $p=(p_1,\dots,p_n)$ whose coefficients are Laurent polynomials in $q=(q_1,\dots,q_n)$.
with the Poisson bracket
$$\{ p_j,q_j \}= q_j$$
and all other variables Poisson-commuting. The Poisson complex is isomorphic to the $n$-fold wedge product of the one considered above. According to K\"unneth formula:
$$ H^\bullet (A)={\bigwedge}^\bullet\oplus_{i=1}^n K[\d_{p_i}].$$
So analogous to de Rham theory, we recover the cohomology of the $n$-torus. 
\end{example}
 
\section{Pairs and homotopic stability}
We now describe a general framework in which one can formulate the
problem posed by Kolmogorov. Of course, one can not hope for very general
answers, and the answers one can give will very much depend on the particulars. 
Nevertheless, we hope that this general setup provides a useful conceptual framework.

\begin{definition}\label{D::Pair}
\begin{enumerate}[{\rm A)}]
 
 \item {\em pair} $(H,I)$ in a Poisson-algebra $A$ consists of an
element $H \in A$ and an $H$-invariant involutive ideal $I \subset A$:
\[ \{H,I\} \subset I,\;\;\;\{I,I\} \subset I \;.\]
\item The normal space to a pair $(H,I)$, denoted $N(H,I)$, is defined by
$$N(H,I)=A/\left( \{H,A\}+I^2+H^0(A) \right) .$$
\end{enumerate}
Note that the subspace that is divided out is just a linear subspace and not
an ideal. 
\end{definition}

\begin{example} The ideal $I=(p_1,p_2,\dots,p_n)$ is involutive in the Poisson-algebra $A:=K[q,q^{-1}][[p]]$, which
models the formal completion of the zero section of the cotangent bundle to the algebraic torus. For any 
$H \in K[[p]]$ we obtain a pair $(H,I)$ in $A$. We have seen that  for a non-resonant Hamiltonian
 $$H=\sum_{i=1}^n \omega_i p_i$$
the space $K[[p]]$ complements  the image of the operator $\{H,-\}$. Therefore the normal space to $(H,I)$ is the
$n$-dimensional vector space
$$N(H,I)= K [p_1] \oplus K [p_2] \oplus \dots \oplus K [p_n] $$
where we denote by $[-]$ the class of $-$ in the factor space.
\end{example}

\begin{example}
 Consider now the case of the symplectic two-dimensional algebra $A:=K[[q,p]]$ and consider the Hamiltonian
 $$H=pq $$
 and the invariant ideal $I=(H)$. The Hamiltonian derivation
 $$\{ H,- \}:A \to A $$
 maps $p^iq^j$ to $(i-j)p^iq^j$ therefore  $K[[pq]]$ complements  the image of the operator. Thus
 the normal space to $(H,I)$ is the $1$-dimensional vector space
$$N(H,I)= K [pq] $$
\end{example}

We want to describe the behaviour of pairs under deformation.
So we consider Poisson-algebras $B$ with a non-zero divisor $t$ as central element.
Hence one has an exact sequence
\[ 0 \to B \stackrel{t \cdot}{\to} B \to A \to 0\]
where $A$ is the factor ring $B/tB$. From any pair $(H',I')$ in $B$ one obtains a pair $(H,I)$ in $A$ by reduction modulo $t$, i.e. by setting $t=0$.
We will only consider pairs $(H',I')$ with the {\em flatness property}:
\[ t f \in I' \implies f \in I'\]
One then has a corresponding exact sequence
\[ 0 \to B/I' \stackrel{t \cdot}{\to}  B/I' \to A/I \to 0\] 

One wants to understand properties of the map

\[{\rm Pairs\ in\ B} \to {\rm Pairs\ in\ A}, (H',I') \mapsto (H,I) \]

The persistence property of invariant tori in KAM theory leads to the following general notion.
\begin{definition} {\em \bf  (Persistence property)} 

We say that a pair $(H,I)$ in $A$ is {\em persistent}, if 
for {\em any}  deformation $H'$ of $H$ there exists at least one pair $(H',I')$ in $B$.\\

\end{definition}

Of course, it is of great interest to find conditions that imply the persistence of a given pair.\\

In the classical situation the parameter $t$ in $A$ provides a trivial modification of the algebra $A$, like $B=A[[t]]$.  In such 
cases we have the additional property that the canonical projection $B \to A$ has a section $A \hookrightarrow B$ and we can 
consider $A$ as a {\em subring} of $B$. In this situation one can formulate the following stronger property.\\ 

\begin{definition} {\em \bf (Homotopic stability property)}

We say that a pair $(H,I)$ is {\em homotopically stable},
if for any deformation $H'$ of $H$, there exists a central Poisson morphism 
\[\p: B \to B,\;\;\p(t)=t, \p_{| A}=Id_{A}\] 
such that $\p(H')$ is $I^2$-equivalent to $H$, i.e. there exist $c \in H^0(B)$, $i \in I^2$ such that
\[\p(H')=H+c+i\] 
\end{definition}

\begin{proposition}

{If the pair $(H,I)$ is homotopically stable, then it is persistent}

\end{proposition}
\begin{proof}
Take the ideal $IB$. It has the same generators as $I \subset A$, but now we look at what they generate in $B$. 
Then take $I':=\p^{-1}(IB)$. Clearly $\{I',I'\} \subset I'$ and $\{H',I'\} \subset I'$, so $(H',I')$ is a pair in 
$B$ that lifts $(H,I)$.
\end{proof}

\section{The formal Kolmogorov theorem}

The normal space $N(H,I)$ of a pair can also be understood in terms of 
{\em first order deformations}, where we only look at the first order
in $t$ and ignare all higher order terms. Algebraically, this means that
we work in the ring $B=A[t]/(t^2)$.
Note that given element $Q \in A$ which maps to zero in $N(H,I)$, then 
there exists $h \in A$ and $i \in I^2$ such that:
$$\{H,h\}=Q+i. $$
Then 
$$\p=\Id-t\{-,h\} $$
is an Poisson automorphism of the Poisson algebra
$$B=A[t]/(t^2)$$ 
that maps the first order deformation
$$H'=H+tQ $$
to an element $I^2$-equivalent to $H$. 

This remark can be lifted to obtain the following {\em generalised formal Kolmogorov theorem.}

\begin{theorem}
 Let $(H,I)$ be a pair in $A$ and let $B=A[[t]]$. If the map
 $$H^1(A) \to N(H,I),\ v \mapsto [v(H)] $$
 is surjective, then $(H,I)$ is homotopically stable.
\end{theorem}
\begin{proof}
We will construct a Poisson automorphism $\p$ of $B=A[[t]]$ that brings
a perturbation $H+tQ$, $Q \in B$ to the normal-form
\[\p(H+tQ)=H+c+i \]
where $c \in H^0(B)=H^0(A)[[t]]$, $i \in I^2B$.
Surjectivity of the above map means that all elements $Q$ of $A$ can be
represented in the form
\[ Q=v(H)+c+i\]
where $v \in Poiss(A)$, $c \in H^0(A)$ and $i \in I^2 \subset A$.
We work modulo $(t^k)$ and use induction on $k$ and construct a sequence of
Poisson-automorphisms $\p_0=Id,\p_1,\p_2,\ldots$ with 
\[ \p_{k+1}=\p_{k} + (t^k).\]
Assume that we found such a sequence of Poisson morphisms such that
\[\p_k(H+tQ)=H +c_k+i_k+ (t^k),\;\;c_k \in H^0(B), i_k \in I^2 \]
If we look one order in $t$ further we have 
\[\p_k(H+tQ)=H +c_k+i_k+ a_kt^k+(t^{k+1})\]
By assumption, we may find 
$v_{k+1} \in {\Pt}oiss(A)$ and 8.
$$ v_{k+1}(H)= $$
Define 
$$a_{k+1}=a_k+t^k\a_{k} \in I^2 +H^0(A) $$ we get that
$$e^{-t^kv_k}\p_k(H_0)=H_0+a_{k+1}\ \mod (t^{k+1}). $$
This proves the theorem.
\end{proof}

\begin{example}
Let us consider the case $n=1$, $A=K[[q,p]]$ with
$$H=pq $$
and let $I=(H)$ be the ideal generated by $H$. The coordinate lines are invariant under the Hamiltonian flow. We know that
$$N(H,I)=K [pq] $$
We have
$$H^0(A)=K,\ H^1(A)=0.  $$
The map
$$H^1(A) \to N(H,I),\ v \mapsto [v(H)] $$
maps a zero-dimensional space to a one-dimensional space, so it cannot be 
surjective. There is indeed a non-trivial deformation of the pair $(H,I)$ 
namely $$H'=(1+t)H,\ I'=pq. $$
The difference with the torus case is that here there is a single special Lagrangian variety consisting of two lines. Along this variety there is a well 
defined frequency equal to $(1+t)$.
\end{example}

\begin{example}
We make the following variation of the previous example:
we let
\[ A=K[[\l,q,p]],\ H=(1+\l)pq, I=pq \]
where $\lambda$ is a central element and let $I=(H)$ be the 
ideal generated by $H$.
We have
$$N(H,I)=K[[\l]] [pq] $$
but now
$$H^0(A)=K[[\l]],\ H^1(A)=K[[\l]].[\d_\l]$$
The map
$$H^1(A) \to N(H,I),\ v \mapsto [v(H)] $$
is now an isomorphism. The pair $(H,I)$ is now homotopically stable.
\end{example}

Like in the example, the map involved in the theorem is in general a morphism of finite type modules over the ring of Casimir operators. Let us now work out the torus case in details. 

\begin{proposition} Let $I$ be the ideal generated by $p_1,\dots,p_n$ in $A:=K[q,q^{-1}][[p]]$.
Let $H=\sum_{i=1}^n \omega_i p_i+\sum_{i,j} \a_{ij} p_ip_j+\dots \in K[[p]]$ be
a power series in the $p$ variables such that
\begin{enumerate}
\item[{\rm (R)}] the vector $\omega=(\omega_1,\dots,\omega_n)$ is non resonant,
\item[{\rm (N)}] the $n \times n$ matrix $(\a_{ij})$ is invertible
\end{enumerate}
then these conditions are respectively equivalent to:
\begin{enumerate}
\item[{\rm (R)}] $N(H,I)$ is an $n$-dimensional vector space generated by the classes of $p_1,\dots,p_n$,
\item[{\rm (N)}] the natural map $H^1(A) \to N(H,I)$ is surjective.
\end{enumerate}
\end{proposition}
\begin{proof}
The first assertion is due to the fact that the absence of resonances implies that $\{H,-\}$ is diagonal in the monomial basis and that its image consists of functions with zero mean value. For the second assertion note that
\begin{align*}
H^0(A)=K[1] \\
H^1(A)=\bigoplus_{i=1}^n K[\d_{p_i}]
\end{align*}

and that the map 
$$K^{n} \to N(H,I),\ (c_1,\dots,c_n) \mapsto [\sum_{i=1}^n c_i \d_{p_i}H_0] $$
is an isomorphism if and only if the matrix $(a_{ij})$ is invertible. This proves the theorem.
\end{proof}

In particular the pair $(H,I)$  of the proposition is homotopically stable 
and we see that the formal version of Kolmogorov's theorem is a special 
case of our general theorem.

We will see that in the analytic case, that is, Kolmogorov's invariant torus theorem, persistence is also shown be showing homotopic stability.

\section{Bibliographical notes}

The Poisson complex was introduced in:\\
{\sc A. Lichnerowicz.} {\em Les vari\'et\'es de Poisson et leurs alg\`ebres de Lie associ\'ees. J. Differential Geom. 12 (1977), no. 2, 253--300.}\\

\chapter{The Lie iteration} \index{Lie iteration}
In the previous chapter, we stated Kolmogorov's invariant torus theorem and its generalisations in terms of infinite 
dimensional group actions. In this sense, it is a particular case of a general theorem on group actions. It is fortunate that all basic geometrical ideas already appear in the finite  
dimensional situation, namely for Lie group actions and in fact even for linear group 
actions. So the natural action of a matrix Lie group $G \subset GL(V)$ on a finite dimensional 
vector space $V$ is a natural starting point for our investigation.

 \section{Local transversals to a group action}
When a Lie group $G$ acts smoothly on a manifold $M$, we are given a map
\[ \sigma: G \times M \to M,\;\;(g,x) \mapsto g \cdot x\, .\]
If $a \in M$ is a point, we obtain the so-called {\em orbit map}\index{orbit map} 
\[ \sigma_a: G \to M,\;\;g \mapsto g \cdot a,\]
whose image is the $G$-orbit through $a$.
The derivative this map at the identity element $e \in G$ gives a linear map
\[ \rho:=d \sigma_a:\alg \to T_a M,\;\;\;\xi \mapsto \xi(a):=d\sigma_a(\xi) \]
where $\alg=T_eG$ denotes the Lie-algebra of $G$. We call this map the {\em infinitesimal action}\index{infinitesimal action} (at $a$). The image
\[ \alg \cdot a:=d\sigma_a(\alg)\]
of this map is the {\em tangent space at $a$ to the orbit through $a$}
and might be called the {\em $\alg$-orbit}.\index{$\alg$-orbit}\\
\newpage
\begin{figure}[htb!]
\includegraphics[width=11cm]{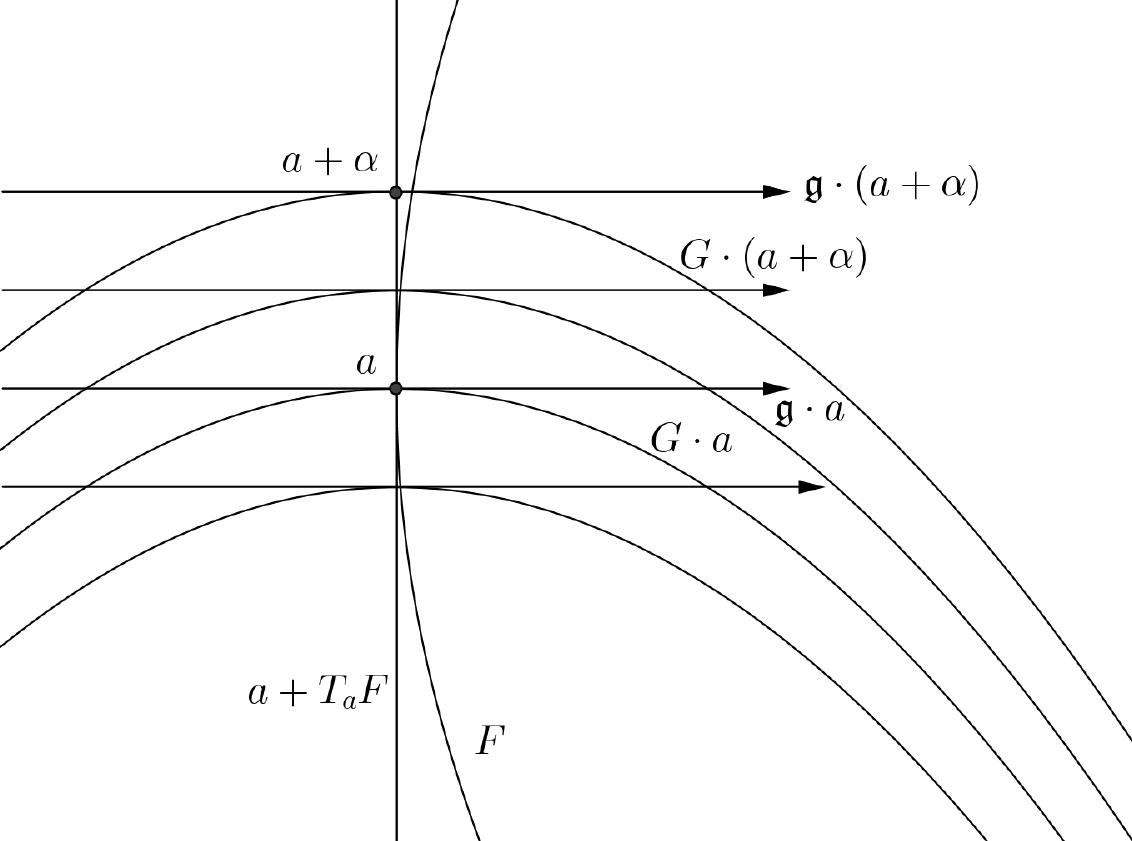}
\end{figure} 

One may ask the following general question:\\

{\em If $F \subset M$ is a submanifold such that $T_aF$ is a transversal to the $\alg$-orbit of $a$, 
is it true that $F$ is a transversal to the $G$-action? \index{transversal}

In other words: does 
$$T_aF+\alg\cdot a=T_aM$$ implies that the map
\[G \times F \to M,\;\;(g,b) \mapsto g \cdot b \]
is locally surjective around $a$?}\\

It follows from the implicit function theorem that the answer is positive. In the next chapter 
we will describe a general iteration schema that provides a constructive solution. 
KAM theory is concerned with the analogous situation in infinite dimensions where no implicit function theorem is available, but the same iteration scheme still can be used.
For example, as we will see, Kolmogorov's invariant torus theorem can be proven using this iteration, if interpreted in an appropriate sense. For that we will introduce the notion of {\em Kolmogorov spaces}\index{Kolmogorov space}, 
which will be developed in the next chapter. The resulting general theory then has many applications.\\
\newpage
We start with some basic examples of Lie group actions.

\begin{example}Consider the action of the Lie group $G=\RM$ on $\RM^2$ by translation along the vector
$u=(1,0)$:
$$\RM \times \RM^n \to \RM^n,\ (t,x,y) \mapsto (x+t,y). $$ 
The orbits are horizontal lines. Any curve which does not have an horizontal tangent at a given point is a transversal to the group action at that point.\\ 

 \vskip0.3cm  \begin{figure}[ht]
\includegraphics[width=11cm]{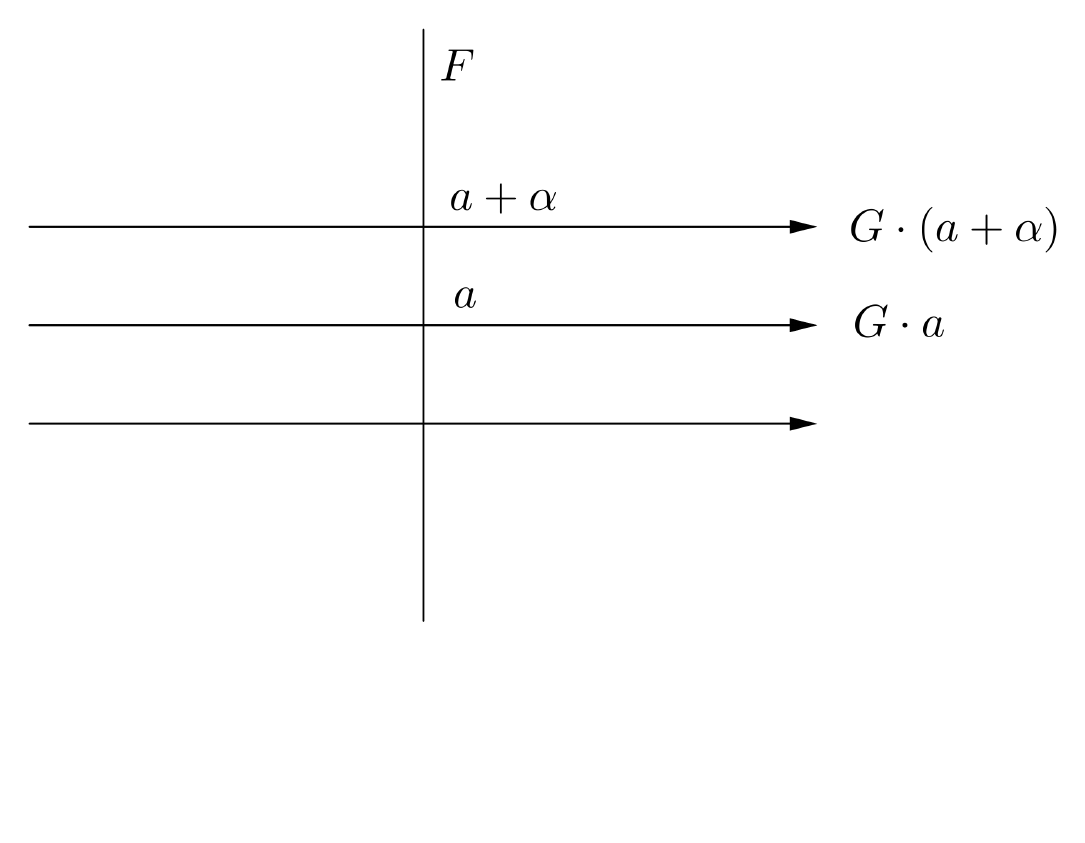}
\end{figure} \vskip0.3cm  
\end{example}

\begin{example}Consider the action of the Lie group $G=\RM$ acting on $\RM^2$ by translation along a parabola:
$$\RM \times \RM^2 \to \RM^2,\ (t,x,y) \mapsto (x+t,y-2xt-t^2). $$ 
The orbit through $(x_0,y_0)$ is the parabola with equation
\[y+x^2=y_0+x_0^2\] 
Any curve which is not tangent at a point to the corresponding parabola is a transversal at the point.

The infinitesimal action at $a=(x_0,y_0)$ is the map
$$\RM \to \RM^2,\ t \mapsto (1,-2x_0) t. $$
 
The $\alg$-orbit is the tangent to the parabola at $(x_0,y_0)$, so any curve which is not tangent at 
a point to the corresponding parabola is a transversal at the point.
In particular, if $x_0=0$ the $\alg$-orbit is an horizontal line.
The situation is summarised by the first picture of this chapter.
\end{example}

\begin{example}Consider the natural action of the Lie group $G=Gl(n,\RM)$ on $\RM^n$:
$$GL(n,\RM) \times \RM^n \to \RM^n,\ (A,x) \mapsto Ax. $$ 
The action is homogeneous\index{homogeneous action} at any point $x \neq 0$ and the 
orbit is $\RM^n \setminus \{ 0 \}$. Therefore there are two orbits: the origin and its complement.
\end{example}
\begin{example}Consider the natural action of the Lie group $G=SO(n,\RM)$ on $\RM^n$:
$$SO(n,\RM) \times \RM^n \to \RM^n,\ (A,x) \mapsto Ax. $$  
The orbit of $a \neq 0$ is the $n-1$ dimensional sphere $S^{n-1}$ of radius $\| a \|$ centred at the origin. Therefore the action is no longer homogeneous.
A transversal to the orbit at $a$ is given by any manifold transversal to the sphere at $a$, for instance the straight line $\RM a$.

The Lie algebra $\alg$ consists of antisymmetric matrices and its orbit at $a$ is the tangent space to our sphere at $a$.\\

 \vskip0.3cm  \begin{figure}[htb!]
\includegraphics[width=12cm]{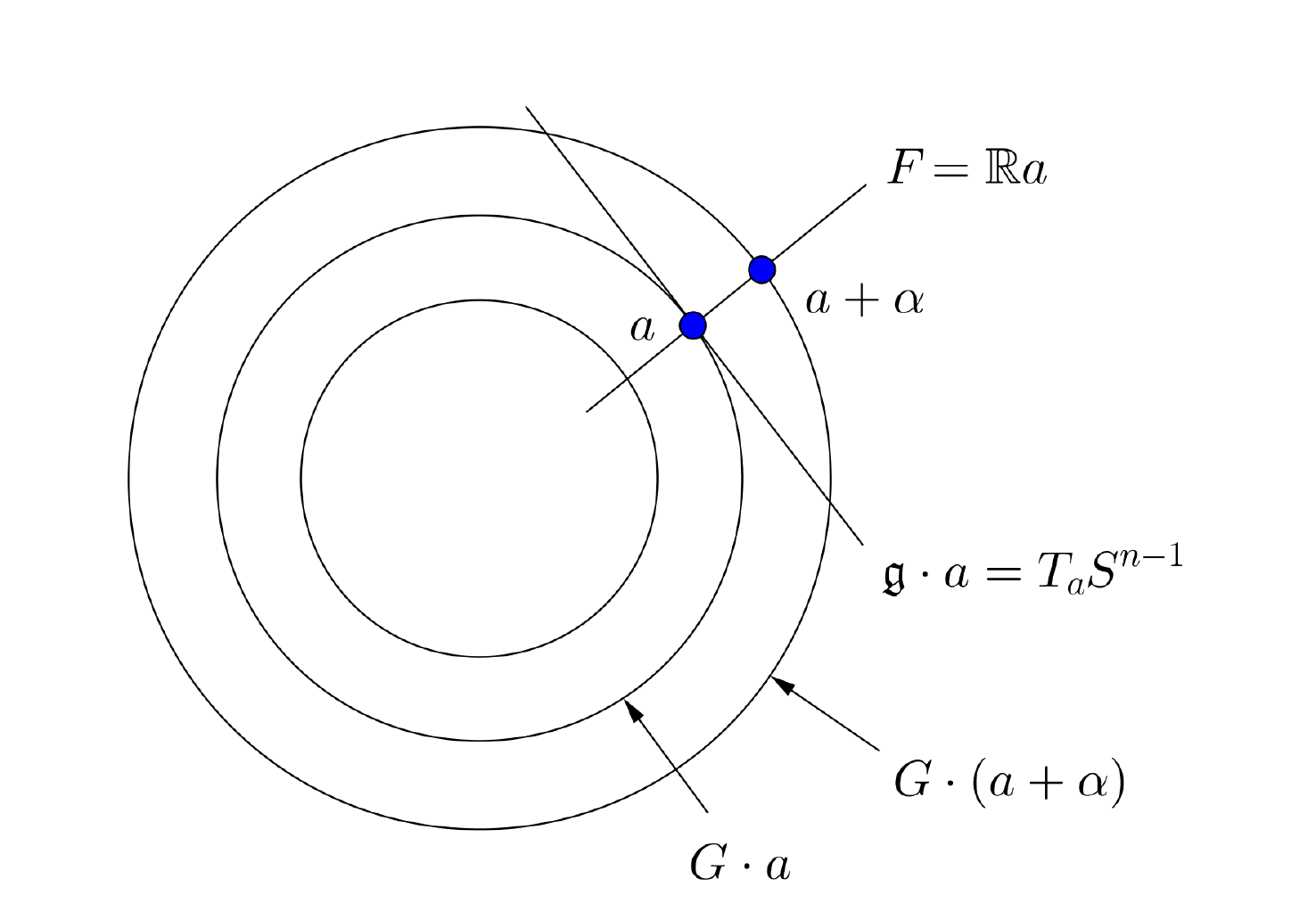}
\end{figure} \vskip0.3cm  
\newpage

\end{example}

\section{Adjoint orbits for the linear group}
If $G$ is a Lie group, it acts on itself via conjugation
\[  G \times G \to G, (g,h) \mapsto g h g^{-1}\]
As this action preserves the unit element $e \in G$, there is an induced linear action
of $G$ on the tangent space $\alg=T_eG$ 
\[ G \times \alg \to \alg\;.\]
This is commonly called the {\em adjoint action}\index{adjoint action} of a group 
on its Lie-algebra.\\

Let us take a closer look at it for the  group $G=GL(n,\RM)$. Its Lie-algebra $\alg$ 
can be identified with the space $M(n,\RM)$ of $n \times n$ matrices and the
adjoint action is given by conjugation
\[ GL(n,\RM) \times M(n,\RM) \to M(n,\RM), (P,A) \to P A P^{-1}\] 
To determine the  corresponding infinitesimal action,  we write down a one-parameter
family of matrices near the identity, with a given $B \in M(n,\RM)$ as tangent
vector:
$$P=\Id+t B+o(t) $$
The inverse then is
 $$P^{-1}=\Id-t B+o(t) $$
 and therefore
 $$ PAP^{-1}=A+t[B,A]+o(t) .$$
 Therefore the infinitesimal action at $A \in M(n,\RM)$ is the map
 $$ M(n,\RM) \to M(n,\RM), B \mapsto [B,A].$$

In order to describe a transversal to the $\alg$-orbit, it is convenient to use the Euclidean scalar product $ \< -,- \> $ on $\alg=M(n,\RM)$ given by the formula
\[ \< A,B \> :=Tr(A B^T) .\]
Here $-^T$ denotes the transposed matrix.

\begin{lemma} The commutant\index{commutant} 
$$C(A):=\{ B \in M(n,\RM) : [B,A]=0 \} $$
of $A$ is mapped by $B \mapsto B^T$ to the orthogonal space of the $\alg$-orbit:
\[ C(A)^T = (\alg\cdot A)^{\perp} .\]
\end{lemma}
\begin{proof} The $\alg$-orbit of $A$ consist of all elements of the form $[A,X]$, where $X \in M(n,\RM)$. Now note the identity
\[
\begin{array}{rcl}
\<  [A,X], B^T \> &=&Tr(AXB)-Tr(XAB)\\
&=&Tr(BAX)-Tr(ABX)= \< [B,A],X^T \> 
\end{array}
\]
So we see that $B^T$ belongs to the orthogonal space of the $\alg$-orbit of $A$
precisely when $B$ commutes with $A$.
\end{proof}

 We compute a local transversal at a matrix $A$ in two extreme cases:\\

{\bf \em {\rm 1)  $A$ is diagonal with distinct eigenvalues}}\\
In this case, the characteristic polynomial has distinct roots.  By transversality, the matrix $A$ admits a neighbourhood in which all matrices share the same property and are therefore all diagonalisable. In terms of group actions: the space of diagonal matrices is a local transversal to the adjoint action at $A$, which indeed is the transpose
of the commutant.
As an example, if
 $$A=\begin{pmatrix} 1 & 0 \\ 0 & 2 \end{pmatrix} $$
 then the space of matrices of the form
 $$D_\l=\begin{pmatrix} 1+\l_1 & 0 \\ 0 & 2+\l_2 \end{pmatrix} $$
is a local transversal. In general it means that $A$ has a neighbourhood $U$ such that
for all $B \in U$ there exits $\l_1,\l_2,\ldots, \l_n$ and $P \in GL(n,\RM)$ with the
property that $PBP^{-1}=D_\l $.

{\bf \em {\rm 2)  $A$ is maximally nilpotent Jordan block.}}
So we let 
\[ A=\begin{pmatrix} 0& 1 & \dots & 0 &0 \\
0& 0 &1 & \dots & 0 \\
 \dots&\dots & \dots &\dots&\dots  \\
 0 &0 &\dots &0 &1 \\
 0 & 0 &\dots & 0 &0 \\  \end{pmatrix} \]

If a matrix $B$ commutes with $A$, $[B,A]=0$, then $B$ preserves the filtration defined by the kernels of the powers of $A$:
$$\{ 0 \}\subset \Ker A \subset \Ker A^2 \subset \dots \Ker A^{n-1} \subset \Ker A^n=\RM^n. $$
and  thus $B$ is seen to be of the form:
$$B=\begin{pmatrix} \l_1 & \l_2 & \l_3 & \dots & \l_{n-1}  & \l_n \\
                                     0 & \l_1 & \l_2 &\dots  &\l_{n-2} & \l_{n-1} \\
                                \dots & \dots&\dots  & \dots  &\dots  & \dots \\
                            \dots & \dots&\dots  & \dots  &\dots  & \dots \\
                              0 &0 &0     & \dots     &\l_1 & \l_2 \\
                                0 &0 &0    & \dots   & 0  & \l_1 \end{pmatrix}=\sum_{i=1}^n \l_i A^{i-1}$$
The transpose of this space is a  transversal to the orbit:
 $$F=\left\{  \sum_{i=1}^n \l_i\, \left(A^{i-1}\right)^\mathrm{T}: \l_1,\dots,\l_n \in \CM \right\}.$$
 For instance, we start with
 $$A=\begin{pmatrix} 0 & 1  \\
 0 &0   \end{pmatrix} $$
 The commutant consists of matrices of the form
 $$\begin{pmatrix} \l_1 & \l_2  \\
 0 &\l_1  \end{pmatrix} $$
 and the orthogonal complement to the orbit consists of matrices of the form
 $$\begin{pmatrix} \l_1 & 0   \\
 \l_2&\l_1  \end{pmatrix} $$
This means that there is a neighbourhood of $A$ such that any matrix can be taken back to the normal form
 $$\begin{pmatrix} \l_1 & 1   \\
 \l_2&\l_1  \end{pmatrix} $$

\section{The Lie iteration in the homogeneous case}
The classical Heron iteration for finding surds and Newton's more general root-finding iteration are quadratically convergent.   
The basic technical tool we will introduce now is an iteration scheme similar (but different) to the Newton iteration in the context of a Lie group action. 
To distinguish it from the Newton iteration, we call it the {\em Lie iteration}.\index{Lie iteration} Its efficiency relies on the following two facts:
\begin{enumerate}[{\rm (1)}]
 \item The Lie group and the Lie algebra agree at first order.
 \item The tangent space to the Lie group is isomorphic to the Lie algebra.
\end{enumerate}
As we shall see, (1) implies that our iteration is quadratic, like for the Newton method and (2) implies that we only need
to linearise the map at the identity, unlike the Newton method which requires global construction of inverses.

Let us first assume that $V$ is infinitesimally $G$-homogeneous at a point $a \in V$. This means that the infinitesimal action of the Lie algebra at $a$
$$\rho:=d\sigma_a:\alg \to V,\ \xi \mapsto \xi(a)$$
is surjective. In this case the implicit function theorem tells us that $V$ is locally $G$-homogeneous: for any $b \in V$ small enough, there exists $g \in G$ such that
$$g(a)=a+b. $$
The following iteration produces a sequence of elements in $G$ which converges rapidly to $g$. It is determined by the choice of a linear map 
 \[j:V  \to \alg\]
that is a right inverse to the infinitesimal action, i.e
$$\rho \circ  j=\Id.$$
and the local isomorphism determined by the exponential map
\[ \textup{exp}: \alg \to G,\;\;\;\xi \mapsto e^{\xi}\,. \]

We start our iteration with:
$$b_0:=b $$
and define 
$$\xi_0:=j(b_0). $$
The element
$$e^{\xi_0} a $$
is close to $a+b$ or equivalently
$$e^{-\xi_0} (a+b) $$
is close to $a$. The error of this approximation is
$$b_1= e^{-\xi_0}(a+b)-a $$
We set $\xi_1:= j(b_1)$ and repeat the process. In this way we get sequences $(b_n),(\xi_n)$ which define the Lie iteration
scheme\index{Lie iteration} in the homogeneous case:\\
\begin{center}
\framebox{ 
$ \begin{matrix} \xi_n&=&j(b_n) \\ 
b_{n+1}&=&e^{-\xi_n}( a+b_n)-a \end{matrix} $}
\end{center}
Note that
$$a+b_{n+1}=e^{-\xi_n} (a+b_n) =e^{-\xi_n}e^{-\xi_{n-1}}(a+b_{n-1}) =\dots=\prod_{i \geq 0}^{n}e^{-\xi_i}(a+b_0).$$

and
\[\xi_n(a)=d\rho_a (\xi_n)=d\rho_a(j(b_n))=b_n\]

  \vskip0.3cm  \begin{figure}[ht]
\includegraphics[width=13cm]{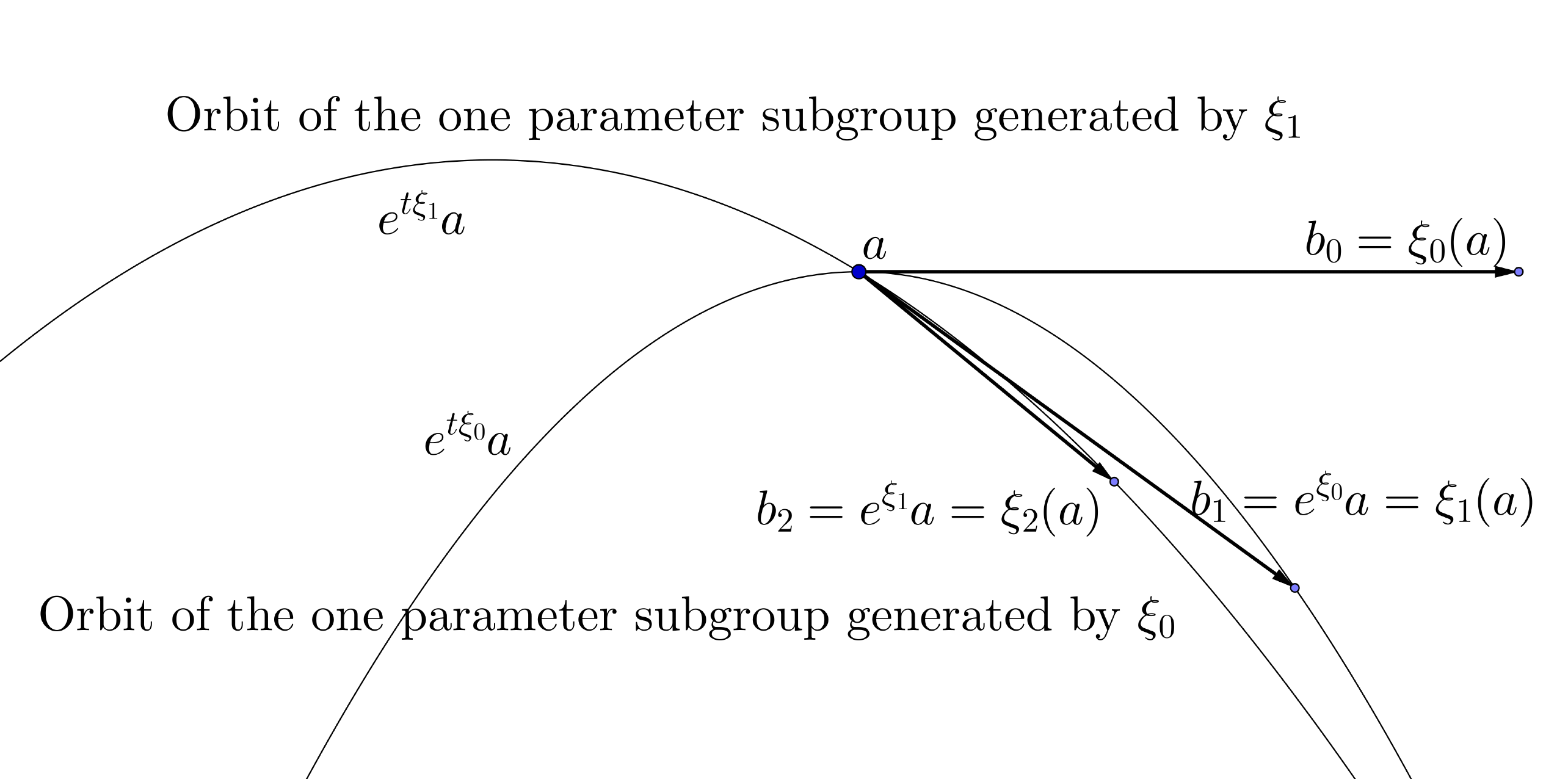}
\end{figure} \vskip0.3cm  

\section{Quadratic convergence of the Lie iteration}

We want to control the rate of convergence of this iteration. 
To do this we use the operator norm:\index{operator norm}
$$\| \xi \|:=\sup_{x \in \RM^n} \frac{\| \xi(x) \|}{\| x \|} .$$
and observe that
 $$\log \| \prod_{i \geq 0}e^{\xi_i} \| \leq \log  \prod_{i \geq 0}e^{\| \xi_i \|}=\sum_{i \geq 0}\| \xi_i \|. $$

 In particular:
 \begin{lemma} Let $(\xi_i)$ be a sequence of operators of a Banach algebra. The infinite product   $\prod_{i \geq 0}e^{\xi_i}$
 is converges in the operator norm provided that the sequence $(\| \xi_i \|)$ is summable.
 \end{lemma}

 The generalisation of this lemma to a more general situation will play a key role in KAM theory. We want to apply the lemma, so we must evaluate
 the rate of convergence of the sequence  $(\xi_n)$. As
 \begin{align*}
 b_n&=\xi_n(a),\\
b_{n+1}&=e^{-\xi_n} (a+b_n)-a,
\end{align*}
we obtain:
$$b_{n+1}=e^{-\xi_n} (a+\xi_n( a))-a.$$
or equivalently:
$$b_{n+1}=(e^{-\xi_n}(\Id+\xi_n)-\Id)( a) .$$
Thus the sequence $(\xi_n)$ is obtained by iteration of the function:
$$F=j \circ f,\ f(x)=e^{-x}(1+x)-1.$$ 
Note that $f$ has a critical point at the origin:
$$f(x)=(1-x)(1+x)-1+o(x^2)=-x^2+o(x^2). $$
In particular, by Taylor's formula, there exists  a neighbourhood of the origin and a constant $C>0$ such that:
$$\| f(\xi) \| \leq C\| \xi \|^2 .$$

We are in the classical situation of a quadratic iteration:
If 
\[\| f(x) \| \leq C \| x \|^2 \]
then 
 \[\| f(f(x)) \| \leq C \| f(x) \|^2 \leq C\cdot C^2 \|x\|^4\]
\[ \| f(f(f(x))) \| \leq C \| f(f(x)) \|^2 \leq C\cdot C^2 \cdot C^4 \|x\|^8\]
so that in general
\[ \| f^{(n)}(x)\| \leq C^{2^{n}-1} \|x\|^{2^n}=\frac{1}{C} (C\|x\|)^{2^n}\]
and where $f^{(n)}$ denotes the $n$-th iterate of $f$.

So we see: 
\begin{theorem} If $\|f(x)\| \leq C \|x\|^2$ and $\|x_0\| < 1/C$, then the sequence of iterates 
\[x_n:=f^{(n)}(x)\]
converges rapidly to zero:\index{rapid convergence}
\[ \|x_n\| \le \frac{1}{C} \rho^{2^n},\;\;\rho:=C \|x_0\| < 1 .\]
\end{theorem}

\begin{theorem}\label{T::groupesfinitedim} Let $V$ be a finite dimensional vector space, 
$G \subset GL(V)$ a group acting linearly on $V$ and  $\alg$ its Lie-algebra,
$a \in V$ a point.
Assume that the map 
\[\rho=d\sigma_a:\alg \to V, \xi \mapsto \xi(a)\] 
admits an inverse $j$.

Then, for any $b \in V$ small enough, the Lie-iteration produces a rapidly convergent sequence $(\xi_n)$ of elements in $\alg$ such that
\[ a =\prod_{i \ge 0} e^{-\xi_i} (a+b) .\]
\end{theorem}

The classical Heron-Newton iteration is also quadratic, but it requires the computation of a inverse to the differential at every new step. In the context of a group action, we need only a {\em single} inverse $j$ to the infinitesimal action. Also note that the quadraticity of the iteration scheme stems from the fact that the Lie-group $G$ and the Lie-algebra $\alg$ 'agree up to first order'.
But it is also important to be aware of the fact that the Lie iteration does neither require the group $G$, nor the fact that $\alg$ is its full Lie-algebra. This is especially important for the applications of the
iteration in infinite dimensional situations, as it is not necessary to construct the complete group as an infinite dimensional
manifold with tangent space $\alg$.  

 \section{The Lie iteration in the general case}
  
We adapt our previous iteration to non-transitive actions. So let $G  \subset GL(V)$ be a Lie group acting naturally on the vector space $V$. Let $F$ be an $\alg$-transversal at some point $a \in V$. By the implicit function theorem, it is also a $G$-transversal. 

Given $b \in V$, we want to find $g \in G$ and $\a \in F$ such that
$$g(a+\a)=a+b$$

To do so, we extend the infinitesimal action
\[ \rho: \alg \to V,\;\;\;\xi \mapsto \xi(a)\] 
to a surjective map
\[\rho_a: F \times \alg \to V, \;\;\;(\a,\xi) \mapsto \xi(a)+\a \]

Let
 $$j_a: V \to F \times \alg$$
be a right inverse to $\rho_a$. 
We also put
\[j: (a+F) \times V \to F \times \alg,\;\;\;(a+\alpha,b) \mapsto j_{a+\alpha}(b) \]

Our problem is now the following: given $b \in V$, we want to find $g \in G$ and $\a \in F$ such that
$$g(a+\a)=a+b$$
To do so, we will modify the original iteration as follows.
First we set $b_0:=b$, $a_0:=a$ and then consider $j(a_0,b_0)=(\a_0,\xi_0)$. Then $\a_0 \in F$ and $\xi_0 \in 
\alg$ are elements such that
\[b_0=\xi_0 (a_0)+\a_0\]
Now the element $e^{\xi_0}(a_0+\a_0)$ will be close to $a_0+b_0$ and we put 

\begin{align*}
a_{1}:=&a_0+\a_0,\\
b_{1}:=&e^{-\xi_0}(a_0+b_0)-a_{1},\\
\end{align*}

 \vskip0.3cm  \begin{figure}[htb!]
  \centering
  \includegraphics[width=13cm]{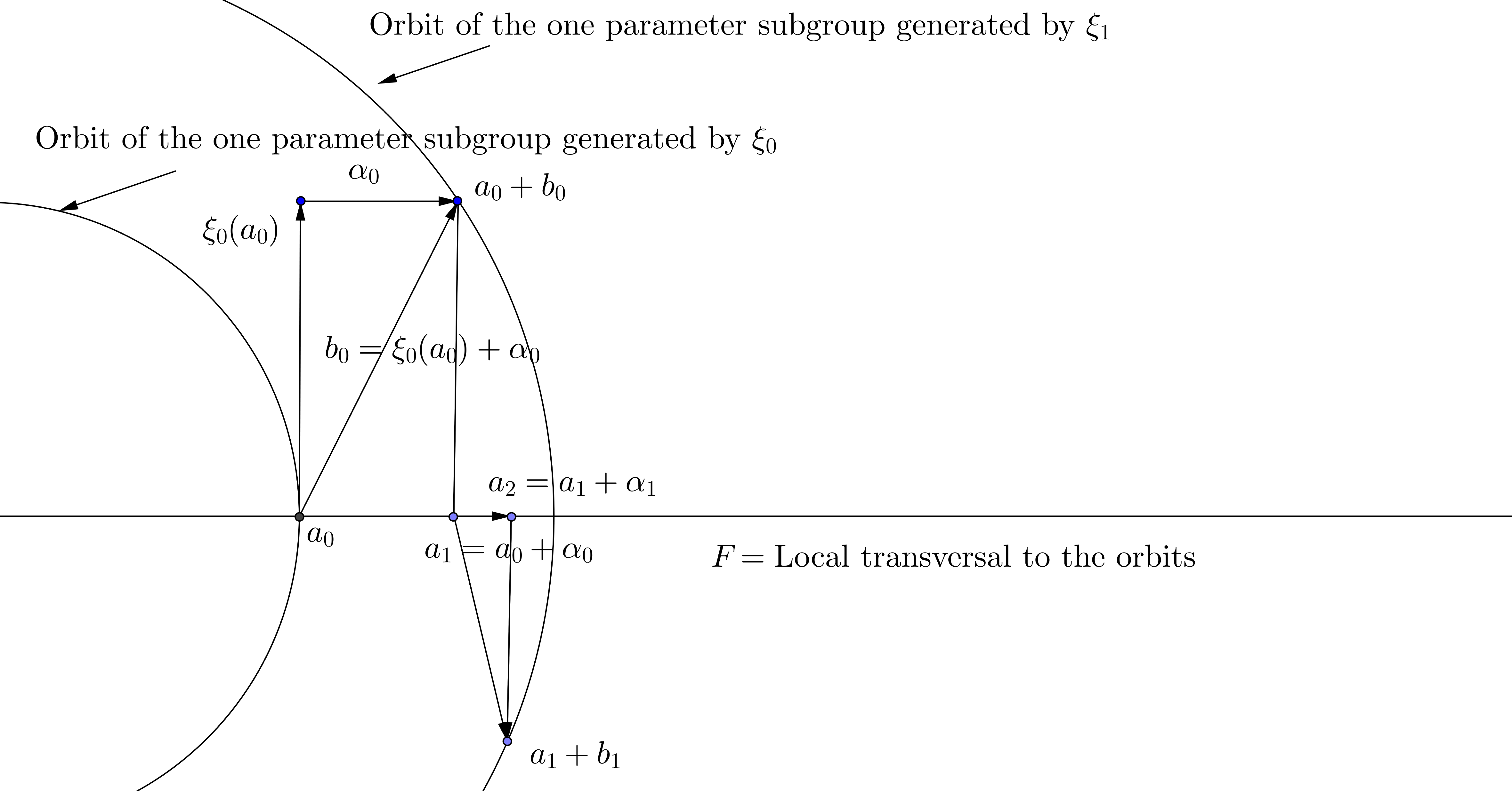}
\end{figure} \vskip0.3cm  

We now can iterate this procedure and obtain sequences $a_n, b_n, \a_n, \xi_n$ defining the Lie iteration\index{Lie iteration} in the general case:

\begin{center}
\framebox{ 
$ \begin{matrix}
a_{n+1}:=&a_n+\a_n,\\
b_{n+1}:=&e^{-\xi_n}(a_n+b_n)-a_{n+1},\\
(\a_{n+1},\xi_{n+1}):=&j(a_{n+1}, b_{n+1})\\ 
\end{matrix}$}
\end{center}

The last equation expresses the fact that at each step we have the decomposition
\[b_{n+1}=\xi_{n+1}(a_{n+1})+ \a_{n+1} \;.\]

Note that, just as in the iteration for the homogeneous case, one has
$$a_{n+1}+b_{n+1}=e^{-\xi_n}(a_n+b_n)=e^{-\xi_{n}}e^{-\xi_{n-1}}(a_{n-1}+b_{n-1})=\dots=\prod_{i \geq 0}^{n}e^{-\xi_i}(a+b)$$

Let us now rewrite this as the iteration of a mapping. If we substitute the relation
\[b_n=\xi_{n}( a_n)+ \a_n \;.\]
in 
$$b_{n+1}=e^{-\xi_n} (a_n+b_n )-a_{n+1},$$
we obtain:
$$b_{n+1}=e^{-\xi_n} (a_n+\xi_n(a_n)+\a_n )-a_{n}-\a_n=\left( e^{-\xi_n} (a_n+\xi_n(a_n))-a_n \right)+\left( e^{-\xi_n}\a_n )-\a_n\right).$$
which can be written as
$$b_{n+1}=A(\xi_n,a_n)+B(\xi_n,\a_n)$$
where
$$A(\xi,a)= (e^{-\xi}(\Id+\xi)-\Id) (a),\ B(\xi,\a)=(e^{-\xi}-\Id) (\a)\;. $$
Note that, for fixed $a$, both $A$ and $B$ have a quadratic singularity at $\xi=0,\ \a=0$.
From this we find
\[(\a_{n+1},\xi_{n+1})=j(a_n+\a_n,A(\xi_n,a_n)+B(\xi_n,\a_n))\]
The right hand side still involves $a_n$, so we write:
\[(a_{n+1},(\a_{n+1},\xi_{n+1}))=(a_n+\a_n,j(a_n+\a_n,A(\xi_n,a_n)+B(\xi_n,\a_n)))\]

The formula is a bit involved, but the only important point is that it is  of the form
$$(x_{n+1},y_{n+1})=(x_n+Ly_n,f(x_n,y_n)) $$
where $L$ is a linear map and $f(x,-)$ is quadratic. 
In our situation, we have 
$x_n=a_n, y_n=(\a_n,\xi_n), Ly_n=\a_n$, 
and, as we observed:
\[y \mapsto f(x,y)=j(x+Ly, A(x,y)+B(y))\] 
is quadratic.

We may adapt the fixed point theorem to this situation:

\begin{theorem} Let $E,F$ be Banach space and let
$$F: B_E \times B_F \to E \times F,\ (x,y) \mapsto (x+Ly,f(x,y))   $$
be such that there exists a constant $C $ with
$$\| f(x,y)\| \leq C\| y\| ^2  $$
for any $x \in B_E$. Let $y_0 \in B_F$ be such that
$$C\|y_0\| \leq 1 \ {\rm\ and\ }\frac{\| L \|}{C} \sum_{n \geq 0} (C\|y_0\|)^{2^n}<1 $$
Then the sequence
 $$(x_n,y_n)=F^n(x,y),\ x_0=0$$
 converges to a limit $(l,0)$  and 
$$\| y_n\| \leq \rho^{2^n},\ \| x_n\| \leq    \frac{\| L \|}{C} \sum_{i=1}^{n-1} \rho^{2^i}$$
with $\rho=C\|y_0\|$.
\end{theorem}
\begin{proof}
The proof is a small variation to the non-parametric case.
As
\[\| f(x,y) \| \leq C \| y \|^2 \]
we get that:
 \[ \| y_n\|=\| f^{(n)}(x_0,y_0)\| \leq C^{2^{n}-1} \|y_0\|^{2^n}=\frac{1}{C} (C\|y_0\|)^{2^n}\]
 and 
 \begin{align*}
 \| x_n\| &\leq \| L \|\, \| y_{n-1} \| +\| x_{n-1}\| \\ 
   & \leq   \| L \|\, \| y_{n-1} \| +\| L \|\, \| y_{n-2} \|+ \| x_{n-2}\| \\
  &\leq \dots \\
  &\leq   \| L \| \sum_{i=1}^{n-1} \| y_{i} \|.
  \end{align*}
\end{proof}

The main point of difference between the homogeneous and the parametric iteration is that 
in the first case we need only a single right inverse to the infinitesimal action at $a$. 
In the parametric situation one needs such a right inverse at the complete sequence of 
points $a=a_0,a_1,\ldots$ lying in the transversal slice. In practise this transversal 
consists of simple modification of our initial problem and is well under control.

\section{Bibliographical Notes}

In his lectures on singularity theory, Arnold used Lie groups actions as a 
preliminary step to the study of hypersurface singularities and their 
deformations. The study of $GL(n,\RM)$ orbits was studied in:\\

\noindent {\sc Arnold V.I.,} {\em On matrices depending on parameters,} Russian Mathematical Surveys {\bf 26}(2), 29-43, (1971).\\ 

Inspired by Kolmogorov's invariant tori paper, the Lie iteration was formulated by Moser in the absolute case in:\\

\noindent  {\sc J. Moser}, {\em A rapidly convergent iteration method and non-linear partial differential equations II}, {Ann. Scuola Norm Sup. Pisa - Classe di Scienze S\'er. 3},{\bf 20}(3), {499-535}, (1966).\\

see also:\\

\noindent {\sc  J. F\'ejoz, J. and M. Garay.}, {\em Un th\'eor\`eme sur les actions de groupes de dimension infinie},
{Comptes Rendus \`a l'Acad\'emie des Sciences}, 348, 427-430, (2010).\\  
{\sc M. Garay}, {\em An Abstract KAM Theorem}, Moscow Math. Journal, Volume 14, Number 4, p.745-772, (2014).\\

Moser does not treat normal forms as an application of his implicit function theorems (now called Nash-Moser theorems). On the contrary, Moser noticed that the requirement of these theorems are in general not fulfilled. When possible application of  implicit function theorems in the presence of small denominators remains a non trivial problem. For instance, following an idea of Herman this has been done by Bost for the case of the Kolmogorov theorem in:\\

\noindent {\sc J.-B. Bost},	{\em Tores invariants des syst\`emes dynamiques hamiltoniens},  {S\'em. Bourbaki 639}, {Ast\'erisque},   {133-134}, p. {113- 157} (1986).
